\documentclass[11pt,reqno]{amsproc}

\usepackage{amsmath,amsfonts,amssymb,amsthm}
\usepackage[abbrev,lite,nobysame]{amsrefs}
\usepackage{mathrsfs}
\usepackage{graphics,graphicx}
\usepackage[dvipsnames]{xcolor}
\usepackage{times}
\usepackage{bbm}
\usepackage[margin=1in]{geometry}

\usepackage{enumerate}
\usepackage{enumitem}
\usepackage{mathtools}

\usepackage{longfbox}

\usepackage{tikz}
\usetikzlibrary{patterns,calc,decorations.markings}

\usepackage[]{hyperref}
\hypersetup{
    colorlinks=true,       
    linkcolor=BrickRed,          
    citecolor=black,        
    filecolor=magenta,      
    urlcolor=cyan           
}

\usepackage{tcolorbox}

\usepackage{nicematrix}

\definecolor{blush}{rgb}{0.87, 0.36, 0.51}

 \usepackage[toc]{appendix}

\newcommand\e{{\rm e}}
\newcommand\dd{{\rm d}}
\def\Re{{\rm Re}} 
\def\Im{{\rm Im}}

\newcommand{\R}{\mathbb{R}}
\newcommand{\Z}{\mathbb{Z}}
\newcommand{\T}{\mathbb{T}}

\newcommand{\N}{\mathbb{N}}
\newcommand{\C}{\mathbb{C}}

\newcommand{\eps}{\varepsilon}

\newcommand{\vertiii}[1]{{\left\vert\kern-0.25ex\left\vert\kern-0.25ex\left\vert #1 
		\right\vert\kern-0.25ex\right\vert\kern-0.25ex\right\vert}}

\newtheoremstyle{mystyle}
  {}
  {}
  {\itshape}
  {}
  {\bfseries}
  {.}
  { }
  {}

  \newtheoremstyle{mystyle2}
  {}
  {}
  {}
  {}
  {\bfseries}
  {.}
  { }
  {}

\theoremstyle{mystyle}

\newtheorem{Thm}{Theorem}[section]
\newtheorem{Cor}[Thm]{Corollary}
\newtheorem{Prop}[Thm]{Proposition}
\newtheorem{Lem}[Thm]{Lemma}

\theoremstyle{mystyle2}
\newtheorem{Def}[Thm]{Definition}
\newtheorem{Rem}[Thm]{Remark}

\numberwithin{equation}{section}

\makeatletter

\renewcommand\subsubsection{\@startsection{subsubsection}{3}%
\normalparindent{.5\linespacing\@plus.7\linespacing}{-.5em}
{\normalfont\bfseries}}

\def\@tocline#1#2#3#4#5#6#7{\relax
  \ifnum #1>\c@tocdepth 
  \else
    \par \addpenalty\@secpenalty\addvspace{#2}%
    \begingroup \hyphenpenalty\@M
    \@ifempty{#4}{%
      \@tempdima\csname r@tocindent\number#1\endcsname\relax
    }{%
      \@tempdima#4\relax
    }%
    \parindent\z@ \leftskip#3\relax \advance\leftskip\@tempdima\relax
    \rightskip\@pnumwidth plus4em \parfillskip-\@pnumwidth
    #5\leavevmode\hskip-\@tempdima
      \ifcase #1
       \or\or \hskip 1em \or \hskip 2em \else \hskip 3em \fi%
      #6\nobreak\relax
    \dotfill\hbox to\@pnumwidth{\@tocpagenum{#7}}\par
    \nobreak
    \endgroup
  \fi}
\makeatother

\begin{document}

\title[Stability analysis for active Brownian particle models]{\vspace*{-3cm}Stability analysis for active Brownian particle models}

\author[M. Coti Zelati]{Michele Coti Zelati}
\address{Department of Mathematics, Imperial College London, SW7 2AZ, UK}
\email{m.coti-zelati@imperial.ac.uk}

\author[L. Ertzbischoff]{Lucas Ertzbischoff}
\address{CEREMADE, CNRS, Université Paris-Dauphine, PSL Research University, 75016 Paris, France}
\email{ertzbischoff@ceremade.dauphine.fr}

\author[D. Gérard-Varet]{David Gérard-Varet}
\address{Université Paris Cité and Sorbonne Université, CNRS, Institut de Mathématiques de Jussieu-Paris Rive Gauche (IMJ-PRG), F-75013, Paris, France}
\email{david.gerard-varet@u-paris.fr}

 \begin{abstract} 
      We carry out a comprehensive linear stability analysis of active Brownian particle systems around a constant homogeneous state. These scalar models, being important prototypes for the continuous description of active matter, are Fokker–Planck type equations in position–orientation and are known to exhibit motility-induced phase separation. We fully characterize the linear stability and instability regimes, with an explicit threshold depending on the effective speed of the particles. In this way, we rigorously confirm a conjecture on phase separation originating in the physics and applied literature. Our sharp and quantitative (in)stability results are valid both in the non-diffusive case and in the case of small angular diffusion. In the stable non-diffusive regime, we uncover a mixing mechanism reminiscent of Landau damping for the Vlasov equation, albeit with significantly weaker decay. This decay is non-integrable in time and gives rise to substantial mathematical difficulties; in particular, it prevents the use of classical perturbative arguments to treat the case of small angular diffusion.
 \end{abstract}

\maketitle

\tableofcontents

\section{Introduction, motivations and main statements}\label{SectionIntro-Intro}

In this paper, we study stability and instability phenomena in models of active matter. Active matter systems consist of self-propelled particles that convert chemical energy into mechanical work, and they encompass a wide variety of physical and biological systems such as colloidal suspensions and bacterial populations. In contrast to systems of passive particles, active matter is typically far from thermodynamic equilibrium and displays a broad range of collective behaviors.

Among these, motility-induced phase separation (MIPS) is a central phenomenon: self-propelled particles spontaneously segregate into dense and dilute regions despite the absence of attractive interparticle forces. An oversimplified explanation for MIPS is the following feedback mechanism: particles slow down in high-density regions, thereby spending more time where they move more slowly, which leads to further accumulation. 
This phenomenon has been extensively investigated in the physics literature---see, among many others, \cites{CatesTailleur2013active, CatesTailleur2015motility, o2023introduction, speck2015dynamical, bauerle2018self, fily2012athermal}---and has been linked to instability features of active matter systems, both at the microscopic particle level and at the kinetic or macroscopic scale.

\subsection{Active Brownian particle model}
At the kinetic level, a standard toy model for active Brownian particles is the Fokker--Planck type equation
\begin{align}\label{eq:fulleq-INTRO}
\partial_t f + \mathrm{div}_x \big[ f\, v(\rho_f)\, e(\theta) \big] = \nu\, \partial^2_{\theta} f,
\end{align}
where $f=f(t,x,\theta)\in\mathbb{R}^+$ denotes the particle distribution function. Setting $\T \simeq \R / 2\pi \Z$, the spatial variable is $x\in\T^2$, and the orientation variable is $\theta\in\T$. 
We write 
$$
e(\theta)\vcentcolon=(\cos(\theta),\sin(\theta)),
\qquad 
\rho_f(t,x)\vcentcolon=\int_0^{2\pi} f(t,x,\theta)\, \mathrm{d}\theta,
$$
and normalize $f$ so that 
$$
\phi_f \vcentcolon= \frac{1}{2\pi}\int_{\T^2}\rho_f(x)\,\mathrm{d}x
$$
is the solid volume fraction of the particles. 
The constant $\nu\ge 0$ represents the rotational diffusion coefficient (the inverse Péclet number), modeling random reorientation.

The function $v:\mathbb{R}^+\to\mathbb{R}^+$ encodes the effective particle speed and depends only on the local density. Throughout the paper we assume the structural condition
\begin{align}\label{assumption-velocity}
v\in\mathscr{C}^1([0,1];[0,1]), 
\qquad 
v((0,1))\subset (0,1), 
\qquad 
v'<0,
\end{align}
that is, the speed is strictly decreasing in the density. Under this assumption, \eqref{eq:fulleq-INTRO} may also be written as
\begin{equation} \label{eq:fulleq-INTRO-bis}
\partial_t f + v(\rho_f)\, e(\theta)\cdot \nabla_x f 
    + v'(\rho_f)\, f\, e(\theta)\cdot \nabla_x \rho_f 
    = \nu\, \partial^2_{\theta} f.
\end{equation}
The mechanism represented by \eqref{eq:fulleq-INTRO} is precisely the feedback mechanism discussed earlier: particles move with orientation $e(\theta)$, but their speed decreases when the local density is higher. This corresponds to a repulsive, purely density-dependent ``quorum sensing'' interaction \cite{CatesTailleur2015motility}. A frequent modelling choice, depending on the physical context, is the linear law $v(\rho_f)=1-\rho_f$ (see \cite{fily2012athermal}).

\medskip

Equation \eqref{eq:fulleq-INTRO} has appeared in several works on active Brownian particles, see for instance \cites{CatesTailleur2013active, o2023introduction}, and stands as an important prototype for the understanding of phase separation. A formal derivation from a microscopic system with repulsive interactions can be found in \cite{Bruna-models}, while a full hydrodynamic limit from an Active Lattice Gas (leading to a richer system than \eqref{eq:fulleq-INTRO}) has been obtained in \cite{mason2023hydrodynamics}, building on \cite{Erignoux-full}.

\medskip

A key observation in the physics literature (see e.g. \cites{CatesTailleur2013active, CatesTailleur2015motility, o2023introduction}) is that motility-induced phase separation is typically preceded by a linear instability of the homogeneous state. As a preliminary step toward understanding MIPS from a rigorous mathematical point of view, we therefore consider perturbations around the steady state
\begin{equation}\label{eq:equilibrium}
 f^\star = \frac{\phi}{2\pi}, 
\qquad 
\rho_{f^\star}=\phi, 
\qquad 
\phi\in(0,1).   
\end{equation}
The linearized equation for a perturbation $f$ around $f^\star$ is
\begin{align}\label{eq:linearizedEQ-intro-full}
  \partial_t f 
  + v(\phi)\, e(\theta)\cdot \nabla_x f 
  + \frac{\phi v'(\phi)}{2\pi}\, e(\theta)\cdot \nabla_x \rho  = \nu\, \partial^2_{\theta} f , \quad \rho = \int_0^{2\pi} f d\theta,
\end{align}
and in the non-diffusive regime $\nu=0$ it reduces to
\begin{align}\label{eq:linearizedEQ-intro-full-inviscid}
  \partial_t f
  + v(\phi)\, e(\theta)\cdot \nabla_x f
  + \frac{\phi v'(\phi)}{2\pi}\, e(\theta)\cdot \nabla_x \rho = 0 , \quad \rho = \int_0^{2\pi} f d\theta.
\end{align}
Understanding the stability properties of \eqref{eq:linearizedEQ-intro-full} and \eqref{eq:linearizedEQ-intro-full-inviscid} is the main objective of this work, shedding light on the onset of phase separation in active matter models. Broadly speaking, we identify a sharp, explicitly computable criterion, depending only on the value of the homogeneous state $\phi$, which completely characterizes the linear stability or instability of the dynamics. In particular, our analysis establishes rigorously the instability condition for motility-induced phase separation derived in the physics literature \cites{ CatesTailleur2015motility, o2023introduction}, and supported by numerical simulations \cite{Bruna-models}.

In the non-diffusive case $\nu=0$, we will actually obtain explicit formulas that allow for a precise spectral description and the identification of the sharp instability condition. One of our goals will be to show that the non-diffusive stable/unstable behavior as well as the former criterion persists and remains sharp when $0<\nu\ll 1$, despite the lack of explicit structure. It is one of the main analytical challenges of the paper: as we will explain later on, our arguments will be genuinely non-perturbative in $\nu$.  

\begin{Rem}\label{rem_neglect_space_diffusion}
   In \eqref{eq:fulleq-INTRO} and \eqref{eq:linearizedEQ-intro-full} we have omitted translational diffusion in the $x$-variable. As will be shown later, such a term can be included in the linearized analysis with no additional difficulty.
\end{Rem}

\subsection*{Notation}
For quantities $A$ and $B$, we write $A\lesssim B$ if $A\le cB$ for some universal constant $c>0$. 
If the dependence on a parameter $m$ needs to be explicit, we write $\lesssim_m$.
For a smooth periodic function $\mathrm{g}:\T^2\to\mathbb{R}^d$, its Fourier coefficients are denoted by
\[
\mathrm{g}_k 
\vcentcolon= 
\frac{1}{(2\pi)^2}
\int_{\T^2} \e^{-ik\cdot x}\, \mathrm{g}(x)\,\mathrm{d}x,
\qquad k\in\mathbb{Z}^2.
\]

\subsection{Main statements}\label{SectionIntro-Results}

Our two main theorems identify the sign of the quantity
\[
\partial_\rho (\rho v)\big\vert_{\rho=\phi}
\]
as the fundamental parameter governing linear (in)stability. Consistently with the formal analyses in \cites{o2023introduction, CatesTailleur2015motility}, we will prove that:
\begin{itemize}
    \item if $\phi \in (0,1)$ satisfies $\partial_\rho (\rho v)\big\vert_{\rho=\phi} > 0$, then the homogeneous state is linearly stable;
    \item if $\phi \in (0,1)$ is such that $\partial_\rho (\rho v)\big\vert_{\rho=\phi} < 0$, then linear instability arises and signals a phase--transition mechanism.
\end{itemize}
To our knowledge, a complete and quantitative description of the stability and instability regions for active matter models of the form \eqref{eq:fulleq-INTRO}, even at the linearized level, has not previously been available from a rigorous mathematical standpoint.

The mapping $\rho \mapsto F(\rho)=\rho v(\rho)$ naturally plays the role of a flux (see Remark~\ref{rem-crowdmotion} for an interesting connection with traffic--flow models). When $F$ is hump-shaped, the previous alternative corresponds to the following situation: there exists a threshold $\phi^\star \in (0,1)$ for which stability is obtained for $\phi<\phi^\star$, while instability holds for $\phi>\phi^\star$. For instance, for the widely used affine law $v(\rho)=1-\rho$, the criterion reduces to a simple low- versus high-density threshold at $\phi^*=\tfrac12$, in agreement with the numerical stability diagrams reported in \cite{Bruna-models}.

\medskip

Before turning to the precise statements, recall that for the linearized equations \eqref{eq:linearizedEQ-intro-full}--\eqref{eq:linearizedEQ-intro-full-inviscid}, the spatial average (the $k=0$ Fourier mode) simply evolves according to a conservation law when $\nu=0$ and according to the heat equation when $\nu>0$. In both cases it plays no role in the stability analysis, and we therefore focus on the nonzero modes $k\neq 0$.

\medskip

Our first theorem concerns the spectral instability of \eqref{eq:linearizedEQ-intro-full} and \eqref{eq:linearizedEQ-intro-full-inviscid}. Roughly speaking, this corresponds to the existence of exponentially growing mode solutions in the regime 
$\partial_\rho (\rho v)\vert_{\rho=\phi}<0$.

\begin{Thm}[Linear instability]\label{thm-instability} 
Assume that \eqref{assumption-velocity} holds, and let 
$\phi \in (0,1)$ satisfy the instability condition
\begin{align}\label{instab-condition-thm}
    \partial_\rho (\rho v)\big\vert_{\rho=\phi}<0.
\end{align}
Then the following hold:
\begin{itemize}

    \item \underline{Case $\nu=0$}: there exists an exponentially growing mode solution to the linearized equation \eqref{eq:linearizedEQ-intro-full-inviscid}. More precisely, for any 
    $k \in \Z^2 \setminus \{0\}$, there exists a solution of the form
    \begin{align}\label{eq:growing-mode-result_k}
        f(t,x,\theta)=\e^{\lambda t} \e^{ik \cdot x}\, f_k(\theta),
    \end{align}
    where $f_k$ is smooth and
    \begin{align*}
        \lambda = c_\phi\, |k|,
    \end{align*}
    for some constant $c_\phi \in \R^+ \setminus \{0\}$ independent of $k$.

    \item \underline{Case $\nu>0$}: for sufficiently small angular diffusion, the same unstable behaviour persists for the linearized equation \eqref{eq:linearizedEQ-intro-full}. More precisely, there exist constants $\nu_0>0$ and $C_{0,\phi}>0$ such that for any 
    $k \in \Z^2 \setminus \{0\}$ and any 
    $\nu \in (0, |k| \nu_0)$, there exists a solution of the form
    \begin{align}\label{eq:growing-mode-result-bis}
        f(t,x,\theta)=\e^{\lambda t} \e^{ik \cdot x}\, f_{k,\nu}(\theta),
    \end{align}
    with $\mathrm{Re}(\lambda)>0$, where $f_{k,\nu}$ is smooth and where $\lambda$ satisfies the quantitative estimate
    \begin{align*}
        |\lambda - c_\phi |k|| \,\leq\, C_{0,\phi}\, \nu.
    \end{align*}

\end{itemize}
\end{Thm}

\begin{Rem}\label{rmk-THM-instab}
Several remarks are now in order.
\begin{enumerate}
\item As already mentioned, the instability condition \eqref{instab-condition-thm} recovered in Theorem~\ref{thm-instability} coincides with the criterion derived in the physics literature (see, for instance, \cites{o2023introduction, CatesTailleur2015motility}). We shall later prove that this condition is in fact sharp at the linear level; see Theorem~\ref{thm-stability} below. For example, when the velocity law is affine, $v(\rho)=1-\rho$, the instability criterion becomes $\phi>1/2$, in agreement with the numerical observations of \cite{Bruna-models}.

\item Theorem~\ref{thm-instability} in particular implies that the linearized operator around any $\phi\in(0,1)$ satisfying \eqref{instab-condition-thm} possesses an unbounded unstable spectrum. This strong form of spectral instability can be attributed to the loss of a spatial derivative generated by the feedback term $f\,e(\theta)\cdot\nabla_x\rho$ in \eqref{eq:fulleq-INTRO-bis}. The existence of unstable modes of the form \eqref{eq:growing-mode-result_k}–\eqref{eq:growing-mode-result-bis}, with growth rate proportional to the spatial frequency (up to an $\mathcal{O}(\nu)$ correction), in fact implies that \eqref{eq:fulleq-INTRO} is linearly ill-posed in all Sobolev and Gevrey classes (but well-posed in analytic regularity).

When $\nu=0$, one may apply the framework of \cite{HKN} to deduce nonlinear ill-posedness in Sobolev regularity for \eqref{eq:fulleq-INTRO}. In particular, one may show failure of Hölder continuity of the flow map from $H^s$ to $L^2$, for arbitrarily large $s$ and arbitrarily small times. See also \cite{bianchini2025ill} for a recent application of these techniques. When $\nu>0$, nonlinear ill-posedness persists: at least the flow map fails to be Lipschitz from $H^s$ to $L^2$; see \cites{GuoTice,DavidToan} for results in this direction.

\item At the linear level with $\nu=0$, it is natural to notice an analogy between \eqref{eq:linearizedEQ-intro-full-inviscid} and the kinetic Vlasov--Benney equation for $f(t,x,v)$, linearized around a spatially homogeneous profile $\mu(v)$:
\begin{align*}
    \partial_t f + v \cdot \nabla_x f 
    = \nabla_v \mu \cdot \nabla_x \rho, 
    \qquad 
    \rho(t,x)=\int_{\R^d} f(t,x,v)\,\mathrm{d}v.
\end{align*}
The corresponding nonlinear model,
\begin{align*}
    \partial_t f + v \cdot \nabla_x f 
    - \nabla_x \rho_f \cdot \nabla_v f = 0,
    \qquad 
    \rho_f(t,x)=\int_{\R^d} f(t,x,v)\,\mathrm{d}v,
\end{align*}
has attracted significant attention in recent years as a notable example of a singular Vlasov equation. Questions of linear and nonlinear well-posedness (in finite regularity) for such systems are governed by the Penrose stability condition from plasma physics; see, for instance, \cites{BardosBesse, BardosNouri, HKN, HKR}. 

\item As noted in Remark~\ref{rem_neglect_space_diffusion}, we have omitted spatial diffusion for simplicity. Its inclusion is straightforward: one may consider
\begin{equation}\label{MIPS_with_diff}
\begin{aligned}
    \partial_t f 
    + v(\phi)e(\theta)\cdot \nabla_x f 
    + \frac{\phi v'(\phi)}{2\pi}\, e(\theta) \cdot \nabla_x \rho
    = \nu\, \partial_\theta^2 f + \kappa\, \Delta_x f,
\end{aligned}
\end{equation}
with $\kappa>0$. This model is considered in \cite{Bruna-models} in the special case $\kappa=\nu$, motivated by an adimensionalization based on the length scale $L=\sqrt{D_T/D_R}$, where $D_T$ and $D_R$ denote translational and rotational diffusion coefficients, respectively. In general, however, one may select a different characteristic spatial scale (e.g.\ the typical scale of the initial data), resulting in \eqref{MIPS_with_diff} with two distinct coefficients $\kappa$ and $\nu$.

In this setting, Theorem~\ref{thm-instability} still holds, except that the unstable eigenmodes become
\begin{align*}
    f(t,x,\theta)
    = \e^{(\lambda - \kappa |k|^2)\, t} \e^{ik\cdot x}\, f_{k}(\theta),
    \qquad \mathrm{Re}(\lambda)>0.
\end{align*}
In particular, if $\kappa < c_\phi - C_{0,\phi}\nu$, then this solution remains unstable for all wave numbers $k\in\Z^2$ satisfying
\[
    |k|
    < \sup\big\{ |k'| : \kappa  |k'|^2 < c_\phi  |k'| - C_{0,\phi}\nu  \big\}.
\]
At sufficiently high spatial frequencies, the instability is suppressed---a manifestation of the smoo\-thing effect of spatial diffusion. This aligns with the fact that, for $\kappa>0$, both the linearized and nonlinear equations become well-posed. In this small-diffusion regime, one can additionally transfer linear spectral instability to nonlinear Lyapunov instability by exploiting the parabolic regularization of the semigroup; see \cites{Shatah-Strauss, Fried-Strauss-Vishik, Fried-Pavlovic-Shvydkoy}.
\end{enumerate}
\end{Rem}
When the flux-sign condition 
\begin{equation}\label{stab-condition-thm}
  \partial_\rho (\rho v)\big\vert_{\rho=\phi} > 0
\end{equation}
holds, our second main result establishes the linear stability of the homogeneous equilibrium.

\begin{Thm}[Linear stability]\label{thm-stability}
Assume that \eqref{assumption-velocity} holds, and let $\phi \in (0,1)$ satisfy \eqref{stab-condition-thm}.
Then the following conclusions hold:

\begin{itemize}

    \item \underline{Case $\nu=0$.}  
    There exists no exponentially growing mode for the linearized equation \eqref{eq:linearizedEQ-intro-full-inviscid}. More precisely, if a solution of the form \eqref{eq:growing-mode-result_k} exists, then necessarily $\mathrm{Re}(\lambda)=0$.  
    Furthermore, there exists $C_\phi>0$ such that for any solution $f$ to \eqref{eq:linearizedEQ-intro-full-inviscid} with initial data $f^{\mathrm{in}}$, and for any Fourier mode $k \in \mathbb{Z}^2$, one has
    \begin{equation}\label{bound-rhoTHM-LANDADAMPING}
        |\rho_k(t)| \leq \frac{C_\phi}{\sqrt{1 + |k|  t}} 
         \| f^{\mathrm{in}}_k \|_{H^1_\theta},
        \qquad t \ge 0.
    \end{equation}

    \item \underline{Case $\nu>0$.}  
    There exist constants $\eta_0,\nu_0>0$, depending on $\phi$, and $C=C(\eta_0,\nu_0,\phi)>0$ such that the following holds. For any $k \in \mathbb{Z}^2\setminus\{0\}$ and any $\nu |k|^{-1} \le \nu_0$, every solution $f$ to the linearized equation \eqref{eq:linearizedEQ-intro-full} with initial data $f^{\mathrm{in}}$ satisfies
    \begin{align}
        \label{bound-rhoTHM}
        |\rho_k(t)| 
        &\le 
        C\left(1 + \nu^{-5/4} |k|^{5/4}\right)
        \e^{-\eta_0 \nu t}\,
        \| f^{\mathrm{in}}_k \|_{L^2_\theta}, \qquad t \ge 0 \\[0.3em]
        \label{bound-fTHM}
        \| f_k(t) \|_{L^2_\theta}
        &\le 
        C\left(1 + \nu^{-7/4} |k|^{7/4}\right)
        \e^{-\eta_0 \nu t}\,
        \| f^{\mathrm{in}}_k \|_{L^2_\theta} \qquad t \ge 0.
    \end{align}
    
\end{itemize}
\end{Thm}

\begin{Rem} \label{rem_thm2}
We gather here several comments regarding Theorem~\ref{thm-stability}.
\begin{enumerate}

\item  
Combined with Theorem~\ref{thm-instability}, which establishes linear instability of the homogeneous state $\phi$ when  
\[
\partial_\rho (\rho v)\big\vert_{\rho=\phi}<0,
\]
Theorem~\ref{thm-stability} shows that the threshold  
\[
\partial_\rho (\rho v)\big\vert_{\rho=\phi}=0
\]
is sharp for linear (in)stability.  
For instance, for the affine velocity law $v(\rho)=1-\rho$, the instability threshold is exactly $\phi^*=1/2$.

\item 
The non-diffusive decay estimate \eqref{bound-rhoTHM-LANDADAMPING} should be interpreted as a manifestation of a \emph{mixing} mechanism intrinsic to the model: solutions generate increasingly fine oscillations in the angular variable, producing a transfer of energy from low to high Fourier modes in $\theta$.  
This leads to decay in time of averaged quantities such as the spatial density, even though no genuine damping mechanism is present in the equation.  
Mixing-induced decay mechanisms play a central role in several areas:
Landau damping for the Vlasov--Poisson equation (on the torus) \cites{MouhotVillani, BMM-torus, GNR}, synchronization in Kuramoto models \cite{DFGV18}, and inviscid damping near shear flows in fluid mechanics \cites{BM15,BCZV19,WZZ18,WZZ19}.  
In our setting, this decay is very weak (in particular not integrable): estimate \eqref{bound-rhoTHM-LANDADAMPING} is expected to be essentially sharp even for highly regular initial data.  Indeed, solutions to the (Fourier transformed) ``free transport'' equation, 
\[
\partial_t g_k + i v(\phi)\, k \cdot  e(\theta) g_k = 0
\]
with generic $H^1_\theta$ initial data already exhibit the same $(1+t)^{-1/2}$ decay of the density (thanks a stationary phase argument, see Lemma~\ref{LM:inviscid-decay-FREETRANSPORT}).  
This originates from the presence of nondegenerate critical points in the phase  
\[
\T  \ni\theta \mapsto -i v(\phi) k \cdot e(\theta),
\]
in contrast with the usual Euclidean transport phase $\R^d  \ni \mathrm{v} \mapsto -i k \cdot \mathrm{v}$, for which higher regularity of the data yields much faster decay.

\item  
In the diffusive case $\nu>0$, it is likely that the decay estimates \eqref{bound-rhoTHM}--\eqref{bound-fTHM} are not optimal.  
Regarding the exponential decay rate, classical enhanced-diffusion estimates for transport--diffusion equations suggest an optimal rate of order  $\nu^{1/2} |k|^{1/2}$,
rather than $\nu$.  
Indeed, results from \cites{MicheleJacob, Wei, MicheleThierry} imply that $\nu^{1/2}|k|^{1/2}$ is the expected rate when the reaction term  
$\frac{\phi v'(\phi)}{2\pi} e(\theta)\cdot \nabla_x \rho$  
is absent from \eqref{eq:linearizedEQ-intro-full}.  

Likewise, the prefactors in \eqref{bound-rhoTHM}--\eqref{bound-fTHM} likely remain suboptimal, producing a fractional derivative loss greater than~$1$.  
If an improved decay rate of the form $\nu^{1/2}|k|^{1/2}$ (or more generally $\nu^\alpha |k|^\beta$ with $\beta>0$) were available, Gevrey regularization at positive times would immediately compensate for this derivative loss.  
Moreover, if a spatial diffusion term $-\kappa \Delta_x f$ were added, the factor $\e^{-\kappa |k|^2 t}$ would provide instantaneous strong smoothing for all $t>0$.

\item  
One of the main technical challenges in establishing \eqref{bound-rhoTHM}--\eqref{bound-fTHM} is the extreme weakness of the mixing produced by the transport field, which depends on a single angular variable and possesses a critical point.  
As shown by \eqref{bound-rhoTHM-LANDADAMPING}, this mixing yields at best a decay of order $(1+t)^{-1/2}$ for the spatial density.
This is one of the main mathematical obstruction to a standard perturbative approach when $\nu \ll 1$.  
We refer to Section~\ref{SectionIntro-Proof} for further discussion.

 \item  
Nonlinear stability in the regime $\partial_\rho (\rho v)\big\vert_{\rho=\phi}>0$ is a natural next question.  
Retaining only angular diffusion, the equation \eqref{eq:fulleq-INTRO} appears to lose one derivative in $x$, and its nonlinear well-posedness may already be delicate.  
As in stability problems for fluid or plasma models (see, e.g.,~\cites{BGM-navierstokes, MZ-navier-stokes, WZ-navier-stokes, B-VPFK, CLT} and references therein), an important goal would be to obtain a nonlinear stability threshold quantifying the basin of attraction of the homogeneous state in terms of the diffusion parameter.  
Such a result has recently been proved for a related active particle system in \cite{CZHGV-nonlin}.  
We plan to address the nonlinear problem in future work.
\end{enumerate}
\end{Rem}

Finally, as a byproduct of our analysis, we obtain a complete spectral description of the non-diffusive linearized operator, in the spirit of \cite{degond1986spectral} for the Vlasov--Poisson equation.  
We refer to Theorem~\ref{thm-SPECTRUM} for details.

To conclude, let us mention that Theorems~\ref{thm-instability}--\ref{thm-stability} also remain valid for a variant of the main equation \eqref{eq:fulleq-INTRO}, considered for instance in \cites{Bruna-models, AGSregularity, ASregularity}.  
The strategy followed to prove the former theorems is robust enough to be adapted and the analysis is similar to the one used for \eqref{eq:fulleq-INTRO}, though with additional technicalities.  
We state the corresponding results and outline the proof in Section~\ref{Appendix-Variant-eq}.

\subsection{Related works}\label{SectionIntro-Literature}
The toy model \eqref{eq:fulleq-INTRO} for active Brownian particles, studied at the continuous level, has recently attracted significant attention from the mathematical community. The existence and uniqueness of global weak solutions in the case of full diffusion and an affine density-dependent speed $v(\rho)=1-\rho$ was established in \cite{bruna-wellposedness-fulldiff}. Questions of uniqueness and regularity at the nonlinear level have been addressed in \cites{AGSregularity, ASregularity, BS} for more involved variants of the system with full diffusion, including stability results under sufficiently strong diffusion. 

Regarding the (in)stability mechanisms associated with phase separation, a recent contribution is \cite{Bruna-models}, where \eqref{eq:fulleq-INTRO} corresponds to \emph{Model~2} in their notation. The authors investigate the linearized dynamics around the same homogeneous equilibrium  \eqref{eq:equilibrium},
again with the affine choice $v(\rho)=1-\rho$,  and in the fully diffusive setting (diffusion in both space and orientation).  
More precisely, they solve analytically the eigenvalue problem associated with the \emph{symmetric part} of the linearized operator, identifying a threshold $\phi_{\mathrm{sym}} < \tfrac12$ for the stability of this symmetric evolution. As emphasized in \cite{Bruna-models}, while stability of the symmetric operator implies linear stability of the full model for $\phi < \phi_{\mathrm{sym}}$, the converse direction is not true: instability of the symmetric part does not, in general, imply instability of the full operator.

A numerical investigation of the full linearized operator is also performed in \cite{Bruna-models}, yielding a detailed stability/instability diagram in terms of the parameters $\phi$ and the Péclet number (which corresponds, up to rescaling, to the inverse of the rotational diffusion coefficient $\nu$). Notably, their numerical computations confirm the sharp threshold $\phi=\tfrac12$ that we rigorously establish in Theorems~\ref{thm-instability}--\ref{thm-stability}.  
We also refer to the numerical stability analysis in \cite{mason2023hydrodynamics}, which considers a more complete active Brownian particle model.

At a broader level, there has been a growing interest in understand stability or instability properties of active matter systems through the lens of mixing and enhanced dissipation phenomena. Recent works such as \cites{OhmShelley, albritton2023stabilizing, CZHGV, CZHGV-nonlin} provide rigorous results in that direction for the so-called Doi-Saintillan-Shelley, investigating the (in)stability of the uniform distribution of particles. We also refer to \cite{GuHe} for some result in the same spirit for the kinetic Vicsek equation.

\begin{Rem}\label{rem-crowdmotion}
    Let us mention\footnote{We would like to thank Charlotte Perrin for explaining this analogy to us.} a formal connection between the above instability dichotomy in active matter (depending on the constant-density equilibrium being perturbed) and well-known PDE models for crowd motion or traffic flow. In standard traffic flow systems, such as the LWR model \cites{lighthill1955kinematic, richards1956shock}, the quantity $F(\rho) = \rho\, v(\rho)$ is referred to as the flux function. Typically, one expects that before the critical density at which the flux reaches its maximum (i.e.\ when $F'(\rho)>0$), the system is in a \emph{free-flow} phase; whereas beyond this threshold (i.e.\ when $F'(\rho)<0$), a \emph{congested} phase is observed. We also note that the prescribed speed law $v(\rho)$ is chosen empirically from real data; an affine decreasing speed is a common modeling choice in that context as well.
\end{Rem}

\section{Strategy of proof}\label{SectionIntro-Proof}

The proofs of Theorems~\ref{thm-instability} and~\ref{thm-stability} rely on a detailed spectral analysis of the linearized dynamics around the homogeneous equilibrium.  
Before turning to the construction of unstable modes or decay estimates, it is convenient to exploit the invariances of the system and perform a natural rescaling.  
This reduces the problem to a one-parameter family of one-dimensional equations on the angular variable, which contain all the relevant stability information while stripping away inessential geometric and dimensional factors.

\subsection{Reduction} \label{Section:reductionScaling-STAB-MODELA}
We begin by writing the evolution equation for the spatial Fourier coefficients $f_k$ of the solution $f$ of \eqref{eq:linearizedEQ-intro-full}. For $t>0$ and $\theta\in\T$, we have
 \begin{equation}\label{eq:evolFouriercoeff-viscous-MODELA}
\begin{cases}
      \partial_t f_k+ v(\phi)i k \cdot e(\theta) f_k -\nu \partial_{ \theta}^2 f_k=  -\frac{\phi v'(\phi)}{2 \pi} i k \cdot e(\theta) \rho_k, \\[1mm]
f_k(0,\theta)=f^{\mathrm{in}}_k(\theta),
\end{cases}
     \qquad  \rho_k(t)=\int_{0}^{2\pi} f_k(t,\theta) \, \mathrm{d}\theta.
\end{equation}
For $k=0$, the equation reduces to the one-dimensional heat equation when $\nu >0$ (and to a trivial conserved dynamics when $\nu=0$), so from now on we restrict to wavevectors $k\neq 0$.
To simplify the analysis, we perform a reduction based on symmetries and a natural rescaling.

\smallskip

\noindent\textbf{Rotational invariance.} 
If $f_k[g]$ denotes the solution of the first equation of \eqref{eq:evolFouriercoeff-viscous-MODELA} with initial datum $g(\theta)$, then for any rotation $R_\alpha\in SO_2(\R)$ we have
\[
    f_{R_\alpha k}\big[g(\cdot-\alpha)\big](t,\theta)
    = f_k[g](t,\theta-\alpha).
\]
Thus, without loss of generality, we may assume that $k$ points in the vertical direction, namely $k = (0,|k|)$.
The transport term then becomes $ik\!\cdot\! e(\theta)= i|k| \sin\theta$, and the equation takes the form
\[
    \partial_t f_k
    + v(\phi)\, i|k|\sin\theta\, f_k
    - \nu\, \partial_\theta^2 f_k
    = -\frac{\phi v'(\phi)}{2\pi}\, i|k|\sin\theta\, \rho_k ,
\]
for a new unknown still denoted $f_k$.

\smallskip

\noindent\textbf{Rescaling.} 
Define the rescaled variables
\begin{equation}
    t' \vcentcolon= v(\phi)|k|\, t,
    \qquad
    k' \vcentcolon= \frac{k}{|k|},
    \qquad
    \nu' \vcentcolon= \frac{\nu}{v(\phi)|k|},
\end{equation}
the rescaled unknown
\[
    f'(t',\theta) \vcentcolon= f_k(t,\theta),
\]
and the parameter
\begin{equation}\label{def:zeta}
\zeta \vcentcolon= -\frac{\phi v'(\phi)}{2\pi\, v(\phi)}.
\end{equation}
Since $\partial_t f_k = v(\phi)|k|\, \partial_{t'} f'$, dividing the reduced equation by $v(\phi)|k|$ yields
\[
    \partial_{t'} f'
    + i\sin\theta\, f'
    - \nu'\, \partial_\theta^2 f'
    = \zeta\, i\sin\theta\, \rho',
    \qquad
    \rho'(t') = \int_\T f'(t',\theta)\, \mathrm{d}\theta.
\]
Dropping the primes for readability, we arrive at the reduced diffusive equation
\begin{equation}\label{eq:rescaled eq2-MODELA}
    \partial_t f
    + i\sin\theta\, f
    - \nu\, \partial_\theta^2 f
    = \zeta\, i\sin\theta\, \rho,
    \qquad
    \rho(t)=\int_\T f(t,\theta)\, \mathrm{d}\theta,
\end{equation}
and, in the non-diffusive case,
\begin{equation}\label{eq:rescaled eq2-MODELA-inv}
    \partial_t f
    + i\sin\theta\, f
    = \zeta\, i\sin\theta\, \rho,
    \qquad
    \rho(t)=\int_\T f(t,\theta)\, \mathrm{d}\theta.
\end{equation}
Note that, since $v'(\phi)<0$ (see Assumption \eqref{assumption-velocity}), we have $\zeta>0$. Furthermore, $\zeta> \frac{1}{2\pi}$ (resp. $\zeta< \frac{1}{2\pi}$) if and only if $\partial_\rho (\rho v)\big\vert_{\rho=\phi}
<0$ (resp. $\partial_\rho (\rho v)\big\vert_{\rho=\phi}
>0$).

\medskip

We will henceforth work with the reduced equations \eqref{eq:rescaled eq2-MODELA} and \eqref{eq:rescaled eq2-MODELA-inv}.  
It is straightforward to verify that Theorems~\ref{thm-instability} and~\ref{thm-stability} follow from the corresponding spectral and damping statements established for these reduced systems.

\begin{Thm}[Linear instability - reduced]\label{thm-instability-reduced} 
Let $\zeta$ defined as in \eqref{def:zeta}. If $\zeta > \frac{1}{2\pi}$, then the following conclusions hold:
\begin{itemize}

    \item \underline{Case $\nu=0$}: there exists a  solution to  \eqref{eq:rescaled eq2-MODELA-inv} of the form 
    \begin{align}\label{eq:growing-mode-result}
        f(t,\theta)=\e^{\lambda_\zeta t} \mathrm{f}_\zeta(\theta),
    \end{align}
    where $\mathrm{f}_\zeta$ is smooth and $\lambda_\zeta \in \R^+ \setminus \{0\}$. 
    
    \item \underline{Case $\nu>0$}: there exist constants $\nu_0>0$ and $C_0 >0$ such that for any any 
    $\nu \in (0, \nu_0)$, there exists a solution to  \eqref{eq:rescaled eq2-MODELA} of the form
    \begin{align}
        f(t,\theta)=\e^{\lambda^\nu_\zeta t} \mathrm{f}^\nu_\zeta(\theta),
    \end{align}
    with $\mathrm{Re}(\lambda_{\zeta,\nu})>0$, where $\mathrm{f}^\nu_\zeta$ is smooth and where $\lambda^\nu_\zeta$ satisfies the quantitative estimate
    \begin{align*}
        |\lambda_\zeta^\nu - \lambda_\zeta | \,\leq\, C_0 \, \nu.
    \end{align*}
\end{itemize}
\end{Thm}

\begin{Thm}[Linear stability - reduced]\label{thm-stability-reduced}
Let $\zeta$ defined as in \eqref{def:zeta}. If  $\zeta < \frac{1}{2\pi}$, then the following conclusions hold:

\begin{itemize}

    \item \underline{Case $\nu=0$.}  
    There exists no exponentially growing mode for the linearized equation \eqref{eq:rescaled eq2-MODELA-inv}. More precisely, if a solution of the form \eqref{eq:growing-mode-result} exists, then necessarily $\mathrm{Re}(\lambda_\zeta)=0$.  
    Furthermore, there exists $C_\zeta >0$ such that for any solution $f$ to \eqref{eq:rescaled eq2-MODELA-inv} with initial data $f^{\mathrm{in}}$, one has
    \begin{equation}\label{bound-rhoTHM-LANDADAMPING-reduced}
        |\rho(t)| \leq \frac{C_\zeta}{\sqrt{1 + t}} 
         \| f^{\mathrm{in}}\|_{H^1_\theta},
        \qquad t \ge 0.
    \end{equation}

    \item \underline{Case $\nu>0$.}  
    There exist constants $\eta_0,\nu_0>0$, depending on $\eta$, and $C=C(\eta_0,\nu_0,\zeta)>0$ such that the following holds. For any $\nu \le \nu_0$, every solution $f$ to the linearized equation \eqref{eq:rescaled eq2-MODELA} with initial data $f^{\mathrm{in}}$ satisfies
    \begin{align}
        \label{bound-rhoTHM-reduced}
        |\rho(t)| 
        &\le 
        C\left(1 + \nu^{-5/4}\right)
        \e^{-\eta_0 \nu t}\,
        \| f^{\mathrm{in}} \|_{L^2_\theta}, \qquad t\ge 0,\\[0.3em]
        \label{bound-fTHM-reduced}
        \| f(t) \|_{L^2_\theta}
        &\le 
        C\left(1 + \nu^{-7/4}\right)
        \e^{-\eta_0 \nu t}\,
        \| f^{\mathrm{in}} \|_{L^2_\theta}, \qquad t\ge 0.
    \end{align}
\end{itemize}
\end{Thm}

In the next paragraph we present the main ideas underlying the proofs of Theorems~\ref{thm-instability-reduced}--\ref{thm-stability-reduced}, together with a brief outline of the article.

\subsection{Main steps of the proof}
We begin in \textbf{Section~\ref{Section-spectral-inviscid}} with the search for unstable modes in the non-diffusive problem, i.e.\ system~\eqref{eq:rescaled eq2-MODELA-inv}. We look for solutions of the form
\begin{align}\label{def:eigenmode-intro}
    f(t,\theta)=\e^{\lambda t} \mathrm{f}(\theta),
    \qquad \Re (\lambda) \neq 0.
\end{align}
We then obtain a necessary and sufficient condition on $\lambda$: eigenvalues must solve the \emph{non-diffusive dispersion relation} $D_\zeta(\lambda) = 0$, where
\begin{equation}\label{eq:disprel}
     D_\zeta(\lambda)
    \vcentcolon= 1 - i \zeta \int_0^{2\pi}
        \frac{\sin(\theta)}{\lambda + i \sin(\theta)}
    \, \mathrm{d}\theta .   
\end{equation}
Remarkably, use of the residue theorem allows to derive a very explicit expression for $D_\zeta$:
    \begin{align}\label{def:rescaledDispersionRelation}
D_\zeta(\lambda)=\frac{(1-2\pi \zeta)\sqrt{1+\frac{1}{\lambda^2}}+2\pi \zeta}{\sqrt{1+\frac{1}{\lambda^2}}}\, ,
\end{align}
from which one can completely determine its zeros.  
In the regime $\zeta > \frac{1}{2\pi}$, one finds non-trivial zeros $\lambda_\zeta$ with $\mathrm{Re}(\lambda_\zeta)\neq 0$, producing an unstable mode; while in the regime $\zeta < \frac{1}{2\pi}$ no zeros exist, yielding spectral stability.  
This essentially proves Theorems~\ref{thm-instability-reduced} and \ref{thm-stability-reduced} in the non-diffusive case $\nu=0$, except for the decay estimate \eqref{bound-rhoTHM-LANDADAMPING-reduced}.  
\medskip

Next, \textbf{Section~\ref{Section-Landau-damping}} is devoted to establishing the polynomial decay rate \eqref{bound-rhoTHM-LANDADAMPING-reduced} in the case $\nu=0$ and $\zeta < \frac{1}{2\pi}$ (the stable regime).  
Applying the Laplace transform in time  to the equation satisfied by $f$, and using computations analogous to those in Section~\ref{Section-spectral-inviscid}, we arrive at the algebraic identity
\[
    \mathcal{L}[\rho](\lambda)
    = \frac{1}{D_\zeta(\lambda)}\, F(\lambda),
    \qquad \Re (\lambda) \gg 0,\qquad k\in \mathbb{Z}^2,
\]
where $\mathcal{L}$ denotes the Laplace transform in time, $F$ depends only on the initial data and on the resolvent of the free-transport operator, and $D_\zeta$ is given in \eqref{eq:disprel}.  
Inverting the Laplace transform yields a convolution representation
\[
\rho(t)
    = S(t) +  \int_0^t G(t-\tau)\, S(\tau)\, \mathrm{d}\tau ,
\]
where $S$ is the contribution generated by the free-transport dynamics, and $G$ is the inverse Laplace transform of $\frac{1}{D_\zeta} - 1$.

As explained earlier, the free-transport term $S$ is easily controlled via a stationary phase argument (see Lemma~\ref{LM:inviscid-decay-FREETRANSPORT}), giving  the decay
\[
    |S(t)| \;\lesssim\; \frac{1}{\sqrt{1+t}}\,
        \| f^{\mathrm{in}} \|_{H^1_\theta},
    \qquad t>0.
\]
Thus the main difficulty lies in establishing the decay of $G$. A suitable contour deformation, avoiding the branch cut $i[-1,1]$ (in the spirit of a limiting absorption argument), leads to an oscillatory integral representation of~$G$:
\begin{align}\label{eq:formulaGreen}
        G(t)=4\pi \zeta \int_{-1}^1 \e^{i s  t} \sqrt{1-s^2} \frac{s}{(4\pi \zeta-1)s^2 + (1-2\pi \zeta)^2} \, \mathrm{d}s.
        \end{align} 
Although the integral exhibits endpoint singularities at $s=\pm1$, a careful decomposition into a regular part and a thin boundary layer allows us to exploit the oscillatory structure and infer the decay estimate
\[
    |G(t)|
    \;\lesssim\;
    \frac{1}{(1+t)^{3/2}} .
\]
Finally, combining the convolution formula for $\rho$, the $t^{-1/2}$ decay for $S$, and the $t^{-3/2}$ decay for $G$, we obtain the polynomial Landau damping rate stated in Theorem~\ref{thm-stability} for $\nu=0$.

\medskip

In \textbf{Section \ref{Section-instab-diffusive}}, we aim at proving that the non-diffusive instability obtained in Section \ref{Section-spectral-inviscid}, holding for $\zeta > \frac{1}{2\pi}$, remains valid for \eqref{eq:rescaled eq2-MODELA}
at least in the weak diffusive regime $0 <\nu \ll 1$. This is the result contained in Theorem \ref{thm-instability-reduced} in the case $\nu>0$.

At this stage, explicit expressions are not available anymore because rotational diffusion has been reintroduced. To prove that instability still persists, we set up a natural perturbative approach in the small diffusion regime $\nu \rightarrow 0$. The main starting point is  to derive a diffusive dispersion relation $   D_\zeta^\nu(\lambda)=0$ where 
\begin{align*}
    D_\zeta^\nu(\lambda)=1- i\zeta \int_0^{2\pi} (\lambda+i\sin(\theta) -\nu \partial_\theta^2)^{-1}[\sin](\theta) \, \mathrm{d}\theta.
\end{align*}
We do not have a very explicit formula for this function. Still, we are able to show that 
\begin{align*}
    \left\vert D_\zeta^\nu(\lambda) - D_\zeta(\lambda) \right\vert \lesssim \nu
\end{align*}
for small enough $\nu>0$ and for $\lambda$  in the vicinity of the non-diffusive unstable root $\lambda_\zeta > 0$. Obtaining such inequality is based on a direct energy estimate at the level of the resolvent equations. To conclude, we rely on Rouché's theorem and the analyticity of the previous dispersion relations in $\lambda$, yielding the existence of a zero $\lambda^\nu_\zeta$ of $D_\zeta^\nu$ close to $\lambda_\zeta$ for $\nu>0$ small enough. This strategy is inspired by previous works \cites{DDGM, HG} on linear instability of boundary-layers type models. Quantifying the distance between the new diffusive unstable eigenvalue  and the former non-diffusive one is then a consequence of the order of vanishing of $D_\zeta$, which is easily computed.

\medskip

In \textbf{Section \ref{Section-stab-diffusive}}, we turn to the proof of the  linear stability for \eqref{eq:rescaled eq2-MODELA} when $\zeta < \frac{1}{2\pi}$, that is Theorem \ref{thm-stability-reduced} in the case $\nu>0$, and from which we can infer that the threshold $\zeta = \frac{1}{2\pi}$ is the correct sharp threshold. The main goal is to derive some decay estimates precisely quantified in terms of $\nu$. Our approach is based on the study of a Volterra equation on the density.  This is  fairly standard for such type of equations  \cites{B-VPFK,CLT,albritton2023stabilizing,CZHGV}. It reads 
\begin{align*}
    \rho(t)-\zeta \int_0^t \mathrm{K}_{\nu}(t-s) \rho(s) \, \mathrm{d}s = S(t),
\end{align*}
with a kernel and a source terms defined by \begin{align*}
    \mathrm{K}_{\nu}(t)=\int_0^{2\pi} \left(\e^{-\mathrm{L}_{\nu} t} i \sin \right)(\theta) \, \mathrm{d}\theta, \qquad 
    S(t)=\int_0^{2\pi} \left(\e^{-\mathrm{L}_{\nu} t} f^{\mathrm{in}} \right)(\theta) \, \mathrm{d}\theta.
\end{align*}
where $L_\nu = i \sin(\theta) - \nu \partial^2_\theta$ is the advection-diffusion operator. A key estimate in the rest of the proof is the following exponential decay of the semigroup $\e^{-\mathrm{L}_{\nu} t}$:
\begin{align*}
\Vert \e^{-\mathrm{L}_{\nu}t} \Vert_{\mathcal{L}(L^2_\theta)} \leq C \e^{-\eps \nu^{1/2}t}, \qquad t>0, \  \nu \ll 1,
\end{align*}
for some universal constants $C,\eps>0$, referred to as an enhanced-dissipation estimate (see e.g. \cites{MicheleJacob, MicheleMathiasTarek, MicheleThierry}).
This same exponential decay in time passes easily  to the kernel $K_\nu$ and source $S$ above. The main point is to infer from it the decay estimate on $\rho$ expressed in \eqref{bound-rhoTHM-reduced}. The estimate \eqref{bound-fTHM-reduced} then follows directly, by combining Duhamel's formula and the enhanced dissipation estimate.

\medskip
The proof of \eqref{bound-rhoTHM-reduced} relies crucially on the two properties:
\begin{align}\label{Intro-stab-cond}
    1-\zeta \mathcal{L}[\mathrm{K}_{\nu}](\lambda) \neq 0, \qquad \lambda \in \lbrace \mathrm{Re} \geq -\delta \rbrace,
\end{align}
and 
\begin{align}\label{Intro-stab-cond-quantitative}
    \vert 1-\zeta \mathcal{L}[\mathrm{K}_{\nu}](\lambda) \vert  \geq m, \qquad \lambda \in \lbrace -\delta \le \mathrm{Re} <0 \rbrace,
\end{align}
that we show to hold for $\delta = c\nu$ and $m = c' \nu$, where $c,c'$ are positive constant independent of $\nu$. 
Following the classical theory of the Volterra equation, for which a concise reminder is recorded in Appendix \ref{Appendix-Volterra}, one can deduce from \eqref{Intro-stab-cond} a bound of the form  $|\rho(t)| \le C_\nu e^{- c \nu t}$, while the quantitative inequality \eqref{Intro-stab-cond-quantitative} allows to quantify $C_\nu$ in terms of $\nu$, leading to \eqref{bound-rhoTHM-reduced}.  
Echoing Remark \ref{rem_thm2}-(3), the limitations $\delta = c \nu$ and $m = c' \nu$ needed to  show  \eqref{Intro-stab-cond-quantitative} explain the final decay rates and prefactors in the estimates \eqref{bound-rhoTHM-reduced}--\eqref{bound-fTHM-reduced} from Theorem \ref{thm-stability-reduced}. We are in particular unable to reach the values\footnote{Such improvement of \eqref{Intro-stab-cond}--\eqref{Intro-stab-cond-quantitative} would rather yield a decay estimate like $\vert \rho(t) \vert \lesssim (1+\nu^{-3/4}) \e^{-\eps \nu^{1/2} t}\Vert f^{\mathrm{in}} \Vert_{L^2_\theta}$. 
This estimate, that we conjecture to be optimal, seems out of reach within our approach.} $\delta \sim \nu^{1/2}$ and $m \sim \nu^{1/2}$. 

\medskip
We stress that the derivation of \eqref{Intro-stab-cond}--\eqref{Intro-stab-cond-quantitative}  is one of the main contributions of our paper, notably because our approach is not  of a perturbative nature: it means that we do not rely on a comparison between the case $\nu=0$ and $\nu \rightarrow 0$ to prove that the stability condition \eqref{Intro-stab-cond} holds for $\nu>0$ small enough. This contrasts with the strategy devised for instance in \cites{CLT, CZHGV}.  
In our case, the limiting non-diffusive kernel $K_{\mid \nu=0}$ at most displays a very slow decay of order $(1+t)^{-1/2}$, which is not integrable for large times. This weak decay is likely to be sharp, since it is based on the standard decay of 1d Bessel function. This lack of time integrability prevents from obtaining a good estimate in $L^1_t(\R^+)$ for the difference between the viscous kernel $K_\nu$ and its non-diffusive counterpart, an estimate that is crucial to the proofs in \cites{CLT, CZHGV}.  
We therefore adopt a more direct point of view to prove \eqref{Intro-stab-cond}-\eqref{Intro-stab-cond-quantitative}, using the identity
\begin{align*}
        \mathcal{L}[\mathrm{K}_{\nu}](\lambda) =\int_0^{2\pi} f_{\lambda,\nu} \mathrm{d}\theta,
    \end{align*}
where $(\lambda + L_\nu) f_{\lambda , \nu} = i \sin(\theta)$. 
We derive a refined set of  estimates for this latter equation in the region $\{\Re (\lambda) \ge - c \nu\}$, distinguishing into several sub-regions. This allows to derive the expected  lower bounds on $|1 - \zeta \mathcal{L}[K_\nu]|$. 
We refer to \textbf{Section \ref{Section-stab-diffusive}} for all details.

\section{Spectral study without diffusion}\label{Section-spectral-inviscid}
We begin the study of the linearized equations in the absence of spatial or rotational diffusion, that is assuming $\nu=0$ in the rest of this section. We therefore focus on \eqref{eq:rescaled eq2-MODELA-inv}.

We introduce the following non-diffusive dispersion relation that will prove to be central in our analysis of the non-diffusive case.
\begin{Def}\label{def-dispersion-rel-inviscid}
    For $\lambda \in \C \setminus i[-1 , 1]$, we define 
    \begin{align}\label{eq:dispersion-rel-inviscid}
        D_\zeta(\lambda)\vcentcolon=1-i\zeta\int_0^{2\pi} \frac{\sin(\theta)}{\lambda+i\sin(\theta)} \, \mathrm{d}\theta,
    \end{align}
    where we recall that the parameter $\zeta>0$ has been introduced in \eqref{def:zeta}.
\end{Def}
By standard holomorphy under the integral, $\lambda \mapsto D_\zeta(\lambda)$ is holomorphic on $\C \setminus i[-1, 1 ]$. Furthermore, by a change of variable and periodicity, we have the symmetry property $D_\zeta(\lambda)=D_\zeta(-\lambda)$. The importance of $D_\zeta$ is given by the following: 

\medskip

\begin{Prop}\label{Prop-iifcond-growingmode}
For \eqref{eq:rescaled eq2-MODELA-inv} to admit a non-trivial solution of the form 
\begin{align}\label{eq:growing-mode-inviscid}
    f(t,\theta)=\e^{\lambda t}  \mathrm{f}(\theta),  \: \lambda \notin i[-1,1],
\end{align}
 it is necessary and sufficient that
    \begin{align*}
        D_\zeta(\lambda)=0.
    \end{align*}
\end{Prop}
\begin{proof}
By looking for a solution of \eqref{eq:rescaled eq2-MODELA-inv} under the form \eqref{eq:growing-mode-inviscid}, we observe that it has to satisfy
 \begin{align}\label{eq:Fourier-inviscid}
     (\lambda+i\sin(\theta))\mathrm{f}(\theta)=i\zeta \sin(\theta) \rho, \qquad \rho =\int_{0}^{2\pi} \-\mathrm{f}(\theta) \, \mathrm{d} \theta.
 \end{align}
After division and integration of \eqref{eq:Fourier-inviscid}, we also get from \eqref{def-dispersion-rel-inviscid} that
\begin{align*}
    D_\zeta(\lambda) \rho=0.
\end{align*}
Now, because of \eqref{eq:Fourier-inviscid},  we cannot have $\rho=0$ (using the equation \eqref{eq:Fourier-inviscid}, it would imply $\mathrm{f}(\theta)=0$ for all $\theta$). Hence we have $D_\zeta(\lambda)=0$. 

Conversely, let us suppose that $D_\zeta(\lambda)=0$ for some $ \lambda \notin i[-1,1]$. By definition, we have
\begin{align*}
    1=i\zeta \int_0^{2\pi} \frac{\sin(\theta)}{\lambda+i\sin(\theta)} \, \mathrm{d}\theta,
\end{align*}
therefore, if we define $\mathrm{f}(\theta)$ by
\begin{align*}
\mathrm{f}(\theta)&\vcentcolon=i\zeta \frac{\sin(\theta)}{\lambda+i\sin(\theta)},
\end{align*}
we get $ \rho=\int_{0}^{2\pi} \mathrm{f}(\theta) \, \mathrm{d} \theta=1$ and then
\begin{align*}
    (\lambda+i\sin(\theta))\mathrm{f}(\theta)=i\zeta \sin(\theta) \rho, \qquad \rho =\int_{0}^{2\pi} \-\mathrm{f}(\theta) \, \mathrm{d} \theta.
\end{align*}
This precisely means that the function $f(t,x,\theta)\vcentcolon=\e^{\lambda t}  \mathrm{f}(\theta)$ is solution to \eqref{eq:rescaled eq2-MODELA-inv}, which concludes the proof.
\end{proof}

To analyze the zeros of the dispersion relation $D_\zeta(\lambda) = 0$, we will rely on the following lemma.
\begin{Lem}\label{LM:formula-dispersionD}
For all  $\lambda \in \C \setminus i[-1, 1]$, $D_\zeta(\lambda)$ is given by the formula  \eqref{def:rescaledDispersionRelation}, in which  $\sqrt{\cdot}$ is the principal determination of the square-root, that is defined and holomorphic on $\C \setminus \R_-$ and that has a positive real part.
\end{Lem}

\begin{Rem}
    In the case of an affine decreasing speed $v(\rho)=1-\rho$, $\zeta = \frac{\phi}{2\pi (1-\phi)}$ and \eqref{def:rescaledDispersionRelation} simplifies into
    \begin{align*}
        D_\zeta(\lambda)= \frac{1-2\phi}{1-\phi}+\frac{\phi}{1-\phi}\frac{1}{\sqrt{1+\frac{1 }{\lambda^2}}}.
    \end{align*}
We have not been aware of such an explicit formula in the literature.
\end{Rem}

\begin{proof}[Proof of Lemma \ref{LM:formula-dispersionD}]
Let us compute the integral 
$$ \mathcal{I} \vcentcolon= \int_0^{2\pi} \frac{ \sin(\theta)}{\lambda+i\sin(\theta)} \, \mathrm{d}\theta
=\frac{2\pi}{i}-\frac{\lambda}{i}\int_{-\pi}^{\pi} \frac{ 1}{\lambda+ i \sin(\theta)} \, \mathrm{d}\theta,
 $$
and by changing variables with $$t=\tan\left( \frac{\theta}{2} \right), \qquad \mathrm{d}\theta=\frac{2}{1+t^2} \, \mathrm{d}t, \qquad \sin(\theta)=\frac{2t}{1+t^2},$$ 
we get \begin{align*}
\mathcal{I} =\frac{2\pi}{i }-\frac{2\lambda}{i}\int_{-\infty}^{\infty} \frac{ 1}{\lambda(1+t^2)+2 i t} \, \mathrm{d}t =\frac{2\pi}{i}-\frac{2}{i}\mathcal{J}, 
\end{align*}
where $$\mathcal{J} \vcentcolon= \int_{-\infty}^{\infty} \frac{ 1}{t^2+2i t/\lambda +1} \, \mathrm{d}t.$$
Note that for all $\lambda \in \C\setminus i[-1 , 1]$, there holds $i/\lambda  \in \C\setminus\left((-\infty,-1] \cup [1,+\infty)\right)$. We can therefore apply Lemma \ref{LM:integral-rationalfrac} in the Appendix  with $z = \frac{i}{\lambda}$ that yields
\begin{align*}
\mathcal{J}=\frac{\pi}{\sqrt{1+\frac{1}{\lambda^2}}}, \qquad \mathcal{I} = \frac{2\pi}{i}-\frac{2}{i}\frac{\pi}{\sqrt{1+\frac{1}{\lambda^2}}},
\end{align*}
and then leads to 
\begin{align*}
\mathcal{D}_\zeta(\lambda )=1-2\pi \zeta  \left[1-\frac{1}{\sqrt{1+\frac{1}{\lambda^2}}} \right] = 1-2\pi \zeta+\frac{2\pi \zeta}{\sqrt{1+\frac{1}{\lambda^2}}},
\end{align*}
that is
\begin{align*}
D_\zeta(\lambda)=\frac{(1-2\pi \zeta)\sqrt{1+\frac{1}{\lambda^2}}+2\pi \zeta}{\sqrt{1+\frac{1}{\lambda^2}}},
\end{align*}
as claimed.
\end{proof}
We are now interested in the roots of $D_\zeta(\lambda)$. Roughly speaking, when $\zeta<\frac{1}{2\pi}$, it has no roots where it is defined. When $ \zeta>\frac{1}{2\pi}$, two roots having a nonzero real part exist, hence the existence of growing modes with a growth in time.
\begin{Prop}\label{Prop-instalinear-inviscid}
We have the following alternative:
\begin{itemize}
    \item if $\zeta>\frac{1}{2\pi}$, then $\lambda \mapsto D_\zeta(\lambda)$ has exactly two distinct zeros in $\C \setminus i[-1, 1]$, given by  
    \begin{align*}
\lambda_\zeta^\pm&=\mp\frac{ 1-2\pi \zeta}{2 \pi^{1/2} \zeta^{1/2}} \in \R^\pm. 
    \end{align*}
    \item if $\zeta<\frac{1}{2\pi}$, then  $\lambda \mapsto D_\zeta(\lambda)$ has no zeros in $\C \setminus i[-1, 1] $.
\end{itemize}    
\end{Prop} 
\begin{proof}
By the symmetry property $D_\zeta(\lambda) = D_\zeta(-\lambda)$, it is enough to look for zeros of the dispersion relation for $\lambda \in \lbrace \mathrm{Re} \ge 0 \rbrace$. According to Lemma \ref{LM:formula-dispersionD} and formula \eqref{def:rescaledDispersionRelation}, we have
\begin{align*}
\mathcal{D}_\zeta(\lambda)=0 \Longleftrightarrow (1-2\pi \zeta)\sqrt{1+\frac{1}{\lambda^2}}=-2\pi \zeta.
\end{align*}
As $\zeta>0$, the right-hand side is negative. As a consequence, if $ \zeta <\frac{1}{2\pi}$, we observe that the former equation has no solutions since the real part of the left-hand side is non-negative. If $ \zeta >\frac{1}{2\pi}$, there is a solution of the equation if and only if
\begin{align*}
1+\frac{1}{\lambda^2}=\frac{4\pi^2 \zeta^2}{(1-2\pi \zeta)^2},
\end{align*}
that is after some simplifications
\begin{align*}
\lambda^2=\frac{(1-2\pi \zeta)^2}{4\pi \zeta}.
\end{align*}
Since $\mathrm{Re}(\lambda)>0$, we consider the positive root which yields
\begin{align*}
\lambda =-\frac{ 1-2\pi \zeta}{2 \pi^{1/2} \zeta^{1/2}}.
\end{align*}
 It concludes the proof. 
\end{proof}

As a conclusion, and more precisely thanks to Proposition \ref{Prop-instalinear-inviscid}, we have obtained the first part of Theorem \ref{thm-instability-reduced} and of Theorem \ref{thm-stability-reduced} for the absence of unstable eigenmode, in the case $\nu=0$.

\medskip

To conclude this section, we provide a non-diffusive spectral picture for the projected linearized problem, in the following sense. We define the linearized operator
\begin{align*}
    A_\zeta\vcentcolon=-i \sin(\theta)f+i\zeta  \sin(\theta)\rho, \qquad \rho=\int_0^{2\pi} f(\theta) \, \mathrm{d}\theta
\end{align*}
on $L^{2}_\theta$. It is clear that $A_{\zeta}$ is  bounded  on $L^2_\theta$. Its spectrum can be described as follows.

\begin{Thm}\label{thm-SPECTRUM}
If $\zeta \neq \frac{1}{2\pi}$, then the spectrum of $L_\zeta$ is 
    \begin{align*}
    \mathrm{Sp}(A_\zeta)=\big\lbrace i\tau \mid \tau \in [-1, 1] \big\rbrace \cup \lbrace \lambda \in \C \setminus i \R \mid D_{\zeta}(\lambda)=0 \rbrace,
\end{align*} 
and is symmetric with respect to $z \mapsto -z$. If $\lambda \notin \mathrm{Sp}(A_\zeta)$, the resolvent $R(\lambda)\vcentcolon=(\lambda-L_{\zeta})^{-1}$ satisfies for all $H \in L^2_\theta$
    \begin{align}\label{eq:formula-RESOLVENT}
        [R(\lambda)H](\theta)=\frac{H(\theta)}{\lambda+i\sin(\theta)} + \zeta D_\zeta(\lambda)^{-1}\left(\int_0^{2\pi} \frac{H(\theta')}{\lambda+i\sin(\theta')} \, \mathrm{d}\theta'\right) \frac{i \sin(\theta)}{\lambda+i\sin(\theta)}
        ,
    \end{align}
    and 
    \begin{align}\label{eq:formula-RESOLVENT-average}
        \int_0^{2\pi} [R(\lambda)H](\theta) \, \mathrm{d}\theta=D_\zeta(\lambda)^{-1}\int_0^{2\pi} \frac{H(\theta')}{\lambda+i\sin(\theta')} \, \mathrm{d}\theta'.
        \end{align}
\end{Thm}
In view of Proposition \ref{Prop-instalinear-inviscid}, a direct consequence of the former theorem in the stable case is the following.
\begin{Cor}\label{coro-stable-SPECTRUM}
    If $\zeta < \frac{1}{2\pi}$, then 
\begin{align*}
    \mathrm{Sp}(A_{\zeta})=\big\lbrace i\tau \mid \tau \in [-1 , 1] \big\rbrace.
\end{align*}
\end{Cor}

\begin{proof}[Proof of Theorem \ref{thm-SPECTRUM}]
Our approach is reminiscent of the one of \cite{degond1986spectral} for the Vlasov-Poisson system.

\medskip

\noindent \textbf{Step 1.} 
Let $\lambda \in \C$ such that $\mathrm{Re}(\lambda) \neq 0$ with $D_\zeta(\lambda) \neq 0$ or $\lambda=i \sigma \in i\R$ with $\vert \sigma \vert> 1$ (in this latter case, by Proposition \ref{Prop-instalinear-inviscid}, $D_\zeta(\lambda) \neq 0$ as well). First, we study the resolvent equation
\begin{align*}
    (\lambda-A_\zeta)f=H,
\end{align*}
for some given $H \in L^2_\theta$, which reads
\begin{align*}
    \left(\lambda + i\sin(\theta) \right)f -i\zeta  \rho=H, \qquad \rho=\int_0^{2\pi} f(\theta) \, \mathrm{d}\theta.
\end{align*}
Since $\lambda \notin i\mathrm{Ran}(\sin)=i [-1, 1]$, we can divide and it is equivalent to
\begin{align*}
    f=i\zeta\frac{\sin(\theta)}{\lambda+i\sin(\theta)}\rho +\frac{H}{\lambda+i\sin(\theta)}
\end{align*}
By integrating in $\theta$, we get
\begin{align*}
    D_\zeta(\lambda)\rho=\int_0^{2\pi} \frac{H(\theta')}{\lambda+i\sin(\theta')} \, \mathrm{d}\theta'
\end{align*}
where $D_\zeta(\lambda)$ is the dispersion relation. Hence, for such $\lambda$, we finally obtain thanks to the two previous relations: 
\begin{align*}
f(\theta)= i\zeta D_\zeta(\lambda)^{-1}\left(\int_0^{2\pi} \frac{H(\theta')}{\lambda+i\sin(\theta')} \, \mathrm{d}\theta'\right) \frac{ \sin(\theta)}{\lambda+i\sin(\theta)}
+\frac{H(\theta)}{\lambda+i\sin(\theta)},
\end{align*}
which is in $L^2_\theta$. Conversely, the former formula defines an $L^2_{\theta}$ function satisfying the resolvent equation with source $H$, when $\mathrm{Re}(\lambda) \neq 0$ with $D_\zeta(\lambda) \neq 0$, or $\lambda=i \tau \in i\R$ with $\vert \tau \vert> 1$. Furthermore, we have a bound of the form
    \begin{align*}
\Vert f \Vert_{L^2_{\theta}} \lesssim_{\lambda}  \Vert H \Vert_{L^2_{\theta}}.
\end{align*}
All in all, we have proven that we can define the resolvent operator (given by the former formula, that is $f=R(\lambda)H$) when  $\mathrm{Re}(\lambda) \neq 0$ with $D_\zeta(\lambda) \neq 0$ or $\lambda=i \sigma \in i\R$ with $\vert \sigma \vert> 1$, that is
\begin{align*}
        \mathrm{Sp}(A_{\zeta})\subset \big\lbrace i\sigma \mid \sigma \in [-1, 1] \big\rbrace \cup \lbrace \lambda \in \C \setminus i \R \mid D_{\zeta}(\lambda)=0 \rbrace.
    \end{align*}
In particular, the formula \eqref{eq:formula-RESOLVENT} holds and the formula \eqref{eq:formula-RESOLVENT-average} is then deduced by integrating in $\theta$.

\medskip

\noindent \textbf{Step 2.} 
Let us show the reverse inclusion. If $D_\zeta(\lambda)=0$ for some $\lambda \in \C \setminus i \R$, we have already proven in Proposition \ref{Prop-iifcond-growingmode} that there exists a non-trivial $f$ such that $A_{\zeta} f=\lambda f$ therefore $\lambda \in \mathrm{Sp}(A_\zeta)$. If $\lambda= i\sigma $ with $\sigma \in [-1,1]$, let us also show that $\lambda \in \mathrm{Sp}(A_\zeta)$ by constructing a Weyl sequence for $\lambda$, that is, a sequence of normalized $(\psi_n)$ such that $\Vert (\lambda-A_\zeta)\psi_n \Vert_{L^2_\theta} \rightarrow 0$ as $n \rightarrow + \infty$.
First, there exists $\theta_0 \in [0,2\pi]$ such that $\lambda+i \sin(\theta_0)=0$, as $\sigma \in \mathrm{Ran}(\sin)$. Furthermore, we have $\partial_\theta(\lambda+i \sin(\theta))_{\theta=\theta_0} \neq 0$. Now pick $\psi \in \mathscr{C}^\infty_c(-\frac{1}{2}, \frac{1}{2})$ such that $\Vert \psi \Vert_{L^2(\R)}= 1$. 
Let us define
\begin{align*}
    \psi_n\vcentcolon= n^{1/2} \psi (n(\theta-\theta_0)), \qquad n \in \N,
\end{align*}
such that, for $n$ large enough, we have  $\mathrm{supp}(\psi_n) \subset (\theta_0-1/2n, \theta_0+1/2n) \subset [0, 2\pi)$ (so that we extend it by periodicity) and $\Vert \psi_n \Vert_{L^2_\theta}=1$. 
Using the notation $\rho_n=\int_0^{2\pi} \psi_n \, \mathrm{d}\theta$, we directly compute by Hölder inequality
\begin{align*}
    \Vert (\lambda-A_\zeta)\psi_n \Vert_{L^2_\theta} &\lesssim \Vert (\lambda + i \sin(\cdot))\psi_n\Vert_{L^2_\theta}+ \Vert k \cdot e(\cdot) \rho_n\Vert_{L^2_\theta} \\
    & \lesssim \Vert  \lambda + i \sin(\cdot) \Vert_{L^\infty_\theta(\mathrm{supp}(\psi_n))}\Vert \psi_n \Vert_{L^2_\theta}   + \vert \mathrm{supp}(\psi_n)\vert^{1/2} \Vert \psi_n \Vert_{L^2_\theta} \\
    &\lesssim \frac{1}{n}+ \frac{1}{n^{1/2}},
\end{align*}
where we have used that $\vert \lambda + i \sin(\theta) \vert \lesssim \vert \theta-\theta_0\vert \leq n^{-1} $ on the support of $\psi_n$, by Taylor formula. It therefore yields the desired sequence, and concludes the proof.
\end{proof}

\section{Mixing estimates in the non-diffusive case}\label{Section-Landau-damping}
In this section, we continue the study of equation \eqref{eq:rescaled eq2-MODELA-inv} in order to capture the decay in time of $\rho=\int_0^{2\pi} f(\theta) \, \mathrm{d}\theta$. As highlighted in the introduction (see Remark \ref{rem_thm2} after Theorem \ref{thm-stability}), our approach aims at finding some decay mechanism through (phase) mixing. In the more classical context of linearized Vlasov equations \cites{MouhotVillani,CZHGV, Faou-Rousset1, Faou-Rousset2, DFGV18}, this decay is obtained by writing a Volterra equation on $\rho$ with integrable kernel. It is then well-known that spectral stability implies decay, see for instance Theorem \ref{thm-Paley-Wiener} in Appendix \ref{Appendix-Volterra}. In our context, the kernel would only decay like $O(1/\sqrt{t})$, which forbids to apply the classical theory of Volterra equation. We therefore proceed through more explicit calculations at the level of the Laplace transform.

\medskip

First  a direct energy estimate in $L^2_\theta$ on the equation \eqref{eq:rescaled eq2-MODELA-inv} satisfied by $f$ (mutlipliying by $\overline{f}$, integrating and taking the real part) combined with the fact that $\vert \rho(t) \vert \lesssim \Vert f(t) \Vert_{L^2_\theta}$ leads to 
\begin{align*}
    \vert \rho(t) \vert \lesssim \Vert f(t) \Vert_{L^2_\theta} \lesssim \e^{C_\zeta t} \|f^{\mathrm{in}}\|_{L^2_\theta}, \qquad C_\zeta>0,
\end{align*}
by Grönwall's lemma. This proves that the Laplace transforms of $f$ 
and of $\rho$:
\begin{align*}
    \mathcal{L}[f](\lambda,\theta)=\int_0^{+\infty} \e^{-\lambda t} f(t,\theta) \, \mathrm{d}t, \quad  \mathcal{L}[\rho](\lambda)=\int_0^{+\infty} \e^{-\lambda t} \rho(t) \, \mathrm{d}t
\end{align*}
are well-defined for $\lambda \in \lbrace \mathrm{Re}>C_\zeta\rbrace$.
By considering the Laplace transform of the equation satisfied by $f$, we obtain, as in the proof of Proposition \ref{Prop-iifcond-growingmode}:
\begin{align*}
    (\lambda+ i\sin(\theta))\mathcal{L}[f](\lambda, \theta)-i\zeta \sin(\theta)\mathcal{L}[\rho](\lambda) =f^{\mathrm{in}},
\end{align*}
hence
\begin{align*}
   \mathcal{L}[f](\lambda, \theta)= i\zeta\frac{\sin(\theta)}{\lambda+ i\sin(\theta)}\mathcal{L}[\rho](\lambda)+\frac{f^{\mathrm{in}}}{\lambda+ i\sin(\theta)}.
\end{align*}
Integrating in $\theta$, we get
\begin{align*}
    D_\zeta(\lambda) \mathcal{L}[\rho](\lambda)=F(\lambda), \qquad \mathrm{Re}(\lambda)> C_\zeta,
\end{align*}
where $D_\zeta(\lambda)$ is the dispersion relation already introduced in Definition \ref{def-dispersion-rel-inviscid} (that is defined for $\lambda$ in $\C \setminus i[-1, 1]$), and where
\begin{align*}
    F(\lambda)\vcentcolon=\int_0^{2\pi} \frac{f^{\mathrm{in}}(\theta)}{\lambda+i \sin(\theta)} \, \mathrm{d}\theta, \qquad \lambda \in \C \setminus i[-1, 1].
\end{align*}
We know from Proposition \ref{Prop-instalinear-inviscid} that in the stable case $\zeta<1/2\pi$, the function $D_{\zeta}$ has no zero in $\C \setminus i[-1, 1]$.
Let us first collect some basic facts about the function $1/D_\zeta$.
\begin{Lem}\label{LM:prop-GreenKernel}
    If $\zeta<\frac{1}{2\pi}$, then we have the following:
    \begin{itemize}
        \item the function $\lambda \mapsto \frac{1}{D_\zeta\left(\lambda \right)}$ is well-defined and holomorphic on $\C \setminus i[-1, 1]$. Furthermore, there holds
        \begin{align}\label{eq:formula-1/D}
    \frac{1}{D_\zeta(\lambda)}=\frac{\sqrt{1+\frac{1}{\lambda^2}}}{(1-2\pi \zeta)\sqrt{1+\frac{1}{\lambda^2}}+2\pi \zeta}, \qquad \lambda \in \C \setminus i[-1, 1].
\end{align}
\item there holds
\begin{align}\label{eq:asymptotics-1/D}
    \frac{1}{D_\zeta(\lambda)}=1+\mathcal{O}(\vert \lambda \vert^{-2} ), \qquad \vert \lambda \vert \rightarrow + \infty.
\end{align}
    \end{itemize}
\end{Lem}
\begin{proof}
    For the first point, let us recall from Lemma \ref{LM:formula-dispersionD} that the following formula holds
for $\lambda \in \C \setminus i[-1, 1]$
\begin{align*}
D_\zeta(\lambda)=
\frac{(1-2\pi \zeta)\sqrt{1+\frac{1}{\lambda^2}}+2\pi \zeta}{\sqrt{1+\frac{1}{\lambda^2}}},
\end{align*}
and that the former expression does not vanish on $\C \setminus i[-1, 1]$ in view of Proposition \ref{Prop-instalinear-inviscid}, when $ \zeta<\frac{1}{2\pi}$. Furthermore, since $\lambda \mapsto \sqrt{1+\frac{1}{\lambda^2}}$ is holomorphic on $\C \setminus i[-1, 1]$, we infer the same for $D_\zeta$ and eventually for $1/D_\zeta$. The claimed formula follows. For the second point, we perform a Taylor expansion when $\vert \lambda \vert \rightarrow +\infty$ to obtain 
\begin{align*}
 \frac{1}{D_\zeta(\lambda)} 
 &= \frac{1+\frac{1}{2 \vert \lambda\vert^2}+\mathcal{O}(\vert \lambda\vert^{-4})}{2\pi \zeta+ (1-2\pi \zeta) (1+\frac{1}{2\vert \lambda \vert ^2}+\mathcal{O}(\vert \lambda\vert^{-4}))}
 = \frac{1+\frac{1}{2\vert \lambda\vert^2}+\mathcal{O}(\vert \lambda \vert^{-4})}{1+\frac{\frac{1}{2}-\pi \zeta} {\vert \lambda \vert^2}+\mathcal{O}(\vert \lambda\vert ^{-4})}
\end{align*}
and therefore
\begin{align*}
    \frac{1}{D_\zeta(\lambda)}=1+\mathcal{O}(\vert \lambda \vert^{-2} ), \qquad \vert \lambda \vert \rightarrow + \infty.
\end{align*}
This concludes the proof.
\end{proof}
We write the relation $D_\zeta(\lambda) \mathcal{L}[\rho](\lambda)=F(\lambda)$ under the form 
\begin{align} \label{Laplace_Volterra}
    \mathcal{L}[\rho](\lambda)=F(\lambda) + \left( \frac{1}{D_\zeta(\lambda)} - 1 \right) F(\lambda). 
\end{align}
From Lemma \ref{LM:prop-GreenKernel}, we know that $\left( \frac{1}{D_\zeta(\lambda)} - 1 \right) = O(|\lambda|^{-2})$ at infinity. It is also easily seen that $F(\lambda) = O(|\lambda|^{-1})$ at infinity. It follows that the inverse Laplace transforms 
\begin{align}\label{def:Green-Kernel1}
    G(t)\vcentcolon=\frac{1}{2\pi i}\int_{\Re(\lambda)=r_0}
\e^{\lambda t}\,\left( \frac{1}{D_\zeta(\lambda)} - 1 \right) \,\dd\lambda,
\end{align}
and 
$$ S(t) \vcentcolon=  \frac{1}{2\pi i}\int_{\Re(\lambda)=r_0}
\e^{\lambda t}\, F(\lambda) \,\dd\lambda,   $$
are well-defined for any $r_0  > C_\zeta$ and this definition does not depend on  $r_0>C_\zeta$ by standard holomorphy property. Note that the integral defining $F$ is only semi-convergent, and could be expressed in terms of the inverse Fourier transform of the $L^2$ function $\xi \rightarrow F(r_0+i\xi)$. Moreover, relation \eqref{Laplace_Volterra} yields 
\begin{align}\label{convol-LandauDamping1}
\rho(t)
= S(t) + \int_0^t G(t-\tau) S(\tau) \, \mathrm{d}\tau.
\end{align} 
The point is now to get some information on the decay of the source $S$ and of the Green kernel $G$. First, $F$ is easily recognized to be  the integral in $\theta$ of the Laplace transform of  $t \rightarrow \e^{- i \sin(\theta)t} f^{\mathrm{in}}(\theta)$ (solution of free transport), so that its inverse Laplace transform is 
$$ S(t)=\int_0^{2\pi}  \e^{- i \sin(\theta)t} f^{\mathrm{in}}(\theta)\mathrm{d}\theta.$$ 
We recognize a simple oscillatory integral, from which the decay of $S(t)$ can be easily inferred thanks to the standard one–dimensional stationary phase principle (see e.g. \cite{SteinShakarchi}*{Chapter 8}): the phase $\theta \mapsto t\sin(\theta)$ has indeed two critical points at $\theta=\pi/2$ and $\theta=3\pi/2$, which are non-degenerate.  The polynomial decay in time of $S(t)$ is recorded in the following simple lemma. 
\begin{Lem}\label{LM:inviscid-decay-FREETRANSPORT}
There exists $C>0$ such that there holds
   \begin{align*}
    \vert S(t) \vert \leq \frac{C}{\sqrt{1+ t}} \Vert f^{\mathrm{in}} \Vert_{H^1_\theta}, \qquad t>0.
\end{align*} 
\end{Lem}
It therefore remains to understand the decay in time of the Green kernel $G(t)$ defined in \eqref{def:Green-Kernel1}. We claim that is enough to prove that, if $\zeta<\frac{1}{2\pi}$, then
\begin{align}\label{eq:target-decayGreen}
    \left\vert G(t) \right\vert \lesssim \frac{1}{(1+ t)^{3/2}}, \qquad t >0,
\end{align}
to conclude the proof of Theorem \ref{thm-stability-reduced} when $\nu =0$. Indeed, if this holds, we can combine it with the decay estimate for $S(t)$ provided by Lemma \ref{LM:inviscid-decay-FREETRANSPORT} and get for all $t>0$
\begin{align*}
    \left\vert \int_0^t G(t-\tau) S(\tau) \, \mathrm{d}\tau\right\vert \lesssim \Vert \rho^{\mathrm{in}} \Vert_{H^1_\theta} \int_0^t \frac{1}{(1+(t-\tau))^{3/2}}\frac{1}{(1+ \tau)^{1/2}} \, \mathrm{d}\tau.
\end{align*}
By the following splitting, we can infer
\begin{align*}
    &\int_0^t \frac{1}{(1+(t-\tau))^{3/2}}\frac{1}{(1+ \tau)^{1/2}} \, \mathrm{d}\tau \\
    & \qquad\leq  \int_0^{t/2} (1+t-\tau)^{-3/2} (1+\tau)^{-1/2} \mathrm{d}s+\int_{t/2}^{t} (1+t-\tau)^{-3/2} (1+\tau)^{-1/2} \mathrm{d}\tau \\
    &\qquad \lesssim (1+t)^{-3/2}\int_0^{t/2} (1+\tau)^{-1/2} \mathrm{d}\tau+(1+t)^{-1/2}\int_{t/2}^{t} (1+t-\tau)^{-3/2} \mathrm{d}\tau \\
    &\qquad\lesssim (1+t)^{-3/2} (1+t)^{1/2} +(1+t)^{-1/2}\int_{0}^{+\infty} (1+r)^{-3/2} \mathrm{d}r \\
    & \qquad\lesssim (1+ t)^{-1/2}.
\end{align*}
All in all, thanks to \eqref{convol-LandauDamping1}, this would prove that 
\begin{align*}
    \left\vert \rho(t) \right\vert \lesssim \frac{\Vert f^{\mathrm{in}} \Vert_{H^1_\theta}}{(1+ t)^{1/2}}, \qquad t >0,
\end{align*}
and therefore would achieve the proof of Theorem \ref{thm-stability-reduced} in the case $\nu=0$.

\medskip

We will now focus on obtaining the decay estimate \eqref{eq:target-decayGreen} for $G(t)$. It relies on an explicit formula, obtained by deformation of the contour of integration. 
\begin{Lem}\label{LM:formulaGreen}
     If $\zeta < \frac{1}{2\pi}$ then for all $t>0$, $G(t)$ is given by the formula \eqref{eq:formulaGreen}. 
\end{Lem}
\begin{proof}
Our proof is in the spirit of the so-called limiting absorption principle (see e.g. \cites{Jia, CZN1, CZN2, HadzicSchrecker}): the integrand in \eqref{def:Green-Kernel1}  being quite explicit in our case, we will provide a direct treatment by an explicit contour deformation.

Fix $t>0$ and $r_0>0$. For $R>0$ large enough and $\varepsilon>0$ small enough, we consider two contours (see Figure \ref{Figure-contours} below): an exterior closed contour $Q_{r_0, R}$ (oriented counterclockwise) defined as the union of
 \begin{itemize}
     \item a vertical segment $V_{r_0, R}=\{r_0+i\tau:|\tau|\le R\}$, which, for $R \rightarrow + \infty$, yields the so-called Bromwich contour for the inverse Laplace transform defining $G(t)$;
     \item  two horizontal segments from $\gamma_0\pm iR$ to $-R\pm iR$, denoted as $\Gamma_{\text{up},R}$ and $\Gamma_{\text{down},R}$;
     \item a left vertical segment $\Gamma_{\text{left},R}=\{-R+i\tau:|\tau|\le R\}$;
 \end{itemize}
 and an inner closed contour $\mathscr{C}_{\eps}$ (oriented counterclockwise) defined as the union of
 \begin{itemize}
     \item  two horizontal segments from $\pm \eps - i$ to $\pm \eps +i$, denoted as $\gamma_{\eps}^+$ and $\gamma_{\eps}^-$;
     \item two caps around $\pm i$ of radius $\eps$, denoted as $c_{\mathrm{up},\eps}$ and $c_{\mathrm{down},\eps}$ 
 \end{itemize}
 
 \begin{figure}[h]
 \centering
 \begin{tikzpicture}[scale=1.2]

  \tikzset{
    ->-/.style={decoration={markings,
                 mark=at position 0.5 with {\arrow{>}}},
                 postaction={decorate}}
  }

  \def\R{3.3}       
  \def\a{1.8}       
  \def\eps{1}     

  \draw[->] (-4,0) -- (4,0) node[anchor=west] {$\Re(\lambda)$};
  \draw[->] (0,-4) -- (0,4) node[anchor=south] {$\Im(\lambda)$};

  \draw[dashed] (2.5,-4) -- (2.5,4);
  \draw[dashed] (-2.5,-4) -- (-2.5,4);

  \node[above] at (2.8,0) {$r_0$};
  \node[below] at (-2.9,0) {$-R$};

  \node[right]  at (2.5,2.6) {$V_{r_0,R}$};
  \node[left]   at (-2.5,2.6) {$\Gamma_{\text{left},R}$};
  \node[left]   at (-0.7,3.55) {$\Gamma_{\text{up},R}$};
  \node[left]   at (-0.6,-3.6) {$\Gamma_{\text{down},R}$};

  \draw[very thick,red] (0,-\a) -- (0,\a);
  \filldraw[red] (0,\a) circle (1.3pt);
  \filldraw[red] (0,-\a) circle (1.3pt);

  \node[below,red] at (-0.4,\a) { $i$};
  \node[above,red] at (-0.45,-\a) { $-i$};

  \node[right] at (0,\R+0.3) {$iR$};
  \node[right] at (0,-\R-0.2) {$-iR$};

  \coordinate (GtopR) at (2.5,\R);
  \coordinate (GbotR) at (2.5,-\R);
  \coordinate (GtopL) at (-2.5,\R);
  \coordinate (GbotL) at (-2.5,-\R);

  \draw[very thick,->-] (GbotR) -- (GtopR);
  \filldraw (GbotR) circle (1pt);
  \filldraw (GtopR) circle (1pt);

  \draw[very thick,->-] (GtopR) -- (GtopL);
  \filldraw (GtopL) circle (1pt);

  \draw[very thick,->-] (GtopL) -- (GbotL);
  \filldraw (GbotL) circle (1pt);

  \draw[very thick,->-] (GbotL) -- (GbotR);

  \node[below] at (\eps+0.2,0) {$\varepsilon$};
  \node[below] at (-\eps-0.3,0) {$-\varepsilon$};

  \coordinate (EtopR) at (\eps,\a);
  \coordinate (EbotR) at (\eps,-\a);
  \coordinate (EtopL) at (-\eps,\a);
  \coordinate (EbotL) at (-\eps,-\a);

  \node[right, blue]  at (0.95,\a-0.6) {$\gamma_{\varepsilon}^+$};
  \node[right, blue]  at (-1.55,-\a+0.7) {$\gamma_{\varepsilon}^-$};
    \node[left, blue]  at (0,\a+1.15) {$c_{\mathrm{up},\varepsilon}$};
    \node[left, blue]  at (0,-\a-1.1) {$c_{\mathrm{down},\varepsilon}$};

  \draw[thick,blue,->-] (EbotR) -- (EtopR);
  \draw[thick,blue,->-] (EtopL) -- (EbotL);

  \filldraw[blue] (EtopR) circle (1pt);
  \filldraw[blue] (EbotR) circle (1pt);
  \filldraw[blue] (EtopL) circle (1pt);
  \filldraw[blue] (EbotL) circle (1pt);

  \draw[blue,thick,->-]
    (\eps,\a) arc [start angle=0, end angle=180, radius=\eps];

  \draw[blue,thick,->-]
    (-\eps,-\a) arc [start angle=180, end angle=360, radius=\eps];

  \draw[dashed,blue] (0,\a) -- (\eps,\a);
  \draw[dashed,blue] (0,\a) -- (-\eps,\a);
  \draw[dashed,blue] (0,-\a) -- (\eps,-\a);
  \draw[dashed,blue] (0,-\a) -- (-\eps,-\a);

  \coordinate (Rad45End) at ({\eps/sqrt(2)}, {\a + \eps/sqrt(2)});
  \draw[dashed,blue] (0,\a) -- (Rad45End);
  \node[above] at (Rad45End) {$\varepsilon$};

  \coordinate (Rad45End) at ({\eps/sqrt(2)}, {-\a - \eps/sqrt(2)});
  \draw[dashed,blue] (0,-\a) -- (Rad45End);
  \node[below] at (Rad45End) {$\varepsilon$};

 \end{tikzpicture}
 \caption{Contours $Q_{r_0, R}$ and $\mathscr{C}_\eps$}
\label{Figure-contours}
 \end{figure}

Let $\mathcal{R}_\zeta(\lambda) \vcentcolon= \frac{1}{D_\zeta(\lambda)} - 1$.   In the "shell" delimited by $Q_{r_0, R}$ and $\mathscr{C}_\eps$, the integrand $\lambda \mapsto \e^{\lambda t} \mathcal{R}_\zeta(\lambda)$ defining $ G(t)$ is holomorphic (because we avoid the cut $i[-1,1]$), so by Cauchy's formula we get 
 \begin{align}\label{eq:Cauchy-contour-gen}
     \int_{Q_{r_0,R}} \e^{\lambda t}\mathcal{R}_\zeta(\lambda) \,\mathrm{d}\lambda= \int_{\mathscr{C}_\eps}\e^{\lambda t}\mathcal{R}_\zeta(\lambda) \,\mathrm{d}\lambda.
 \end{align}
 The goal is to take the limit as $R \rightarrow + \infty$ on the l.h.s, while taking the limit $\eps \rightarrow 0$ on the r.h.s to obtain the jump through the cut, and then the desired formula.

\medskip

\noindent \textbf{Step 1.} 
The following holds for any $t>0$ and $\eps>0$
\begin{align}\label{eq:Step1-Contour}
    G(t)= \lim_{R \rightarrow + \infty} \int_{Q_{r_0,R}} \e^{\lambda t} \mathcal{R}_\zeta(\lambda) \,\mathrm{d}\lambda= \int_{\mathscr{C}_\eps}\e^{\lambda t}\mathcal{R}_\zeta(\lambda) \,\mathrm{d}\lambda.
\end{align}
Let us prove this fact. The second equality follows directly from \eqref{eq:Cauchy-contour-gen} since the integral on the right-hand side does not depend on $R$. To obtain the first equality, we use the decomposition 
\begin{align*}
    \int_{Q_{r_0,R}} \e^{\lambda t} \mathcal{R}_\zeta(\lambda) \,\mathrm{d}\lambda=\left\lbrace \int_{V_{r_0,R}}+\int_{\Gamma_{\text{up},R}}+\int_{\Gamma_{\text{left},R}}+\int_{\Gamma_{\text{down},R}} \right\rbrace \e^{\lambda t} \mathcal{R}_\zeta(\lambda) \,\mathrm{d}\lambda.
\end{align*}
and we need to prove that the contribution of the three last terms vanishes as $R\rightarrow +\infty$, because we already know that the first one yields by definition
\begin{align*}
    \int_{V_{r_0,R}} \e^{\lambda t} \mathcal{R}_\zeta(\lambda) \,\mathrm{d}\lambda \underset{R \rightarrow +\infty}{\longrightarrow} G(t).
\end{align*}
The role of $\Gamma_{\text{up},R}$ and $\Gamma_{\text{down},R}$ being symmetric, let us only detail the case of $\Gamma_{\text{down},R}$. We have for $R$ large enough and in view of the asymptotics \eqref{eq:asymptotics-1/D}
\begin{align*}
    \left\vert \int_{\Gamma_{\text{down},R}} \e^{\lambda t} \mathcal{R}_\zeta(\lambda) \,\mathrm{d}\lambda\right\vert = \left\vert \int_{-R}^{r_0} \e^{(x-iR) t} \mathcal{R}_\zeta(x-iR) \, \mathrm{d}x \right\vert \lesssim  \int_{-R}^{r_0} \e^{x t} \frac{1}{|x-iR|^2} \, \mathrm{d}x   \lesssim \frac{1}{t R^2} \e^{t r_0} \underset{R \rightarrow +\infty}{\longrightarrow} 0.
\end{align*}
For the last piece from $\Gamma_{\text{left},R}$, we also have for $R$ large enough
\begin{align*}
    \left\vert \int_{\Gamma_{\text{left},R}} \e^{\lambda t} \mathcal{R}_\zeta(\lambda) \,\mathrm{d}\lambda\right\vert = \left\vert \int_{-R}^{R} \e^{(-R+iy) t} \mathcal{R}_\zeta(-R-iy) \, \mathrm{d}y \right\vert \lesssim \e^{-Rt} \int_{-\infty}^{+\infty} \frac{1}{1+y^2} \, \mathrm{d}y \underset{R \rightarrow +\infty}{\longrightarrow} 0.
\end{align*}
We therefore obtain \eqref{eq:Step1-Contour}.
\medskip

\noindent \textbf{Step 2.} 
The following holds for any $t>0$
\begin{align}\label{eq:Step2-Contour}
    G(t)=i\int_{-1}^{1} \lim_{\eps \rightarrow 0^+} \left( \e^{(i\tau+\eps)t}\frac{1}{D_\zeta(i\tau+\eps)}-\e^{(i\tau-\eps)t}\frac{1}{D_\zeta(i\tau-\eps)} \, \mathrm{d}\tau \right) \, \mathrm{d}\tau
\end{align}
In view of \eqref{eq:Step1-Contour}, it is enough to pass to the limit when $\eps \rightarrow 0$ in the right-hand side of this identity to obtain \eqref{eq:Step2-Contour}. By using first Cauchy's formula, we have 
\begin{align*}
    \int_{\mathscr{C}_\eps} \e^{\lambda t} \mathcal{R}_\zeta(\lambda) \,\mathrm{d}\lambda=\int_{\mathscr{C}_\eps} \e^{\lambda t} \frac{1}{D_\zeta(\lambda)} \,\mathrm{d}\lambda =\left\lbrace \int_{\gamma_\eps^+}+\int_{\gamma_\eps^-}+\int_{c_{\mathrm{up},\eps}}+\int_{c_{\mathrm{down},\eps}} \right\rbrace \e^{\lambda t} \frac{1}{D_\zeta(\lambda)}\,\mathrm{d}\lambda,
\end{align*}
and we will prove that the contribution of the two last terms vanishes when $\eps \rightarrow 0$, while the two first ones will give the desired jump through the cut $i[-1,1]$ as $\eps \rightarrow 0$. To this end, let us recall the formula 
\begin{align*}
    \frac{1}{D_\zeta(\lambda)}=\frac{\sqrt{1+\frac{1}{\lambda^2}}}{(1-2\pi \zeta)\sqrt{1+\frac{1}{\lambda^2}}+2\pi \zeta}, \qquad \lambda \in \C \setminus i[-1, 1],
\end{align*}
from \eqref{eq:formula-1/D}. First, if $\lambda = i\tau \pm \eps$, with $0 < \eps \ll 1$ and fixed $\tau \neq 0$, we get 
\begin{align*}
    1 + \frac{1}{(i\tau \pm \eps)^2} = 1 - \frac{1}{\tau^2} \left( 1 \pm \frac{2i\eps}{\tau} + O(\eps^2)\right),  
\end{align*}
from which we infer that 
\begin{align}\label{pikachu}
 \lim_{\eps \rightarrow 0^+}   \sqrt{1 + \frac{1}{(i\tau \pm \eps)^2}} & =  \left\{
    \begin{array}{ll}
        \sqrt{1 - \frac{1}{\tau^2}} & \quad \text{if } |\tau| \geq 1, \\
        \mp i \text{sign($\tau$)} \sqrt{\frac{1}{\tau^2} - 1} &  \quad \text{if } 0<|\tau|  \leq    1.
    \end{array}
\right.
\end{align}
In particular, it proves that the function $\lambda \mapsto \frac{1}{D_\zeta(\lambda)}$ is bounded by a constant $M_{\eps_0}>0$ on the half-disks delimited by $c_{\mathrm{up},\eps}$ and $c_{\mathrm{down},\eps}$ if $\eps<\eps_0$ is small enough. Hence, since we have 
\begin{align*}
    \int_{c_{\mathrm{up},\eps}} \e^{\lambda t} \frac{1}{D_\zeta(\lambda)} \,\mathrm{d}\lambda=\eps \int_0^{\pi} \e^{(i+\eps \e^{i\theta}) t} \frac{1}{D_\zeta(i+\eps \e^{i\theta})} \, \mathrm{d}\theta,
\end{align*}
we directly deduce that
\begin{align*}
    \int_{c_{\mathrm{up},\eps}} \e^{\lambda t} \frac{1}{D_\zeta(\lambda)} \,\mathrm{d}\lambda \underset{\eps \rightarrow 0}{\longrightarrow} 0.
\end{align*}
The same holds for the term with $c_{\mathrm{down},\eps}$. It remains to treat the jump through the cut. To do so, we write
\begin{align*}
    \left\lbrace \int_{\gamma_\eps^+}+\int_{\gamma_\eps^-} \right\rbrace \e^{\lambda t} \frac{1}{D_\zeta(\lambda)} \,\mathrm{d}\lambda  =i\int_{-1}^{1} \left( \e^{(i\tau+\eps)t}\frac{1}{D_\zeta(i\tau+\eps)}-\e^{(i\tau-\eps)t}\frac{1}{D_\zeta(i\tau-\eps)} \, \right) \, \mathrm{d}\tau.
\end{align*}
The goal is therefore to justify that one can interchange the integral and the limit when $\eps \rightarrow 0$, which will yield \eqref{eq:Step2-Contour}. We claim that there exists $\eps_0>0$ and $C_{\eps_0}>0$ such that for all $\eps \in (0,\eps_0)$, for all $\tau \in [-1, 1]$
\begin{align}\label{eq:Bound-DCT}
    \left\vert \frac{1}{D_\zeta(i \tau \pm\eps)} \right\vert \leq C_{\eps_0}.
\end{align}
To ease readability, we postpone the proof of this bound in Appendix \ref{Appendix-proofBOUND-DCT}. This is enough to apply the dominated convergence theorem when $\eps \rightarrow 0$, which yields \eqref{eq:Step2-Contour}. 

\medskip

\noindent \textbf{Step 3.} 
Let us eventually prove that the formula \eqref{eq:formulaGreen} holds true. By using the identity \eqref{eq:Step2-Contour} together with \eqref{pikachu}, we get
\begin{align*}
    G(t)
    &=i\int_{-1}^{1} \lim_{\eps \rightarrow 0^+} \left( \e^{(i\tau+\eps)t}\frac{1}{D_\zeta(i\tau+\eps)}-\e^{(i\tau-\eps)t}\frac{1}{D_\zeta(i\tau-\eps)} \, \mathrm{d}\tau \right) \, \mathrm{d}\tau \\
&=i \int_{-1}^{1} \e^{i \tau t} \left( \frac{-i \text{sign($\tau$)} \sqrt{\frac{1}{\tau^2} - 1}}{2\pi \zeta - (1-2\pi \zeta) i \text{sign($\tau$)}\sqrt{\frac{1}{\tau^2} - 1}} -  \frac{ i \text{sign($\tau$)} \sqrt{\frac{1}{\tau^2} - 1}}{2\pi \zeta + (1-2\pi \zeta) i \text{sign($\tau$)} \sqrt{\frac{1}{\tau^2} - 1}}  \right) \dd\tau.
\end{align*}
By simplifying the expression, we end up with
\begin{align*}
    G(t) & =  \int_{-1}^1 \e^{i s  t}  \sqrt{\frac{1}{s^2}-1} \, \frac{4 \pi \zeta \text{sign($s$)}}{(2\pi \zeta)^2 + (1-2\pi \zeta)^2 \left( \frac{1}{s^2}-1 \right)} \, \mathrm{d}s \\ 
    & = 4\pi \zeta \int_{-1}^1 \e^{i s  t} \sqrt{1-s^2} \frac{s}{(4\pi \zeta-1)s^2 + (1-2\pi \zeta)^2} \, \mathrm{d}s.
\end{align*}
This gives the desired formula for $G(t)$ and concludes the proof.
\end{proof}

We are now in position to prove the desired decay estimate \eqref{eq:target-decayGreen} for the Green's function $G(t)$.

\begin{proof}[Proof of \eqref{eq:target-decayGreen}]
    To prove the claim, we use the formula \eqref{eq:formulaGreen} from Lemma \ref{LM:formulaGreen} and decompose 
$$ G(t) = 4\pi \zeta \big( \mathcal{G}_{+}(t) - \mathcal{G}_{-}(t)\big),$$
with 
$$\mathcal{G}_{+}(t) \vcentcolon= \int_{0}^1 \e^{i s  t}  \frac{\sqrt{1-s^2} s}{(4\pi \zeta-1)s^2 + (1-2\pi \zeta)^2} \, \mathrm{d}s, $$
and $\mathcal{G}_{-}(t) \vcentcolon= \overline{\mathcal{G}_{+}(t)}$. To ease readability, we set
\begin{align*}
    \alpha_\zeta\vcentcolon=4\pi \zeta-1, \qquad \beta_\zeta\vcentcolon=(1-2\pi \zeta)^2,
\end{align*}
so that we have
$$\mathcal{G}_{+}(t)= \int_{0}^1 \e^{i s  t}  \frac{\sqrt{1-s^2} s}{\alpha_\zeta s^2 + \beta_\zeta} \, \mathrm{d}s, \qquad t>0.$$
We claim that the expression $s \mapsto Q_\zeta(s)=(\alpha_\zeta s^2 + \beta_\zeta)^{-1}$ is smooth on $[0,1]$, because the denominator does not vanish on this interval. If $\zeta \ge \tfrac{1}{4\pi}$, then $\alpha_\zeta \ge 0$, and for all $s\in[0,1]$ we have $\alpha_\zeta s^2 + \beta_\zeta \geq \beta_\zeta > 0$. If $0<\zeta<\tfrac{1}{4\pi}$, then $\alpha_\zeta < 0$ so for $s\in[0,1]$ we have $\alpha_\zeta s^2 + \beta_\zeta
        \geq \alpha_\zeta + \beta_\zeta.$
A direct computation gives
    \[
        \alpha_\zeta + \beta_\zeta
        = (4\pi\zeta - 1) + (1 - 4\pi\zeta + 4\pi^2\zeta^2)
        = 4\pi^2\zeta^2 > 0.
    \]
    Therefore $\alpha_\zeta s^2 + \beta_\zeta > 0$ for all $s\in[0,1]$ also in this case.

\medskip

We now need to distinguish between $0<  t\leq 1$ and $  t>1$.

\medskip

\noindent \textbf{Case $ t \leq 1$.} By the smoothness of the integrand on $[0,1$], we directly obtain the crude bound
\begin{align*}
  \vert \mathcal{G}_{+}(t) \vert \lesssim 1.
\end{align*}
Since $t \leq 1$, this yields the desired estimate 
\begin{align*}
  \vert \mathcal{G}_{+}(t) \vert \lesssim \frac{1}{(1+ t)^{3/2}}, \qquad t \le 1 .
\end{align*}
\noindent \textbf{Case $ t > 1$.}
It remains to prove the decay of $\mathcal{G}_{+}(t)$. By integrating by parts once, we find 
\begin{align*}
    \mathcal{G}_{+}(t) 
    &= \frac{1}{it} \int_0^1 \e^{i s t} \partial_s \left(  \frac{\sqrt{1-s^2} s}{\alpha_\zeta s^2 + \beta_\zeta} \right) \, \mathrm{d}s\\
 &=\frac{1}{it} \int_0^1 \e^{i s t} \left( \partial_s \sqrt{1-s^2} \right)   \frac{s}{\alpha_\zeta s^2 + \beta_\zeta} \, \mathrm{d}s + \frac{1}{it} \int_0^1 \e^{i s t}  \sqrt{1-s^2} \partial_s \left(  \frac{s}{\alpha_\zeta s^2 + \beta_\zeta}\right) \, \mathrm{d}s \\
 &= \mathrm{I}(t)+ \mathrm{II}(t).
\end{align*}
since the integrand vanishes at the endpoints of the interval.  The only term that prevents from integrating by parts once more is the first one, that is
\begin{align*}
 \mathrm{I}(t)=\frac{1}{it}\int_0^1 \e^{i s  t} \left( \partial_s \sqrt{1-s^2} \right)   \frac{s}{\alpha_\zeta s^2 + \beta_\zeta}  \, \mathrm{d}s =  \frac{1}{it}\int_0^1 \e^{i s t}  \frac{-1}{\sqrt{(1+s)(1-s) }} \frac{s^2}{\alpha_\zeta s^2 + \beta_\zeta}  \, \mathrm{d}s 
\end{align*}
because of the singularity near $s=1$. For $\delta \in (0,1)$ being a free parameter, we then write 
$$ \mathrm{I}(t) = \frac{1}{it}\left\lbrace \int_0^{1-\delta} + \int_{1-\delta}^1  \right\rbrace \e^{i s t}  \frac{-1}{\sqrt{(1+s)(1-s) }} \frac{s^2}{\alpha_\zeta s^2 + \beta_\zeta}  \, \mathrm{d}s =\mathrm{I}_\delta^{\mathrm{reg}}(t)+\mathrm{I}_\delta^{\mathrm{sing}}(t) $$
and optimize in $\delta$. More precisely, using the bound
\begin{align*}
    \left\vert \frac{1}{\sqrt{(1+s)(1-s)}} \frac{s^2}{\alpha_\zeta s^2 + \beta_\zeta}  \right\vert \lesssim \frac{1}{\sqrt{1-s}}, \qquad s \in [0,1],
\end{align*}
we have 
\begin{align*}
    \vert \mathrm{I}_\delta^{\mathrm{sing}}(t) \vert \lesssim \frac{1}{t}\left|\int_{1-\delta}^1 \frac{1}{\sqrt{1-s}} \, \mathrm{d}s \right| \leq \frac{\delta^{1/2}}{t} .
\end{align*}
For $\mathrm{I}_\delta^{\mathrm{reg}}(t)$, we can safely integrate by parts and get
\begin{align*}
\left|\mathrm{I}_\delta^{\mathrm{reg}}(t)  \right| & \lesssim \frac{1}{t^2} \left( \delta^{-1/2}  \frac{(1-\delta)^2}{\sqrt{2-\delta} (\alpha_\zeta (1-\delta)^2+\beta_\zeta)}+\int_0^{1-\delta} \frac{1}{(1-s)^{3/2}} \, \mathrm{ds} \right.\\
& \left. \qquad \qquad \qquad \qquad \qquad \qquad \qquad \qquad \qquad \quad   + \delta^{-1/2} \int_0^1  \left|\partial_s \left(\frac{s^2}{\sqrt{1+s} (\alpha_\zeta s^2+\beta_\zeta)} \right)\, \right| \mathrm{d}s  \right) \\
&\lesssim \frac{\delta^{-1/2}}{ t^2}.
\end{align*}
All in all, we have obtained
\begin{align*}
    \vert \mathrm{I}(t) \vert \lesssim \frac{\delta^{-1/2}}{ t^2}+ \frac{\delta^{1/2}}{t}, \qquad  t>1.
\end{align*}
Choosing $\delta = 1/t$, it yields for $t>1$
$$\vert \mathrm{I}(t) \vert \lesssim \frac{1}{t^{3/2}}\lesssim \frac{1}{(1+ t)^{3/2}},$$ 
as desired. For the term $\mathrm{II}(t)$, the situation is easier: we can integrate by parts once again and, since the derivative of the amplitude is integrable on $[0,1]$, we directly obtain 
\begin{align*}
    \vert \mathrm{II}(t) \vert \lesssim \frac{1}{t^2} \lesssim \frac{1 }{(1+ t)^{3/2}},
\end{align*}
as before, since $ t >1$. Gathering the bounds for $\mathrm{I}(t)$ and $\mathrm{II}(t)$ together, we end up with
\begin{align*}
    \vert \mathcal{G}_{+}(t) \vert \lesssim \frac{1 }{(1+ t)^{3/2}}, \qquad  t>1.
\end{align*}
This concludes the proof.
\end{proof}

\section{Instability in presence of rotational diffusion}\label{Section-instab-diffusive}
In this section, we study how the linearized instability obtained in Section \ref{Section-spectral-inviscid} for $\nu=0$ (when $\zeta > \frac{1}{2\pi}$)  still holds when rotational diffusion $\nu$ is added, in the regime $0<\nu \ll 1$.

\medskip
We start with the linearized diffusive version of the equation, that is  \eqref{eq:rescaled eq2-MODELA}.
Roughly speaking, we are looking for the ansatz
\begin{align*}
    f(t,\theta)=\e^{\lambda t} \mathrm{f}(\theta), \qquad \mathrm{Re}(\lambda) > 0.
\end{align*}
where $\lambda=\lambda^\nu_\zeta$, $\mathrm{f} = \mathrm{f}^\nu_\zeta$. As in Section \ref{Section-spectral-inviscid},  it has to satisfy
\begin{align}\label{eq:Fourier-viscous}
     (\lambda +i\sin(\theta))\mathrm{f}=i\zeta\sin(\theta) \rho+\nu \partial_\theta^2 \mathrm{f}, \quad \rho = \int_0^{2\pi} \mathrm{f}(\theta) \, \mathrm{d}\theta.
 \end{align}
 We will see \eqref{eq:Fourier-viscous} as a non-singular perturbation problem for $\nu$ small. Rather than using a formal expansion in $\nu$, we will treat this problem by directly working at the level of the dispersion relations for both the non-diffusive and the diffusive problem.

\medskip

Let us now introduce the following diffusive (rescaled) linearized operator on the Hilbert space $L^2_\theta$: 
\begin{align*}
    \mathrm{L}_{\nu}\vcentcolon=i \sin(\theta) -\nu \partial_{ \theta}^2, \qquad \nu \geq 0
\end{align*}
with domain $D(\mathrm{L}_{\nu})=H^2_\theta$. We record the following lemma on its resolvent, which is a direct consequence of Lax-Milgram lemma and elliptic regularity, combined with the fact that the resolvent of an operator is analytic on the resolvent set.
\begin{Lem}\label{LM:resolventModelB}
    For any $\lambda \in \lbrace \mathrm{Re} > 0 \rbrace $ and $\nu>0$, the operator $\lambda+\mathrm{L}_{\nu}$ is invertible with $(\lambda+ \mathrm{L}_{\nu})^{-1}: L^2_\theta \mapsto L^2_\theta$.  Furthermore, the mapping $\lambda \mapsto (\lambda+ \mathrm{L}_{\nu})^{-1}$ is analytic on $\lbrace \mathrm{Re} > 0 \rbrace$.
\end{Lem}
Let us now proceed with our perturbative approach of \eqref{eq:Fourier-viscous}, always assuming that $\zeta>\frac{1}{2\pi}$:  for $\lambda \in \lbrace \mathrm{Re} > 0 \rbrace $, we rewrite it as
\begin{align*}
    \mathrm{f}(\theta)=i \zeta \rho (\lambda+\mathrm{L}_{\nu})^{-1}[\sin](\theta),
\end{align*}
and by integration in $\theta$, we get
\begin{align*}
    \rho \left(1- i\zeta \int_0^{2\pi} (\lambda+\mathrm{L}_{\nu})^{-1}[\sin](\theta) \, \mathrm{d}\theta \right)=0.
\end{align*}
We are thus led to define the 'diffusive' dispersion relation $D^\nu_\zeta =0$, where for all $\lambda \in \{ \Re > 0 \}$,
\begin{align}\label{def-diffusive-ispersion-rel}
    D^\nu_\zeta(\lambda)\vcentcolon=1- i\zeta \int_0^{2\pi} (\lambda+\mathrm{L}_{\nu})^{-1}[\sin](\theta) \, \mathrm{d}\theta.
 \end{align}
Let us also recall the 'non-diffusive' dispersion relation $D_\zeta = 0$ where from Definition \ref{def-dispersion-rel-inviscid}: 
$$D_\zeta(\lambda)\vcentcolon=1- i\zeta \int_0^{2\pi} (\lambda+\mathrm{L})^{-1}[\sin](\theta) \, \mathrm{d}\theta, \quad L = i \sin(\theta).$$
From Proposition \ref{Prop-instalinear-inviscid}, since $\zeta>\frac{1}{2\pi}$, we know that there exists a non-diffusive unstable eigenvalue $\lambda_\zeta > 0$, that is such that $D_\zeta(\lambda_\zeta)=0$. 

\medskip

We want to prove that for $\lambda$ in a neighborhood of $\lambda_\zeta$ in $\lbrace \mathrm{Re} > 0 \rbrace$, the functions $D_\zeta(\lambda)$ and $D^\nu_\zeta(\lambda)$ are sufficiently close  for small  $\nu$.  More precisely, we will show that for 
\begin{align*}
    r_0 \vcentcolon= \frac{\lambda_\zeta}{2}>0,
\end{align*}
we have the following.
\begin{Prop}\label{Prop-compa-reldispersion}
There exists $C(r_0)>0$ such that for any $\nu >0$, the following holds:
\begin{itemize}
\item the functions $\lambda \mapsto D^\nu_\zeta(\lambda)  $ and $\lambda \mapsto D_\zeta(\lambda) $ are analytic on the ball $B(\lambda_\zeta,r_0)$;
    \item for any $\lambda \in B(\lambda_\zeta,r_0)$, we have 
    \begin{align*}
    \vert D^\nu_\zeta(\lambda)-D_\zeta(\lambda) \vert \leq C(r_0) \nu.
\end{align*}
\end{itemize}
 \end{Prop}
Let us directly use this proposition to show that it implies the existence of an unstable eigenvalue for the diffusive problem \eqref{eq:Fourier-viscous}, that is a zero of $D^\nu_\zeta(\lambda)$. We follow here a strategy for instance implemented in \cites{DDGM, HG}. The function $\lambda \mapsto D_\zeta(\lambda)$ being analytic, its zeros are isolated. Since $D_\zeta(\lambda_\zeta)=0$, there exists $\delta \in (0,r_0)$ such that 
\begin{align*}
    \eps \vcentcolon= \underset{\lambda \in \partial B(\lambda_\zeta, \delta) }{\inf} \, \vert D_\zeta(\lambda) \vert>0.
\end{align*}
If $\nu>0$ is such that $C(r_0) \nu \leq \eps/2$, we can conclude by Proposition \ref{Prop-compa-reldispersion} combined with Rouché's theorem that $\lambda \mapsto D^\nu_\zeta(\lambda)$ has a zero $\lambda^\nu_\zeta$ in $B(\lambda_\zeta, \delta)$. 

\medskip

To understand how close to $\lambda_\zeta$ is the new eigenvalue $\lambda_\nu^\zeta$ (in terms of $\nu$), we actually need to refine the argument above. If $m \geq 1$ denotes the order of vanishing of the analytic function $\lambda \mapsto D_\zeta(\lambda)$ at $\lambda=\lambda_\zeta$, there exists $r_1>0$ and an analytic function $\lambda \mapsto \widetilde{D}_\zeta(\lambda)$ on $B(\lambda_\zeta,r_1)$ such that
\begin{align*}
    D_\zeta(\lambda)&=(\lambda-\lambda_\zeta)^m \widetilde{D}_\zeta(\lambda), \qquad \lambda \in B(\lambda_\zeta,r_1), \qquad \chi_1 \vcentcolon= \inf_{ B(\lambda_\zeta, r_1)}\vert \widetilde{D}_\zeta(\lambda) \vert >0.
\end{align*}
Furthermore, we can always suppose that $r_0 \geq r_1$. By Proposition \ref{Prop-compa-reldispersion}, we have
\begin{align*}
    \vert D^\nu_\zeta(\lambda)-D_\zeta(\lambda) \vert \leq C(r_0) \nu, \qquad \lambda \in B(\lambda_\zeta, r_1).
\end{align*}
It follows that for any $\delta > 0$ chosen so that the two constraints
\begin{align*}
    0<\delta<r_1, \qquad C(r_0)\nu < \delta^m \chi_1,
\end{align*}
are satisfied, we have
\begin{align*}
    \vert D^\nu_\zeta(\lambda)-D_\zeta(\lambda) \vert < \vert D_\zeta(\lambda) \vert , \qquad \lambda \in \partial B(\lambda_\zeta, \delta ),
\end{align*}
which by  Rouché's theorem applied with the circle $\partial B(\lambda_\zeta, \delta)$ yields a zero of $D^\nu_\zeta(\lambda)$ in $B(\lambda_\zeta, \delta)$. We choose
\begin{align*}
\delta_\nu\vcentcolon=\left(2\frac{C(r_0)\nu}{\chi_1} \right)^{\frac{1}{m}},
\end{align*}
and for $\nu>0$ small enough, the two previous conditions indeed hold, hence the existence of the zero in $B(\lambda_\zeta, \delta)$.

Let us now prove that we actually have $m=1$ in our case. Recall that by Lemma \ref{LM:formula-dispersionD}, we have the explicit formula 
\begin{align*}
D_\zeta(\lambda)=\frac{(1-2\pi \zeta)\sqrt{1+\frac{1}{\lambda^2}}+2\pi \zeta}{\sqrt{1+\frac{1}{\lambda^2}}}, \qquad \mathrm{Re}(\lambda)>0.
\end{align*}
As  $m$ is the vanishing order of the previous numerator (that is analytic on $\lbrace \mathrm{Re}>0 \rbrace$), it is enough to consider its restriction on the real positive axis. Since the function
\begin{align*}
d: \lambda \mapsto (1-2\pi \zeta)\sqrt{1+\frac{1}{\lambda^2}}+2\pi \zeta
\end{align*}
has its derivative equal to 
\begin{align*}
    d'(\lambda)= -(1-2\pi \zeta) \left(1+\frac{1}{\lambda^2} \right)^{-1/2}\lambda^{-3}, \qquad \lambda>0,
\end{align*}
and since this last expression is non-vanishing, we can conclude that $m=1$. With the former choice of of $\delta_\nu$, we have therefore proven that $\delta_\nu=c_0 \nu$ for some constant $c_0$ only depending on $D_\zeta$ and $\lambda_\zeta$.

\medskip

All in all the former reasoning shows that Proposition \ref{Prop-compa-reldispersion} entails the last part of Theorem \ref{thm-instability-reduced} (for $\nu>0$). We now provide a proof of Proposition \ref{Prop-compa-reldispersion}.
 \begin{proof}[Proof of Proposition \ref{Prop-compa-reldispersion}]
     
Let us prove the first assertion: in view of the definitions \eqref{def-diffusive-ispersion-rel} and \eqref{def-dispersion-rel-inviscid}, and the fact that the involved resolvent operators are analytic on $\lbrace \mathrm{Re} >0 \rbrace$, we easily deduce that $\lambda \mapsto D^\nu_\zeta(\lambda)  $ and $\lambda \mapsto D_\zeta(\lambda)$ are analytic.

Let us now prove the second assertion. First, we rely on a simple stability argument to compare $(\lambda+\mathrm{L}_{\nu})^{-1}[\sin(\cdot)]$ and $(\lambda+\mathrm{L})^{-1}[\sin(\cdot)]$ that define $D^\nu_\zeta(\lambda)$ and $D_\zeta(\lambda)$: for $\nu>0$ and $\lambda \in \lbrace \mathrm{Re} > 0 \rbrace$,  let $h_\nu$ and $h$ be respectively the solutions of 
  \begin{align*}
      (\lambda +i \sin(\theta)-\nu \partial_\theta^2 )h_\nu(\theta)= \sin(\theta),
  \end{align*}
  and \begin{align*}
      (\lambda +i \sin(\theta) )h(\theta)= \sin(\theta). 
  \end{align*}
  The first one exists in view of Lemma \ref{LM:resolventModelB}, which the second one is given explicitly by 
  \begin{align*}
      h(\theta)= \frac{\sin(\theta)}{\lambda +i \sin(\theta) }.
  \end{align*}
By introducing the difference $H_\nu=h_\nu-h$, we see it satisfies
 \begin{align*}
 (\lambda +i \sin(\theta))H_\nu-\nu \partial_\theta^2 H_\nu =\nu \partial_\theta^2 h.
 \end{align*}
 Multiplying by $\overline{H_\nu}$, integrating in $\theta$ and taking the real part, we obtain by Cauchy-Schwarz inequality
 \begin{align*}
     \mathrm{Re}(\lambda) \Vert H_\nu \Vert_{L^2_\theta}^2 \leq \nu \Vert \partial_\theta^2 h \Vert_{L^2_\theta} \Vert H_\nu \Vert_{L^2_\theta}.
 \end{align*}
By dividing by $\Vert H_\nu \Vert_{L^2_\theta}$, we get 
    \begin{align*}
         \left\Vert \big((\lambda+\mathrm{L}_{\nu})^{-1}-(\lambda+\mathrm{L})^{-1}\big)[\sin(\cdot)] \right\Vert_{L^2_\theta}=\Vert h_\nu -h \Vert_{L^2_\theta} \leq \frac{\Vert \partial_\theta^2 h \Vert_{L^2_\theta}  }{\vert \mathrm{Re}(\lambda) \vert}\nu.
\end{align*}
Next, we simply write by Cauchy-Schwarz inequality
 \begin{align*}
     \vert D^\nu_\zeta(\lambda)-D_\zeta(\lambda) \vert &\leq \zeta \int_0^{2\pi} \left\vert \big((\lambda+\mathrm{L}_{\nu})^{-1}-(\lambda+\mathrm{L})^{-1}\big)[\sin(\cdot)](\theta) \right\vert  \, \mathrm{d}\theta \\
     & \leq  \sqrt{2\pi} \zeta  \left\Vert \big((\lambda+\mathrm{L}_{\nu})^{-1}-(\lambda+\mathrm{L})^{-1}\big)[\sin(\cdot)] \right\Vert_{L^2_\theta},
 \end{align*}
 so that we end up with 
\begin{align*}
    \vert D^\nu_\zeta(\lambda)-D_\zeta(\lambda) \vert \leq \sqrt{2\pi} \zeta  \frac{\Vert \partial_\theta^2 h \Vert_{L^2_\theta}  }{\vert \mathrm{Re}(\lambda) \vert}\nu.
\end{align*}
By the explicitly expression of $h$, we can now compute
\begin{align*}
\partial_\theta^2 h(\theta)=-\lambda \left( \frac{ \sin(\theta)}{(\lambda+i\sin(\theta))^2}+2i \frac{(\sin(\theta))^2}{(\lambda+i \sin(\theta))^3} \right),
\end{align*}
so that 
\begin{align*}
    \vert \partial_\theta^2 h(\theta) \vert^2 \leq \vert \lambda \vert^2\left( \frac{2}{\vert \lambda+i\sin(\theta) \vert^4 }+\frac{8}{\vert \lambda+i \sin(\theta) \vert^6 } \right) \leq  \vert \lambda \vert^2 \left( \frac{ 2}{\mathrm{Re}(\lambda)^4 }+\frac{8}{\mathrm{Re}(\lambda)^6 } \right),
\end{align*}
and eventually
\begin{align*}
    \Vert \partial_\theta^2 h \Vert_{L^2_\theta} \lesssim  \vert \lambda \vert \left( \frac{ 1}{\mathrm{Re}(\lambda)^2 }+\frac{1}{\mathrm{Re}(\lambda)^3 } \right).
\end{align*}
If $\lambda \in B(\lambda_\zeta,r_0)$, we have $\mathrm{Re}(\lambda)>r_0$ and $\vert \lambda \vert< 3r_0$. Therefore we conclude that there exists a constant $C(r_0)>0$ such that for all $\lambda \in B(\lambda_\zeta,r_0)$, we have
\begin{align*}
    \vert D^\nu_\zeta(\lambda)-D_\zeta(\lambda) \vert \leq C(r_0)\nu.
\end{align*}
This finally achieves the proof.
 \end{proof}

 \section{Linear stability and decay estimates with rotational diffusion}\label{Section-stab-diffusive}
In this section, we aim at proving the linear stability result from Theorem \ref{thm-stability-reduced}, in the stable case $0<\zeta<\frac{1}{2\pi}$. We restart from the equation  \eqref{eq:rescaled eq2-MODELA} 
and aim at proving that the following estimates
\begin{align}
   \label{bound-to-target1-proof}\vert \rho(t) \vert &\lesssim \left(1+\nu^{-\frac{5}{4}} \right) \e^{-\eta_0 \nu t} \Vert f^{\mathrm{in}} \Vert_{L^2_\theta}, \qquad t \geq 0,  \\ 
    \label{bound-to-target2-proof}\Vert f(t) \Vert_{L^2_\theta} &\lesssim \left(1+\nu^{-\frac{7}{4}}\right)\e^{-\eta_0 \nu t} \Vert f^{\mathrm{in}} \Vert_{L^2_\theta}, \qquad t \geq 0,
\end{align}
for some $\eta_0>0$ and $0<\nu \ll 1$. This would then conclude the proof of the second part of Theorem \ref{thm-stability-reduced}, when $\nu>0$.

\medskip

Recalling the notation
\begin{align*}
    \mathrm{L}_{\nu}=i \sin(\theta) -\nu \partial_{ \theta}^2,
\end{align*}
we first rewrite equation \eqref{eq:rescaled eq2-MODELA}  as
\begin{align*}
    \partial_t f+ \mathrm{L}_{\nu} f=  \zeta i \sin(\theta) \rho, \qquad \rho(t)=\int f(t,\theta)\, \mathrm{d}\theta.
\end{align*}
By Duhamel formula, we obtain
\begin{align*}
    f(t,\theta)=(\e^{-\mathrm{L}_{\nu} t} f^{\mathrm{in}})(\theta)+\zeta  \int_0^t \rho(s) \Big(\e^{-\mathrm{L}_{\nu} (t-s)} i \sin \Big)(\theta) \, \mathrm{d}s.
\end{align*}
Integrating in $\theta$, we classically get the following closed Volterra equation on the density $\rho(t)$:
\begin{align}\label{eq:Volterra-stab}
    \rho(t)-\zeta \int_0^t \mathrm{K}_{\nu}(t-s) \rho(s) \, \mathrm{d}s = \int_0^{2\pi} (\e^{-\mathrm{L}_{\nu} t} f^{\mathrm{in}})(\theta) \, \mathrm{d}\theta, 
\end{align}
where the Volterra kernel is defined as
\begin{align}\label{def:Kernel-stab}
    \mathrm{K}_{\nu}(t)\vcentcolon=\int_0^{2\pi} \Big(\e^{-\mathrm{L}_{\nu} t} i \sin \Big)(\theta) \, \mathrm{d}\theta.
\end{align}
\subsection{Decay estimates for transport--diffusion equations}\label{Section-enhanceddiff}
The purpose of this subsection is to recall a decay estimate for a simple transport--diffusion equation with a non-degenerate critical point in the transport velocity.  
Results of this type are by now standard and appear in various forms in the literature on enhanced dissipation and hypocoercivity; see for instance the works on shear flows and mixing-enhanced dissipation
\cites{MicheleJacob,MicheleThierry,Villringer25,ABN22,BW13}.
The key feature is that the interaction between transport and diffusion leads to a decay rate that is much faster than the purely diffusive one.
We start with the following lemma, which provides an exponential-type decay for the semigroup generated by the operator $\mathrm{L}_\nu$.

\begin{Lem}[Enhanced dissipation]\label{LM-enhanceddiffusion}
There exist constants $C_0>0$, $\varepsilon_0>0$ and $\nu_0>0$ such that the following holds.  
For any $\nu\in(0,\nu_0]$, the solution $g=g(t,\theta)$ of
\[
    \partial_t g + \mathrm{L}_\nu g = 0,
    \qquad
    g|_{t=0}=g_0,
\]
with smooth initial datum $g_0=g_0(\theta)$ satisfies
\begin{equation}\label{eq:enhanced-dissipation}
    \|g(t)\|_{L^2_\theta}
    \;\leq\;
    C_0\, \e^{-\varepsilon_0 \nu^{1/2} t}\,
    \|g_0\|_{L^2_\theta},
\end{equation}
for every $t\geq 0$.
\end{Lem}
This decay estimate reflects an enhanced dissipation mechanism induced by the transport term in $\mathrm{L}_\nu$, whose velocity field $\sin\theta$ possesses a single non-degenerate critical point.  
Recalling the definition of the kernel $\mathrm{K}_\nu$ in \eqref{def:Kernel-stab}, we infer that for all $t\geq 0$,
\[
    |\mathrm{K}_\nu(t)|
    \;\lesssim\;
    \left\| \e^{-\mathrm{L}_\nu t}\, i\sin\theta \right\|_{L^2_\theta}
    \;\lesssim\;
    \e^{-\varepsilon_0 \nu^{1/2} t}\, \|\sin\theta\|_{L^2_\theta},
\]
and therefore
\begin{equation}\label{decay-kernel}
    |\mathrm{K}_\nu(t)|
    \;\lesssim\;
    \e^{-\varepsilon_0 \nu^{1/2} t},
    \qquad
    \forall\, t\geq 0,
\end{equation}
uniformly for $\nu\leq \nu_0$. Similarly, for the forcing term generated by the initial data, we obtain
\begin{equation}\label{decay-forcing}
    \left|
        \int_0^{2\pi}
        \big( \e^{-\mathrm{L}_\nu t} f^{\mathrm{in}} \big)(\theta)\,
        \mathrm{d}\theta
    \right|
    \;\lesssim\;
    \e^{-\varepsilon_0 \nu^{1/2} t}\,
    \|f^{\mathrm{in}}\|_{L^2_\theta},
    \qquad
    \forall\, t\geq 0,
\end{equation}
again for $\nu\leq \nu_0$.

 \subsection{Stability condition for the Volterra equation}\label{Section-Stab-Volterra}
Let us come back to the analysis of the Volterra equation \eqref{eq:Volterra-stab}. In what follows, $\eps_0$ and $\nu_0$ are constants given by Lemma \ref{LM-enhanceddiffusion}. We can always assume that $\nu_0<1$. 
We consider the Laplace transform of the kernel $\mathrm{K}_{\nu}$
\begin{align*}
    \mathcal{L}[\mathrm{K}_{\nu}](\lambda)=\int_0^{+\infty} \e^{-\lambda t}\mathrm{K}_{\nu}(t) \, \mathrm{d}t, 
\end{align*}
In view of the exponential decay \eqref{decay-kernel}, if $\nu \in (0, \nu_0)$, the former expression is well defined and even analytic for $\lambda \in \lbrace \mathrm{Re} > -\eps_0  \nu^{1/2}\rbrace$.

\medskip

Our goal is to prove spectral stability conditions of the form \eqref{Intro-stab-cond} and \eqref{Intro-stab-cond-quantitative} for the Volterra equation \eqref{eq:Volterra-stab},  with $\delta = c \nu$, $c$ small enough  (independent of $\nu$) and $m = c' \nu$, $c'$ small enough and independent of $\nu$.

In what follows, we constantly rely on the following key fact. Since  the kernel $\mathrm{K}_{\nu}$ is defined through the (integral of) the semigroup of $-\mathrm{L}_{\nu}$ (see \eqref{def:Kernel-stab}), its Laplace transform is given by the resolvent: we know that for any  $\lambda \in \lbrace \mathrm{Re} > - \eps_0 \nu^{1/2} \rbrace$, we have
    \begin{align*}
        \mathcal{L}[\mathrm{K}_{\nu}](\lambda) =\int_0^{2\pi} \Big[(\lambda+\mathrm{L}_{\nu})^{-1} i\sin \Big](\theta) \, \mathrm{d}\theta =\int_0^{2\pi} f(\theta) \, \mathrm{d}\theta,
    \end{align*}
    where $f=f_{\lambda, \nu}(\theta)$ satisfies
    \begin{align}\label{spectral-advec-diff}
        (\lambda +i \sin(\theta) -\nu \partial_{ \theta}^2)f=i \sin(\theta).
    \end{align}    

\medskip

Our first result deals with the non-vanishing of the Laplace transform in the half-space $\lbrace \mathrm{Re} \geq 0 \rbrace$.

\begin{Lem}\label{LM:stability-cond-Laplacekernel-REALPART-positive}
Suppose that $\zeta \in (0,\frac{1}{2\pi})$. For all $\nu \in (0, \nu_0)$, we have
\begin{align}
    \forall \lambda \in \lbrace \mathrm{Re} \geq 0 \rbrace, \qquad 1-\zeta \mathcal{L}[\mathrm{K}_{\nu}](\lambda) \neq 0.
\end{align}
\end{Lem}
\begin{proof}
As recalled before, for $\delta \in (0, \eps_0 \nu^{1/2})$ and  $\lambda \in \lbrace \mathrm{Re} \geq -\delta \rbrace$, we have
    \begin{align*}
        \mathcal{L}[\mathrm{K}_{\nu}](\lambda) =\int_0^{2\pi} f(\theta) \, \mathrm{d}\theta, \qquad 
        (\lambda +i \sin(\theta) -\nu \partial_{ \theta}^2)f=i \sin(\theta).
    \end{align*}     
We now argue by contradiction: let us suppose that there exists $\lambda \in \lbrace \mathrm{Re} \geq 0 \rbrace$ such that $$1-\zeta \mathcal{L}[\mathrm{K}_{\nu}](\lambda) = 0.$$ It yields the following key identity
    \begin{align}\label{contrad-arg-Laplace0}
        \frac{1}{\zeta} =\int_0^{2\pi} f(\theta) \, \mathrm{d}\theta,
    \end{align}
    so in particular $\int_0^{2\pi} f(\theta) \, \mathrm{d}\theta \in \R^+ {\setminus \lbrace 0 \rbrace}$. Integrating  equation \eqref{spectral-advec-diff} on $f$, we get by periodicity
    \begin{align}\label{eq-average-resolvent}
    \lambda \int_0^{2\pi} f(\theta) \, \mathrm{d}\theta = \int_0^{2\pi} \left( -i \sin(\theta) f(\theta) +\nu \partial_{ \theta}^2 f(\theta)+i \sin(\theta)\right)\, \mathrm{d}\theta =- i\int_0^{2\pi}  \sin(\theta) f(\theta) \, \mathrm{d}\theta.
    \end{align}
    Taking the complex conjugate of the equation on $f$, we also have
    \begin{align*}
        (\overline{\lambda} -i \sin(\theta) -\nu \partial_{ \theta}^2)\overline{f}=-i \sin(\theta),
    \end{align*}
    so that multiplying by $f$ and integrating, we get
    \begin{align*}
        - i\int_0^{2\pi}  \sin(\theta) f(\theta) \, \mathrm{d}\theta&=\int_0^{2\pi} \left( \overline{\lambda} \, \overline{f}(\theta) -i \sin(\theta) \overline{f}(\theta) -\nu \partial_{ \theta}^2 \overline{f}(\theta) \right) f(\theta) \, \mathrm{d}\theta \\
        &=\overline{\lambda} \int_0^{2\pi}\vert f(\theta) \vert^2 \, \mathrm{d}\theta  
        -i \int_0^{2\pi} \sin(\theta) \vert f(\theta) \vert^2 \, \mathrm{d}\theta+\nu \int_0^{2\pi}\vert \partial_\theta f(\theta) \vert^2 \, \mathrm{d}\theta.
    \end{align*}
    thanks to an integration by parts in $\theta$. 
This provides 
\begin{align*}
    \lambda \int_0^{2\pi} f(\theta) \, \mathrm{d}\theta=\overline{\lambda} \int_0^{2\pi}\vert f(\theta) \vert^2 \, \mathrm{d}\theta  
        -i \int_0^{2\pi} \sin(\theta) \vert f(\theta) \vert^2 \, \mathrm{d}\theta+\nu \int_0^{2\pi}\vert \partial_\theta f(\theta) \vert^2 \, \mathrm{d}\theta,
\end{align*}
so that by taking the real part, we end up with
\begin{align}\label{eq:realpart-arg-Laplace0}
    \mathrm{Re}(\lambda) \int_0^{2\pi} f(\theta) \, \mathrm{d}\theta=\mathrm{Re}(\lambda)\int_0^{2\pi}\vert f(\theta) \vert^2 \, \mathrm{d}\theta   +\nu \int_0^{2\pi}\vert \partial_\theta f(\theta) \vert^2 \, \mathrm{d}\theta.
\end{align}
If $\mathrm{Re}(\lambda)=0$, then $\partial_\theta f=0$ so that $f=c_{\lambda}$ is constant in $\theta$. Coming back to the equation on $f$, we get
\begin{align*}
    (\lambda+ i\sin(\theta))c_{\lambda}=i\sin(\theta)
\end{align*}
for all $\theta$, which is not possible. If $\mathrm{Re}(\lambda)>0$, \eqref{eq:realpart-arg-Laplace0} now yields
\begin{align*}
    \int_0^{2\pi} f(\theta) \, \mathrm{d}\theta=\int_0^{2\pi}\vert f(\theta) \vert^2 \, \mathrm{d}\theta   +\frac{\nu}{ \mathrm{Re}(\lambda)} \int_0^{2\pi}\vert \partial_\theta f(\theta) \vert^2 \, \mathrm{d}\theta>\int_0^{2\pi}\vert f(\theta) \vert^2.
\end{align*}
Since we also have \begin{align*}
    \int_0^{2\pi} f(\theta)  \, \mathrm{d}\theta \leq \sqrt{2 \pi} \left( \int_0^{2\pi}\vert f(\theta) \vert^2 \, \mathrm{d}\theta \right)^{1/2},
\end{align*}
we get the bound
\begin{align*}
    \left( \int_0^{2\pi}\vert f(\theta) \vert^2 \, \mathrm{d}\theta \right)^{1/2} <  \sqrt{2 \pi}.
\end{align*}
Using once again the former identity, we end up with
\begin{align*}
    \int_0^{2\pi} f(\theta)  \, \mathrm{d}\theta < 2\pi.
\end{align*}
Finally, coming back to  \eqref{contrad-arg-Laplace0}, we obtain 
\begin{align*}
    \frac{1}{\zeta}=\int_0^{2\pi} f(\theta)  \, \mathrm{d}\theta<2\pi,
\end{align*}
which contradicts the assumption $\zeta \in (0, \frac{1}{2\pi})$. This ends the proof.  
\end{proof}
We now turn to the study of the Laplace transform in the left half-space $\lbrace \mathrm{Re}<0 \rbrace$. For any $\delta \in (0, \eps_0 \nu^{-\frac12})$, we write
    $$\lbrace -\delta \le \mathrm{Re} <0 \rbrace= S_\delta^{\ge} \sqcup S_\delta^{<},$$
    where
    \begin{align*}
        S_\delta^{\ge} &\vcentcolon= \left\lbrace \lambda \in \C \mid -\delta \le  \mathrm{Re}(\lambda) <0 \quad \text{and} \quad \vert \mathrm{Im}(\lambda) \vert \ge 1 + \sqrt{2} \right\rbrace ,\\
        S_\delta^{<} &\vcentcolon= \left\lbrace \lambda \in \C \mid -\delta < \mathrm{Re}(\lambda) <0 \quad \text{and} \quad \vert \mathrm{Im}(\lambda) \vert < 1 + \sqrt{2} \right\rbrace.
    \end{align*}
We further split $S_\delta^{<}$ as
    \begin{align*}
    S_\delta^{<}= \mathrm{A}_{\delta} \sqcup \mathrm{B}_{\delta},
    \end{align*}
    where for all $\nu \in (0, \nu_0)$ we introduce
    \begin{align*}
        \mathrm{A}_{\delta} &\vcentcolon=  S_\delta^{<} \cap \left\lbrace \lambda \in \C \mid \vert 1-\zeta \mathcal{L}[\mathrm{K}_{\nu}](\lambda) \vert  \leq \frac{\nu}{2} \right\rbrace , \\
        \mathrm{B}_{\delta} &\vcentcolon=  S_\delta^{<} \cap \left\lbrace \lambda \in \C \mid \vert 1-\zeta \mathcal{L}[\mathrm{K}_{\nu}](\lambda) \vert > \frac{\nu}{2} \right\rbrace.
    \end{align*}
We refer to Figure \ref{Figure-splitting} below where the different sub-regions of the complex plane are pictured.
\begin{figure}[h]
\centering
    \begin{tikzpicture}[scale=0.9]
    \def\cnu{1.2}      
    \def\R{1.5}        
    \def\viewH{2.8}    
    \def\viewW{4.5}    
    \def\ahsize{0.2}   

    \draw[thick] (-\viewW,0) -- (\viewW,0);   
    \draw[thick] (0,-\viewH) -- (0,\viewH);   



    \fill[pink!20,opacity=0.8] (0,-\viewH) rectangle (\viewW,\viewH);

    \draw[very thick,red] (-\cnu,-\viewH) -- (-\cnu,\viewH);
    \node[red,anchor=south east,font=\small] at (-\cnu, \viewH) {$\Re=-\delta$};

    \draw[dashed] (-\viewW,\R) -- (\viewW,\R);
    \draw[dashed] (-\viewW,-\R) -- (\viewW,-\R);
    \node[anchor=west,font=\small] at (\viewW,\R) {$\Im=1 + \sqrt{2}$};
    \node[anchor=west,font=\small] at (\viewW,-\R) {$\Im=-1 - \sqrt{2}$};

    \fill[orange!30,opacity=0.8] (-\cnu,-\R) rectangle (0,\R);
    \node at ($(-\cnu/2,0.5)$) {$S_\delta^{<}$};
    \node at ($(\cnu+\cnu,0.5)$) {$\Re>0$};

    \fill[cyan!30,opacity=0.8] (-\cnu,\R) rectangle (0,\viewH);      
    \fill[cyan!30,opacity=0.8] (-\cnu,-\viewH) rectangle (0,-\R);    
    \node at ($(-\cnu/2,{(\R+\viewH)/2})$) {$S_\delta^{\geq}$};
    \node at ($(-\cnu/2,{-(\R+\viewH)/2})$) {$S_\delta^{\geq}$};

    \draw[->,thick] (-\viewW,0) -- (\viewW,0) node[right] {$\Re(\lambda)$};
    \draw[->,thick] (0,-\viewH) -- (0,\viewH) node[above] {$\Im(\lambda)$};
\end{tikzpicture}
\caption{Splitting of $\lbrace -\delta < \mathrm{Re} <0 \rbrace$.}
\label{Figure-splitting}
\end{figure}

\medskip

The goal is then  to study the behavior of $1-\zeta \mathcal{L}[\mathrm{K}_{\nu}]$ in the two sub-regions $\mathrm{A}_{\delta}$ and $S_\delta^{>}$ (since by definition we already have a lower-bound for this quantity on $\mathrm{B}_{\delta}$). 
We have the following key lemmas.

\begin{Lem}\label{LM:Laplace-bound-Im>R}
Suppose that $\zeta \in (0, \frac{1}{2\pi})$. For all $\nu \in (0, \nu_0)$, for all $\delta \in (0, \eps_0 \nu^{1/2})$, the following holds:
    \begin{align*}
        \forall \lambda \in S_{\delta}^{\ge}, \qquad \vert 1-\zeta \mathcal{L}[\mathrm{K}_{\nu}](\lambda) \vert \ge \frac{1}{2}. 
    \end{align*}
\end{Lem}
\begin{Lem}\label{LM:Laplace-bound-Im<R}
Suppose that $\zeta \in (0, \frac{1}{2\pi})$. Let $c_0 = \frac12 \eps_0 \nu_0^{-\frac12}$. There exists  $c \in (0, c_0)$, $c' > 0$ and $\nu'_0 \in (0, \nu_0)$ such that for all $\nu \in (0, \nu'_0)$, we have
    \begin{align*}
        \forall \lambda \in \mathrm{A}_{c\nu}, \qquad \vert 1-\zeta \mathcal{L}[\mathrm{K}_{\nu}](\lambda) \vert \ge c' \nu.
    \end{align*}
\end{Lem}
Assuming for a moment that Lemmas \ref{LM:Laplace-bound-Im>R}--\ref{LM:Laplace-bound-Im<R} hold, and appealing to Lemma \ref{LM:stability-cond-Laplacekernel-REALPART-positive}, we can now state the following corollary. Here, the  constants $c,c', \nu'_0$ are those given by Lemma \ref{LM:Laplace-bound-Im<R}.
\begin{Cor}\label{Coro-Laplace-bound}
   Suppose that $\zeta \in (0,\frac{1}{2\pi})$. For all $ \nu \in (0, \nu'_0)$, we have
    \begin{align} \label{eq:non-vanishing-Laplace-whole-half-space}
    \forall \lambda \in \lbrace -c \nu \le  \mathrm{Re} \rbrace, \qquad  1-\zeta \mathcal{L}[\mathrm{K}_{\nu}](\lambda) \vert \neq 0,
\end{align}
and 
    \begin{align} \label{eq:quantitative-bound-Laplace}
        \forall \lambda \in \lbrace -c \nu \le \mathrm{Re} <0 \rbrace, \qquad \vert 1-\zeta \mathcal{L}[\mathrm{K}_{\nu}](\lambda) \vert \ge  \min\left( \frac{1}{2},c'\right)\nu.
    \end{align}
\end{Cor}
Let us now turn to the proof of 
Lemmas \ref{LM:Laplace-bound-Im>R}--\ref{LM:Laplace-bound-Im<R}.
\begin{proof}[Proof of Lemma \ref{LM:Laplace-bound-Im>R}]
As before, for $\lambda \in \lbrace \mathrm{Re} > - \eps_0 \nu^{\frac12} \rbrace$, we have
    \begin{align*}
        \mathcal{L}[\mathrm{K}_{\nu}](\lambda)=\int_0^{2\pi} f(\theta) \, \mathrm{d}\theta, \qquad 
        (\lambda +i \sin(\theta) -\nu \partial_{ \theta}^2)f=i \sin(\theta).
    \end{align*}
To lighten the notation, we set $\Vert \cdot \Vert=\Vert \cdot \Vert_{L^2_\theta}$ in the rest of the section. Multiplying by the conjugate $\overline{f}$ and integrating $\theta$ in the former equation, we obtain after taking the imaginary part: 
\begin{align*}
    \Im(\lambda) \,  \|f\|^2 + \int_0^{2\pi} \sin(\theta) |f(\theta)|^2 \dd\theta = \mathrm{Im}\left(\int_0^{2\pi} i \sin(\theta) \overline{f}(\theta) \, \mathrm{d}\theta \right),
\end{align*}
By Cauchy-Schwarz inequality, we can therefore infer that for $\vert \mathrm{Im}(\lambda) \vert>1$, we have
\begin{align*}
\left(\vert \mathrm{Im}(\lambda) \vert -1\right) \|f\|^2   \le \|\sin\| \, \|f\| = \sqrt{\pi} \|f\|,
\end{align*}
hence the bound
\begin{align*}
\Vert f \Vert \le \frac{\sqrt{\pi}}{\vert \mathrm{Im}(\lambda) \vert-1}.
\end{align*}
We now write by the triangular inequality and Cauchy-Schwarz inequality
\begin{align*}
\vert 1-\zeta \mathcal{L}[\mathrm{K}_{\nu}](\lambda) \vert \geq  1-\zeta \left\vert \int_0^{2\pi} f(\theta) \, \mathrm{d}\theta \right\vert \geq 1-\sqrt{2\pi} \zeta \Vert f \Vert,
\end{align*}
so that for $\lambda \in S_{\delta}^{\ge}$, by  the former upper bound on $\|f\|$ and as $\zeta < \frac{1}{2\pi}$, we infer
\begin{align*}
\vert 1-\zeta \mathcal{L}[\mathrm{K}_{\nu}](\lambda) \vert \ge  1-\frac{\sqrt{2}\pi \zeta}{|\Im \lambda|-1} \geq  1 - \pi \zeta \ge \frac{1}{2}.
\end{align*}
It concludes the proof.
\end{proof}
\begin{proof}[Proof of Lemma \ref{LM:Laplace-bound-Im<R}]
We shall prove a slightly stronger result:  there exists $c \in (0, c_0)$, $c' > 0$ and $\nu'_0 \in (0, \nu_0)$ such that for all $\nu \leq \nu'_0$, we have
    \begin{align*}
        \forall \lambda \in \mathrm{A}_{c\nu}, \qquad \zeta \vert  \mathrm{Im}( \mathcal{L}[\mathrm{K}_{\nu}](\lambda) )   \vert \ge  c' \nu.
    \end{align*}
Once again, for all  $\lambda \in \lbrace \mathrm{Re} > -\eps_0 \nu^{\frac12} \rbrace$, we have
    \begin{align*}
        \mathcal{L}[\mathrm{K}_{\nu}](\lambda) =\int_0^{2\pi} f_{\lambda, \nu}(\theta) \, \mathrm{d}\theta, \qquad 
        (\lambda +i \sin(\theta) -\nu \partial_{\theta}^2)f_{\lambda, \nu}=i \sin(\theta).
    \end{align*}     
We emphasize the dependency of $f_{\lambda, \nu}$ in $\lambda$ and $\nu$ for the rest of the proof. We also denote
\begin{align*}
  \langle f_{\lambda, \nu} \rangle \vcentcolon= \frac{1}{2\pi} \int_{0}^{2\pi}f_{\lambda, \nu}(\theta) \, \mathrm{d}\theta, \qquad   \widetilde f_{\lambda, \nu} \vcentcolon= f_{\lambda, \nu} - \langle f_{\lambda, \nu} \rangle,
\end{align*}
the mean part and the oscillatory parts of $f_{\lambda, \nu}$ (that are orthogonal in $L^2_\theta$). We now proceed in two steps.

\medskip

\noindent \textbf{Step 1.}  The following holds 
\begin{align}\label{eq:control-liminf-Dissip}
\underset{\nu \rightarrow 0^+}\liminf \, \inf_{\lambda \in A_{c_0\nu}}  \int_0^{2\pi} \vert \partial_\theta  f_{\lambda, \nu}(\theta) \vert^2 \, \mathrm{d}\theta >0. 
\end{align}
Let us prove this fact, arguing by contradiction. We assume that the former limit is zero: there exists a positive sequence $(\nu_n)$ converging to zero and a sequence $(\lambda_n)$ with $\lambda_n \in \mathrm{A}_{c_0 \nu_n}$ such that
\begin{align*}
\int_0^{2\pi} \vert \partial_\theta  f_{\lambda_n, \nu_n}(\theta) \vert^2 \, \mathrm{d}\theta \underset{n \rightarrow + \infty}{\longrightarrow} 0.
\end{align*}
The same holds for $\widetilde{f}_{\lambda_n, \nu_n}$ instead of $f_{\lambda_n, \nu_n}$, so, by Poincaré inequality we can deduce that $\widetilde{f}_{\lambda_n, \nu_n} \overset{n \rightarrow + \infty}{\rightarrow} 0$ in $H^1_\theta$. Furthermore, since $\lambda_n \in \mathrm{A}_{c_0\nu_n}$, we know that
\begin{align*}
\left\vert 1-\zeta \int_0^{2\pi} f_{\lambda_n, \nu_n}(\theta) \, \mathrm{d}\theta](\lambda) \right\vert  \leq \frac{\nu_n}{2},
\end{align*}
therefore $\langle f_{\lambda_n, \nu_n} \rangle  \overset{n \rightarrow + \infty}{\rightarrow} \frac{1}{2\pi \zeta}$. By compactness, we can also assume that $\lambda_n \overset{n \rightarrow + \infty}{\rightarrow} \lambda \in \C$, with $\mathrm{Re}(\lambda)=0$. We can now pass to the limit in the sense of distributions in the equation
\begin{align*}
\lambda_n f_{\lambda_n, \nu_n} +i \sin(\theta) f_{\lambda_n, \nu_n} -\nu_n \partial_{\theta}^2 f_{\lambda_n, \nu_n}=i \sin(\theta),
\end{align*}
that yields
\begin{align*}
\lambda \frac{1}{2\pi \zeta}+ i \sin(\theta)\frac{1}{2\pi \zeta}=i \sin(\theta),
\end{align*}
and in particular $\frac{1}{2\pi \zeta}=1$, that is $\zeta=\frac{1}{2\pi}$. This yields a contradiction.

\medskip

\noindent \textbf{Step 2.}  Let us now bound $\mathrm{Im}( \mathcal{L}[\mathrm{K}_{\nu}](\lambda) ) $ from below. According to \eqref{eq:control-liminf-Dissip} from Step 1, we know that there exists $\nu'_0 \in  (0, \nu_0)$ and $\eta>0$ such that for all $\nu \in (0, \nu'_0)$, we have 
\begin{align}\label{eq:conseq-Step1toStep2-liminf}
 \inf_{\lambda \in \mathrm{A}_{c_0 \nu}}  \int_0^{2\pi} \vert \partial_\theta  f_{\lambda, \nu}(\theta) \vert^2 \, \mathrm{d}\theta \geq \eta.
\end{align}
We now restart from the equation 
$$\lambda f_{\lambda, \nu} +i \sin(\theta) f_{\lambda, \nu} -\nu \partial_{\theta}^2f_{\lambda, \nu}=i \sin(\theta).$$
Taking the complex conjugate, we have
    \begin{align*}
        (\overline{\lambda} -i \sin(\theta) -\nu \partial_{ \theta}^2)\overline{f_{\lambda,\nu}}=-i \sin(\theta),
    \end{align*}
    so that multiplying by $f_{\lambda,\nu}$ and integrating, we get
    \begin{align*}
        - i\int_0^{2\pi}  \sin(\theta) f_{\lambda, \nu}(\theta) \, \mathrm{d}\theta&=\int_0^{2\pi} \left( \overline{\lambda} \, \overline{f_{\lambda, \nu}}(\theta) -i \sin(\theta) \overline{f_{\lambda, \nu}}(\theta) -\nu \partial_{ \theta}^2 \overline{f_{\lambda, \nu}}(\theta) \right) f_{\lambda, \nu}(\theta) \, \mathrm{d}\theta \\
        &=\overline{\lambda} \int_0^{2\pi}\vert f_{\lambda, \nu}(\theta) \vert^2 \, \mathrm{d}\theta  
        -i \int_0^{2\pi} \sin(\theta) \vert f_{\lambda, \nu}(\theta) \vert^2 \, \mathrm{d}\theta+\nu \int_0^{2\pi}\vert \partial_\theta f_{\lambda, \nu}(\theta) \vert^2 \, \mathrm{d}\theta,
    \end{align*}
    thanks to an integration by parts in $\theta$. 
This provides the identity
\begin{align*}
    \lambda \int_0^{2\pi} f_{\lambda, \nu}(\theta) \, \mathrm{d}\theta=\overline{\lambda} \int_0^{2\pi}\vert f_{\lambda, \nu}(\theta) \vert^2 \, \mathrm{d}\theta  
        -i \int_0^{2\pi} \sin(\theta) \vert f_{\lambda, \nu}(\theta) \vert^2 \, \mathrm{d}\theta+\nu \int_0^{2\pi}\vert \partial_\theta f_{\lambda, \nu}(\theta) \vert^2 \, \mathrm{d}\theta.
\end{align*}
Taking the real part in this relation, we obtain
\begin{align*}
&\mathrm{Re}(\lambda) \ \mathrm{Re}\left( \int_0^{2\pi} f_{\lambda, \nu}(\theta) \, \mathrm{d}\theta\right) -\mathrm{Im}(\lambda) \ \mathrm{Im}\left(\int_0^{2\pi} f_{\lambda, \nu}(\theta) \, \mathrm{d}\theta \right)\\ & \qquad=\mathrm{Re}(\lambda) \int_0^{2\pi}\vert f_{\lambda, \nu}(\theta) \vert^2 \, \mathrm{d}\theta +\nu \int_0^{2\pi}\vert \partial_\theta f_{\lambda, \nu}(\theta) \vert^2 \, \mathrm{d}\theta \\
&\qquad=\mathrm{Re}(\lambda) \left( \Vert \langle f_{\lambda, \nu}\rangle \Vert^2+ \Vert \widetilde{f}_{\lambda, \nu}\Vert^2 \right)  +\nu \int_0^{2\pi}\vert \partial_\theta f_{\lambda, \nu}(\theta) \vert^2 \, \mathrm{d}\theta.
\end{align*}
Note that for all $\lambda \in \mathrm{A}_{c_0 \nu}$, we have by definition
\begin{align*}
\langle f_{\lambda, \nu}\rangle=\frac{1}{2\pi \zeta}+\mathcal{O}(\nu),
\end{align*}
as $\nu \rightarrow 0$. As a consequence, for any $0 \le c \le c_0$, if $\lambda \in \mathrm{A}_{c\nu}$ we infer from the previous identities, Poincaré inequality and the bound \eqref{eq:conseq-Step1toStep2-liminf} that
 \begin{align*}
 &\vert \mathrm{Im}(\lambda) \vert \left\vert\mathrm{Im}\left(\int_0^{2\pi} f_{\lambda, \nu}(\theta) \, \mathrm{d}\theta \right)\right\vert \\
 &\qquad\geq \nu \int_0^{2\pi}\vert \partial_\theta f_{\lambda, \nu}(\theta) \vert^2 \, \mathrm{d}\theta
 - \vert \mathrm{Re}(\lambda) \vert \left( \left\vert \mathrm{Re}\left( \int_0^{2\pi} f_{\lambda, \nu}(\theta) \, \mathrm{d}\theta\right) \right\vert+  \Vert \langle f_{\lambda, \nu}\rangle \Vert^2+ \Vert \widetilde{f}_{\lambda, \nu}\Vert^2 \right) \\
 &\qquad\geq (\nu-c C_P\nu) \int_0^{2\pi}\vert \partial_\theta f_{\lambda, \nu}(\theta) \vert^2 \, \mathrm{d}\theta
 - c \nu  \left( \frac{1}{\zeta}+ \frac{1}{(2\pi \zeta)^2}+\mathcal{O}(\nu)\right) \\
 &\qquad\geq
\nu(1-c C_P) \eta
 - c \nu  \left( \frac{1}{\zeta}+ \frac{1}{(2\pi \zeta)^2}+\mathcal{O}(\nu) \right),
 \end{align*}
 for some universal constant $C_P>0$. Choosing $c \in (0, c_0)$ small enough, and up to reducing $\nu'_0$ to control the $\mathcal{O}(\nu)$ remainder, we deduce the lower-bound
 \begin{align*}
 \vert \mathrm{Im}(\lambda) \vert \left\vert\mathrm{Im}\left(\int_0^{2\pi} f_{\lambda, \nu}(\theta) \, \mathrm{d}\theta \right)\right\vert \geq \frac{\eta}{2} \nu,
 \end{align*}
We deduce that 
 \begin{align*}
\left\vert\mathrm{Im}\left(\int_0^{2\pi} f_{\lambda, \nu}(\theta) \, \mathrm{d}\theta \right)\right\vert \geq \frac{\eta \nu}{2 \vert \mathrm{Im}(\lambda) \vert} \geq c' \nu, 
 \end{align*}
 with $c' = \frac{\eta}{2(1+\sqrt{2})}$. This is the desired estimate.
\end{proof}

 \subsection{Final estimates}\label{section-finalestimMODELA}
We aim at proving that the following estimates on the density $\rho(t)$ and on the distribution function $f(t,\theta)$ (under the assumption $0 < \zeta < \frac{1}{2\pi}$):  there exists $\eta_0 \in (0, \eps_0 \nu_0^{-1/2})$ (where $\eps_0, \nu_0$ are defined in Lemma \ref{LM-enhanceddiffusion}) and $\nu'_0>0$ such that for $\nu \in (0, \nu'_0)$, we have 
\begin{align}
   \label{bound-final-rho} \vert \rho(t) \vert &\lesssim \left(1+\nu^{-\frac{5}{4}}\right) \e^{-\eta_0 \nu t} \Vert f^{\mathrm{in}} \Vert_{L^2_\theta}, \qquad t \geq 0,  \\ 
   \label{bound-final-f} \Vert f(t) \Vert_{L^2_\theta} &\lesssim \left(1+\nu^{-\frac{7}{4}}\right)\e^{-\eta_0 \nu t} \Vert f^{\mathrm{in}} \Vert_{L^2_\theta}, \qquad t \geq 0.
\end{align}
In what follows, we can always assume that $\nu'_0<1$.
\medskip

\underline{\eqref{bound-final-rho} implies \eqref{bound-final-f}}: let us assume temporarily that \eqref{bound-final-rho} holds. Coming back to the Duhamel formula for $f$, we know that 
\begin{align*}
    f(t,\theta)=(\e^{-\mathrm{L}_{\nu} t} f^{\mathrm{in}})(\theta)+\zeta  \int_0^t \rho(s) \Big(\e^{-\mathrm{L}_{\nu} (t-s)} i \sin \Big)(\theta) \, \mathrm{d}s,
\end{align*} and therefore, relying on the enhanced-diffusion estimate from Lemma \ref{LM-enhanceddiffusion} and on the bound \eqref{bound-final-rho}, we obtain for all $t \geq 0$
\begin{align*}
    \Vert f(t) \Vert_{L^2_\theta} &\lesssim \left\Vert \e^{-\mathrm{L}_{\nu} t} f^{\mathrm{in}} \right\Vert_{L^2_\theta}+ \int_0^t \vert \rho(s) \vert \left\Vert \e^{-\mathrm{L}_{\nu} (t-s)} i \sin \right\Vert_{L^2_\theta} \, \mathrm{d}s \\
    & \lesssim \e^{-\eps_0 \nu^{\frac{1}{2}} t} \left\Vert f^{\mathrm{in}} \right\Vert_{L^2_\theta} 
    + \left(1+\nu^{-\frac{5}{4}}\right)\Vert f^{\mathrm{in}}  \Vert_{L^2_\theta}\e^{-\eps_0 \nu^{\frac{1}{2}}t}\int_0^t   \e^{(\eps_0\nu^{\frac{1}{2}}-\eta_0 \nu) s} \, \mathrm{d}s .
\end{align*}
Since $\eps_0\nu^{\frac{1}{2}}-\eta_0 \nu=\nu^{\frac{1}{2}}\left(\eps_0-\eta_0 \nu^{\frac{1}{2}}\right)$ and since we can always reduce $\nu'_0$ so that for all $\nu \in (0, \nu'_0)$ we have $\eps_0-\eta_0 \nu^{\frac{1}{2}}> \eps_0/10$, we find $\eps_0\nu^{1/2}-\eta_0 \nu> \eps_0\nu^{1/2}/10>0 $. We get for all $t \geq 0$
\begin{align*}
    \Vert f(t) \Vert_{L^2_\theta} &\lesssim \e^{-\eps_0 \nu^{\frac{1}{2}} t} \left\Vert f^{\mathrm{in}} \right\Vert_{L^2_\theta} 
    + 10 \nu^{-\frac{1}{2}} \left( 1+\nu^{-\frac{5}{4}} \right)\eps_0^{-1} \e^{-\eps_0 \nu^{\frac{1}{2}}t}\left(\e^{(\eps_0\nu^{\frac{1}{2}}-\eta_0 \nu) t} - 1\right)\Vert f^{\mathrm{in}} \Vert_{L^2_\theta} \\
    & \lesssim \left(1+\nu^{-\frac{1}{2}}+\nu^{-\frac{7}{4}}\right)\e^{-\min(\eps_0 \nu^{1/2}, \eta_0 \nu) t}\Vert f^{\mathrm{in}} \Vert_{L^2_\theta} \\
    & \lesssim\left(1+\nu^{-\frac{7}{4}}\right)\e^{-\eta_0 \nu t}\Vert f^{\mathrm{in}} \Vert_{L^2_\theta}.
\end{align*}
We have thus obtained the claimed estimate \eqref{bound-final-f}.

\medskip

\underline{Proof of \eqref{bound-final-rho}}: we first recall the Volterra equation \eqref{eq:Volterra-stab} satisfied by $\rho$ and  that reads for all $t>0$
\begin{align*}
    \rho(t)-\zeta \int_0^t \mathrm{K}_{\nu}(t-s) \rho(s) \, \mathrm{d}s = S(t), \qquad S(t)=\int_0^{2\pi} (\e^{-\mathrm{L}_{\nu} t} f^{\mathrm{in}})(\theta) \, \mathrm{d}\theta,
\end{align*}
We shall use notation from Section \ref{Section-Stab-Volterra}, and  \eqref{eq:non-vanishing-Laplace-whole-half-space}-\eqref{eq:quantitative-bound-Laplace}.
In what follows, we set
\begin{align*}
    \gamma=\frac{c}{2},
\end{align*}
and in particular, we have $\gamma \nu <c \nu<\eps_0 \nu^{1/2}$. In view of the decay estimates \eqref{decay-kernel} and \eqref{decay-forcing} for the kernel $\mathrm{K}_\nu(t)$ and the source $S(t)$, we get for all $\nu \in (0, \nu'_0) \subset (0, \nu_0)$
\begin{align*}
    t \mapsto \e^{\gamma \nu t} \mathrm{K}_\nu(t) \in L^1(\R^+), \qquad t \mapsto \e^{\gamma \nu t} S(t) \in L^1(\R^+),
\end{align*} 
Appealing to \eqref{eq:non-vanishing-Laplace-whole-half-space}, and to the Paley-Wiener theorem (see Theorem \ref{thm-Paley-Wiener} from the Appendix), we obtain the fact that the Volterra equation has a unique solution satisfying $t \mapsto \e^{\gamma \nu t} \rho(t) \in L^1(\R^+)$.

\medskip

We now argue with ideas in the spirit of \cite{MouhotVillani}. First, we extend the kernel $K_\nu$, the source $S$ and the solution $\rho$ by zero for negative times $t<0$, so that the Volterra equation now reads
\begin{align*}
    \underline{\rho}(t)-\zeta \int_{\R} \underline{\mathrm{K}_{\nu}}(t-s) \rho(s) \, \mathrm{d}s = \underline{S}(t), \qquad t \in \R,
\end{align*}
where $\underline{\cdot}$ denotes the extension operator by zero. Multiplying by $\e^{\gamma \nu t}$ in both sides, and defining
\begin{align*}
            (\rho_+, \mathrm{K}_{\nu, +}, S_+)(t)\vcentcolon=\e^{\gamma \nu t} (\overline{\rho}, \overline{\mathrm{K}_\nu}, \overline{S})(t), \qquad t>0
        \end{align*}
 we obtain
        \begin{align*}
        \rho_+(t)-\zeta  (\mathrm{K}_{\nu,+} \star \rho_+)(t) \, \mathrm{d}s = S_+(t), \qquad t \in \R,\end{align*}
where the convolution appearing above is the standard convolution in time on $\R$. By the previous integrability properties, it makes sense to take the Fourier transform of such equation, where we use the normalization
\begin{align*}
    \widehat{g}(\xi)=\int_{\R} \e^{-i t \xi}
 g(t) \, \mathrm{d}t.
 \end{align*}
This yields
\begin{align*}
   \widehat{\rho_+}(\xi)\left(1-\zeta  \widehat{\mathrm{K}_{\nu,+}}(\xi) \right) = \widehat{S_+}(\xi), \qquad \xi \in \R.
\end{align*}
Using the classical observation that
\begin{align*}
    \mathcal{L}[\mathrm{K}_\nu](\lambda)=\widehat{\mathrm{K}_{\nu,+}}(\xi), \qquad \lambda=i \xi-\gamma \nu,
\end{align*}
we are left to consider $\mathrm{Re}(\lambda)=-\gamma \nu$ in the previous stability conditions: in particular, the bound from below \eqref{eq:quantitative-bound-Laplace} implies that 
$1-\zeta  \widehat{\mathrm{K}_{\nu,+}}(\xi)$ is bounded away from zero for all $\xi \in \R$, so that we can write
        \begin{align*}
            \widehat{\rho_+}(\xi)=\frac{\widehat{S_+}(\xi)}{1-\zeta  \widehat{\mathrm{K}_{\nu,+}}(\xi)}, \qquad \xi \in \R.
            \end{align*}
Since $\mathrm{Im}(\lambda)=\xi$, the bound \eqref{eq:quantitative-bound-Laplace} and Plancherel theorem entail
\begin{align*}
    \Vert \rho_+ \Vert_{L^2(\R)} \lesssim  \frac{\Vert S_+ \Vert_{L^2(\R)}}{\min\left( \frac{1}{2},c'\right)\nu}.
\end{align*}
Coming back to the definition of the source $S$ and relying on the bound \eqref{decay-forcing}, we also have
\begin{align*}
    \Vert S_+ \Vert_{L^2(\R)} \lesssim \Vert \e^{-\eps_0 \nu^{\frac{1}{2}}t}\Vert_{L^2(\R^+)}  \Vert f^{\mathrm{in}} \Vert_{L^2_\theta} \le \nu^{-\frac{1}{4}} \Vert f^{\mathrm{in}} \Vert_{L^2_\theta},
\end{align*}
so that, since $\nu <\nu'_0<1$, we end up with
\begin{align*}
    \Vert \rho_+ \Vert_{L^2(\R)} \lesssim  \nu^{-\frac{5}{4}}\Vert f^{\mathrm{in}} \Vert_{L^2_\theta}.
\end{align*}
We now aim at improving this $L^2$ bound into a pointwise bound in time. Rather than relying on Young's inequality for convolution (that would yield a worse factor in terms of $\nu$), we proceed as follows (remembering that $\rho_+$ is not in $H^1_\theta$  because of the extension by $0$): considering a mollification $\chi_{\vert \xi \vert \leq R}$ of the indicator function $\mathbf{1}_{\vert \xi \vert \leq R}$, we split between low and high frequencies by writing
\begin{align*}
    \rho_+(t)=\mathcal{F}^{-1}\left(\chi_{\vert \xi \vert \leq R}\frac{\widehat{S_+}(\xi)}{1-\zeta  \widehat{\mathrm{K}_{\nu,+}}(\xi)} \right)
    +\mathcal{F}^{-1}\left((1-\chi_{\vert \xi \vert \leq R})\frac{\widehat{S_+}(\xi)}{1-\zeta  \widehat{\mathrm{K}_{\nu,+}}(\xi)}\right)\vcentcolon=\mathrm{I}(t)+\mathrm{II}(t).
\end{align*}
where $\mathcal{F}^{-1}$ stands for the standard inverser Fourier transform.
For the low frequency part, we have
\begin{align*}
\Vert \mathrm{I} \Vert_{H^1(\R)} \lesssim_R \Vert \mathrm{I} \Vert_{L^2(\R)} \lesssim_R \frac{\Vert S_+ \Vert_{L^2(\R)}}{\nu} \lesssim_R \nu^{-\frac{5}{4}}\Vert f^{\mathrm{in}} \Vert_{L^2_\theta},
\end{align*}
as before. For the high-frequency part, we first observe that the proof of Lemma \ref{LM:Laplace-bound-Im>R} yields the control
\begin{align*}
    \vert \widehat{\mathrm{K}_{\nu,+}}(\xi) \vert \lesssim \frac{1}{\vert \xi \vert-1}, \qquad \vert \xi \vert>R>1,
\end{align*}
since $\mathrm{Im}(\lambda)=\xi$. We take $R = 1 + 2 \zeta$,  hence $\mathbf{1}_{\vert \xi \ge  R} \zeta \vert \widehat{\mathrm{K}_{\nu,+}}(\xi) \vert <\frac{1}{2}$. We can now expand the previous expression in power series:
\begin{align*}
   \widehat{\mathrm{II}}(\xi)=(1-\chi_{\vert \xi \vert \leq R})\frac{\widehat{S_+}(\xi)}{1-\zeta  \widehat{\mathrm{K}_{\nu,+}}(\xi)}
    =\widehat{S_+}(\xi)-\chi_{\vert \xi \vert \leq R}\widehat{S_+}(\xi) 
    +(1-\chi_{\vert \xi \vert \leq R})\widehat{S_+}(\xi)\displaystyle\sum_{k=1}^{+\infty} \zeta^k  \widehat{\mathrm{K}_{\nu,+}}(\xi)^k.
\end{align*}
Note that the last part satisfies
\begin{align*}
    \left\vert (1-\chi_{\vert \xi \vert \leq R})\widehat{S_+}(\xi)\displaystyle\sum_{k=1}^{+\infty} \zeta^k  \widehat{\mathrm{K}_{\nu,+}}(\xi)^k \right\vert \lesssim (1-\chi_{\vert \xi \vert \leq R})\vert \widehat{S_+}(\xi)\vert \frac{\vert \zeta \widehat{\mathrm{K}_{\nu,+}}(\xi) \vert }{1-\zeta \vert \widehat{\mathrm{K}_{\nu,+}}(\xi) \vert} \lesssim (1-\chi_{\vert \xi \vert \leq R})\frac{\vert \widehat{S_+}(\xi)\vert}{\vert \xi \vert}.
\end{align*}
As before, we can infer that
\begin{align*}
\Vert \mathrm{II}-S_+ \Vert_{H^1(\R)} \lesssim \nu^{-\frac{5}{4}}\Vert f^{\mathrm{in}} \Vert_{L^2_\theta},
\end{align*}
and therefore one has
\begin{align*}
\Vert \rho_+-S_+ \Vert_{H^1(\R)} \lesssim \nu^{-\frac{5}{4}}\Vert f^{\mathrm{in}} \Vert_{L^2_\theta}.  
\end{align*}
By Sobolev embedding in 1d, we can now deduce that
\begin{align*}
    \e^{\gamma \nu t}\vert \rho(t) - S(t) \vert \lesssim \nu^{-\frac{5}{4}}\Vert f^{\mathrm{in}} \Vert_{L^2_\theta}, \qquad t >0,
\end{align*}
and according to the bound \eqref{decay-forcing} on the source $S$, we get
\begin{align*}
    \e^{\gamma \nu t}\vert \rho(t) \vert \lesssim \e^{\gamma \nu t}\vert \rho(t)-S(t) \vert+  \e^{\gamma \nu t} \e^{-\eps_0 \nu^{\frac{1}{2} t}}\Vert f^{\mathrm{in}} \Vert_{L^2_\theta} \lesssim \left(1+\nu^{-\frac{5}{4}}\right) \Vert f^{\mathrm{in}} \Vert_{L^2_\theta}.
\end{align*}
We finally end up with with the decay estimate 
\begin{align*}
 \vert \rho(t) \vert \lesssim  \left(1+\nu^{-\frac{5}{4}}\right) \Vert f^{\mathrm{in}} \Vert_{L^2_\theta} \e^{-\gamma \nu t},
\end{align*}
which yields the desired bound \eqref{bound-final-rho} and concludes the proof.

\section{A variant of the main equation \eqref{eq:fulleq-INTRO}}\label{Appendix-Variant-eq}

The purpose of this section is to show how to extend the study performed in the former sections to a variant of the master equation \eqref{eq:fulleq-INTRO}. Since it will share very close features with what has been done before, we will only highlight the main points of the analysis.

\medskip

We consider the equation
    \begin{align} 
    \partial_t f+ \mathrm{div}_x \big[ f(1-\rho_f) e(\theta)+ p_f f\big]=\kappa \mathrm{div}_x[(1-\rho_f)\nabla_x f+ f \nabla_x \rho_f]+\nu \partial^2_{ \theta} f,
\end{align}
where the local density and momentum are defined by 
$$\rho_f \vcentcolon=\int_{0}^{2\pi} e(\theta) f \, \mathrm{d}\theta, \qquad  p_f\vcentcolon= \int_{0}^{2\pi} e(\theta) f \, \mathrm{d}\theta,$$
and where $\kappa, \nu \geq 0$. This model can for instance be found in \cite{Bruna-models} (see also \cite{ASregularity}). Note that here, we only consider the case of an affine decreasing velocity depending on the density. As for \eqref{eq:fulleq-INTRO} before (See (4) in Remark \ref{rmk-THM-instab}), we can choose a proper time-space rescaling (different from \cite{bruna-wellposedness-fulldiff}) according to which the translational and rotational diffusion coefficients $\kappa$ and $\nu$ can be taken different. To ease the parallel with the analysis performed for \eqref{eq:fulleq-INTRO}, we will enforce $\kappa=0$. 

\medskip
In the rest of this section, we will therefore consider 
\begin{align}\label{Model B}
    \partial_t f+ \mathrm{div}_x \big[ f(1-\rho_f) e(\theta)+ p_f f\big] = \nu \partial^2_{ \theta} f.
\end{align}
As for \eqref{eq:fulleq-INTRO}, the equation \eqref{Model B} admits the following class of homogeneous steady states
\begin{align*}
    f^\star=\frac{\phi}{2\pi}, \qquad \rho_{f^\star}=\phi, \qquad p_{f^\star}=0 \qquad \text{for some} \qquad \phi \in (0,1),
\end{align*}
and the linearization of \eqref{Model B} around such an equilibrium reads
\begin{align}\label{eq:MODELB-linearizedEQ-intro-full}
  \partial_t f+ (1-\phi)e(\theta)\cdot \nabla_x f - \frac{\phi}{2 \pi} e(\theta) \cdot \nabla_x \rho+ \frac{\phi}{2\pi} \mathrm{div}_x(p)=\nu \partial_{ \theta}^2 f,
\end{align}
while its non-diffusive counterpart with $\nu=0$ is
\begin{align}\label{eq:MODELB-linearizedEQ-intro-fullINVISCID}
  \partial_t f+ (1-\phi)e(\theta)\cdot \nabla_x f - \frac{\phi}{2 \pi} e(\theta) \cdot \nabla_x \rho+ \frac{\phi}{2\pi} \mathrm{div}_x(p)=0,
\end{align}
where we set
$$\rho \vcentcolon=\int_{0}^{2\pi} e(\theta) f \, \mathrm{d}\theta, \qquad  p\vcentcolon= \int_{0}^{2\pi} e(\theta) f \, \mathrm{d}\theta,$$
for the sake of simplicity. For this model, our main results are gathered in the following theorem, which is in the same spirit as Theorems \ref{thm-instability}--\ref{thm-stability}. Roughly speaking, it asserts the existence of the same stability threshold $\phi=1/2$ for the density. 

\begin{Thm}\label{thm-modelB}
Let $\phi \in (0, 1)$. Then the following alternative holds:
    \begin{itemize}
    \item \underline{Case $\phi \in (1/2, 1)$}: 
    \begin{itemize}
        \item If $\nu=0$, there exists an exponential growing mode solution to the non-diffusive linearized equation \eqref{eq:MODELB-linearizedEQ-intro-fullINVISCID}
        \item If $\nu>0$ is small enough, the same statement as before persists for the diffusive linearized equation \eqref{eq:MODELB-linearizedEQ-intro-full}. 
    \end{itemize} 
    \item \underline{Case $\phi \in (0,1/2)$}:
    \begin{itemize}
        \item If $\nu=0$, there exists no exponential growing mode solution to the non-diffusive linearized equation \eqref{eq:MODELB-linearizedEQ-intro-fullINVISCID}
        \item If $\nu>0$ is small enough, the same statement as before persists for the diffusive linearized equation \eqref{eq:MODELB-linearizedEQ-intro-full}.
    \end{itemize}
\end{itemize}
\end{Thm}
For what concerns stability or instability of homogeneous steady states at the linear level, Model \eqref{Model B} therefore shares close properties with \eqref{eq:fulleq-INTRO}. In the rest of this section, we aim at providing the main steps to prove Theorem \ref{thm-modelB}: since the main strategy and methods are quite close, we only highlight the key features without detailing everything. Note that regarding the stability of \eqref{eq:MODELB-linearizedEQ-intro-full} when $\nu=0$ and $\nu>0$, we only prove that spectral stability holds and do not provide any quantitative time decay rate. We believe that the approaches taken in Sections \ref{Section-Landau-damping} and \ref{Section-stab-diffusive} are robust enough to be adapted in that case and would lead to suitable decay estimates, in the same spirit as the ones of Theorem \ref{thm-stability}. For the sake of conciseness, we prefer not to delve into such refinements.

Performing the same procedure as in Section \ref{Section:reductionScaling-STAB-MODELA}, it is now enough to study the following reduced equation
\begin{align}\label{eq:rescaled eq2-MODELB}
     \partial_t f + i \sin(\theta)f -\nu \partial_\theta^2 f= i\zeta  \sin(\theta) \rho-i\zeta  p,
\end{align}
and its non-diffusive counterpart
\begin{align}\label{eq:rescaled eq2-MODELB-inv}
 \partial_t f + i \sin(\theta)f = i\zeta  \sin(\theta) \rho-i\zeta  p,
\end{align}
where we have set
\begin{align*}
    \zeta\vcentcolon=\frac{\phi}{2\pi(1-\phi)} \in \R^+_\star,
\end{align*}
and
\begin{align}\label{def:normalize-rho-p-ModelB}
    \rho(t)\vcentcolon=\int f(t,\theta)\, \mathrm{d}\theta, \qquad p(t)\vcentcolon=\int \sin(\theta) f(t,\theta)\, \mathrm{d}\theta.
\end{align}
We will infer an equivalent version of Theorems \ref{thm-instability-reduced}--\ref{thm-stability-reduced} from \eqref{eq:rescaled eq2-MODELB}--\eqref{eq:rescaled eq2-MODELB-inv}, the regime $\zeta>\frac{1}{2\pi}$ (resp. $\zeta<\frac{1}{2\pi}$) corresponding to $\phi>\frac{1}{2}$ (resp. $\phi<\frac{1}{2}$). For the sake of conciseness, we do not write these statements, modifications being straightforward.

\subsection{Dispersion relation in the non-diffusive case}
We focus on the non-diffusive linearized equation, that is \eqref{eq:rescaled eq2-MODELA-inv}. The goal is mainly to derive an explicit dispersion relation for this problem, providing suitable unstable modes. 
\begin{Def}
    For $\lambda \notin i [-1,1]$, we define 
    \begin{align}\label{eq:dispersion-rel-inviscidModelB}
\Gamma_\zeta(\lambda)\vcentcolon=\mathrm{det}\left(\mathrm{I}+\zeta \mathbb{A}_{\lambda} \right),
    \end{align}
    where we set
    \begin{align}\label{eq:defMatrixA}
\mathbb{A}_{\lambda} \vcentcolon=  \begin{pmatrix}
-\mathcal{I}_{\lambda} & \mathcal{J}_{\lambda} \\ 
-\mathcal{K}_{\lambda} & \mathcal{I}_{\lambda}
\end{pmatrix},
\end{align}
with
\begin{align*}
\mathcal{I}_{\lambda} \vcentcolon= i\int_{-\pi}^{\pi} \frac{ \sin(\theta) }{ \lambda + i\sin(\theta)} \, \mathrm{d}\theta, \ \ 
\mathcal{J}_{\lambda}\vcentcolon= i\int_{-\pi}^{\pi} \frac{1}{ \lambda +i \sin(\theta)} \, \mathrm{d}\theta, \ \ 
\mathcal{K}_{\lambda}\vcentcolon= i\int_{-\pi}^{\pi} \frac{ \sin(\theta)^2 }{ \lambda + i\sin(\theta)} \, \mathrm{d}\theta.
\end{align*}
\end{Def}
The equivalent of Proposition \ref{Prop-iifcond-growingmode} is now the following.
\begin{Prop}
    For \eqref{eq:rescaled eq2-MODELB-inv} to admit a non-trivial solution of the form 
\begin{align}\label{eq:growing-mode-inviscid-MODELB}
    f(t,\theta)=\e^{\lambda t}  \mathrm{f}(\theta),  \: \lambda \notin i[-1,1],
\end{align}
 it is necessary and sufficient that
    \begin{align}\label{eig-prob-model3}
        \Gamma_\zeta\left (\lambda \right)=0.
    \end{align}
\end{Prop}
\begin{proof}
    By looking for a growing mode solution to \eqref{eq:MODELB-linearizedEQ-intro-fullINVISCID}, we observe that it has to satisfy
\begin{align*}
    \left( \lambda + i\sin(\theta) \right)\mathrm{f}(\theta)= i\zeta \sin(\theta) \rho-i\zeta p, \ \ \mathrm{with} \ \ \rho=\int_0^{2\pi} \mathrm{f}(\theta)\, \mathrm{d}\theta, \ \  p=\int_0^{2\pi} \sin(\theta) \mathrm{f}(\theta)\, \mathrm{d}\theta,
\end{align*}
and we therefore get
\begin{align*}
\mathrm{f}(\theta)=\frac{i\zeta \sin(\theta)  }{\lambda + i\sin(\theta)}\rho-\frac{i\zeta}{ \lambda + i\sin(\theta)}p.
\end{align*}
We obtain by successive  integration
\begin{align*}
\rho&=\zeta i \rho \int_0^{2\pi} \frac{ \sin(\theta) }{\lambda + i\sin(\theta)} \, \mathrm{d}\theta
-\zeta i p \int_0^{2\pi} \frac{1}{\lambda + i\sin(\theta)} \, \mathrm{d}\theta, \\
p&=\zeta i \rho \int_0^{2\pi} \frac{  \sin(\theta)^2 }{\lambda + i\sin(\theta)} \, \mathrm{d}\theta
-\zeta i p \int_0^{2\pi} \frac{\sin(\theta)}{\lambda + i\sin(\theta)} \, \mathrm{d}\theta.
\end{align*}
 We can rewrite it as a  $2 \times 2$ system on the vectorial unknown $(\rho, p) \in \R^2$ as
\begin{align}\label{disp-rel-model2-nondiffusive}
\left(\mathrm{I}+\zeta \mathbb{A}_{\lambda} \right) \begin{pmatrix}
\rho \\[1mm] p
\end{pmatrix}=0,
\end{align}
where
\begin{align*}
\mathbb{A}_{\lambda} =  \begin{pmatrix}
-\mathcal{I}_{\lambda} & \mathcal{J}_{\lambda} \\ 
-\mathcal{K}_{\lambda} & \mathcal{I}_{\lambda}
\end{pmatrix},
\end{align*}
using Definition \eqref{eq:defMatrixA}. Since we are looking for a non-trivial solution, the condition $\mathrm{det}\left(\mathrm{I}+\zeta \mathbb{A}_{\lambda} \right)=0$
exactly means that $\Gamma_\zeta\left(\lambda \right)=0$, which is enough to conclude the proof.
\end{proof}
Continuing our computations, we can infer the following result providing the solutions of the dispersion relation.
\begin{Prop}\label{Prop-instalinear-inviscid-MODEL2}
For all $\lambda \in \lbrace \mathrm{Re}>0 \rbrace$ and $k \in \Z^2 {\setminus \lbrace 0 \rbrace}$, we have 
\begin{align*}
    \Gamma_\zeta\left (\lambda \right)=1-4 \pi^2 \zeta^2 \left(1-\frac{1}{\sqrt{1+\frac{1}{\lambda^2}}} \right).
\end{align*}
Here, $\sqrt{\cdot}$ is the principal determination of the square-root, that is defined and holomorphic on $\C \setminus \R_-$ and that has a positive real part. In addition, the following alternative holds:
\begin{itemize}
    \item if $\zeta>\frac{1}{2\pi}$ then the function $\lambda \mapsto \Gamma_\zeta\left (\lambda \right)$ has two distinct zeros $\lambda_\zeta^\pm$ in $\C \setminus i[-1, 1]$, which are
    \begin{align*}
        \lambda_\zeta^\pm&=\pm \widetilde{c}_\zeta \in \R^{\pm},
    \end{align*}
    for some $\widetilde{c}_\zeta>0$.
    \item if $ \zeta<\frac{1}{2\pi}$, then the function $\lambda \mapsto \Gamma_\zeta\left (\lambda \right)$ has no zeros in $\C \setminus i[-1, 1]$.
\end{itemize}  
\end{Prop}
\begin{proof}
    According to \eqref{eq:dispersion-rel-inviscidModelB}, using the expression for the determinant defining $\Gamma_\zeta\left (\lambda \right)$, we have for all $ \lambda \in \C \setminus i[-1, 1]$
$$ \Gamma_\phi\left (\lambda \right)=(1-\zeta \mathcal{I}_{\lambda})(1+\zeta \mathcal{I}_{\lambda})+\zeta^2 \mathcal{K}_{\lambda} \mathcal{J}_{\lambda}.$$
By the same transformations as in the proof of Lemma \ref{LM:formula-dispersionD} (relying on Lemma \ref{LM:integral-rationalfrac}), we have
\begin{align*}
\mathcal{I}_{\lambda}=2\pi\frac{\sqrt{1+\frac{1}{\lambda^2}}-1}{\sqrt{1+\frac{1}{\lambda^2}}}, \qquad \mathcal{J}_{\lambda}=\frac{2i\pi}{\lambda\sqrt{1+\frac{1}{\lambda^2}}},
\end{align*}
and similarly we can write
\begin{align*}
    \mathcal{K}_{\lambda}=i\int_{-\pi}^{\pi} \frac{ \sin(\theta)^2 }{ \lambda + i\sin(\theta)} \, \mathrm{d}\theta=\int_{-\pi}^{\pi} \frac{ \sin(\theta)^2-(\lambda/i)^2 }{ \lambda/i +\sin(\theta)} \, \mathrm{d}\theta+(\lambda/i)^2\int_{-\pi}^{\pi} \frac{ 1 }{ \lambda/i + \sin(\theta)} \, \mathrm{d}\theta,
\end{align*}
so that
\begin{align*}
    \mathcal{K}_{\lambda}=2i\pi \lambda-\lambda^2\mathcal{J}_{\lambda}=2i\pi \lambda \left( 1-\frac{1}{\sqrt{1+\frac{1}{\lambda^2}}} \right).
\end{align*}
Then, using the former expression of the coefficients, 
we get after a few manipulations
\begin{align*}
\Gamma_\phi\left (\lambda \right)=1-\zeta^2 \mathcal{I}_{\lambda}^2+\zeta^2 \mathcal{K}_{\lambda} \mathcal{J}_{\lambda}&=1-4\pi^2 \zeta^2 \lambda^2 \frac{\left(\sqrt{1+\frac{1}{\lambda^2}}-1\right)^2}{\lambda^2+1}+4\pi^2\zeta^2 \lambda^2\frac{1-\sqrt{1+\frac{1}{\lambda^2}}}{\lambda^2+1}. \\
&=1+\frac{4\pi^2 \zeta^2 \lambda^2}{\lambda^2+1} \left(1-\sqrt{1+\frac{1}{\lambda^2}}  \right) \sqrt{1+\frac{1}{\lambda^2}},
\end{align*}
which gives the claimed formula
\begin{align*}
\Gamma_\phi\left (\lambda \right)&=1-4\pi^2 \zeta^2\left(1-\frac{1}{\sqrt{1+\frac{1}{\lambda^2}}} \right).
\end{align*}
We can now solve the equation $\Gamma_\phi\left (\lambda \right)=0$ as in the proof of Proposition \ref{Prop-instalinear-inviscid}, by observing that the left-hand side of
\begin{align*}
    1-4\pi^2 \zeta^2 +\frac{4\pi^2 \zeta^2}{\sqrt{1+\frac{1}{\lambda^2}}}=0
\end{align*}
has a real part greater than $1-4\pi^2 \zeta^2$, hence the equation has no solution if $1-4\pi^2 \zeta^2>0$. Details are left to the reader.
\end{proof}

\subsection{Instability with small rotational diffusion}
We now study how instabilities persist if $0 < \zeta < \frac{1}{2\pi}$ and under a weakly diffusive regime assumption, that is considering the linearized equation \eqref{eq:rescaled eq2-MODELB} with $0 <\nu \ll 1$.

\medskip

As in Section \ref{Section-instab-diffusive}, we can rewrite this equation as
\begin{align*}
    f=i\zeta \rho\Big(\lambda+i\sin(\theta)-\nu \partial_\theta^2\Big)^{-1}[\sin]  -i\zeta  p  \Big(\lambda+i\sin(\theta)-\nu \partial_\theta^2\Big)^{-1}[1] .
\end{align*}
Integrating in $\theta$ against $1$ and $\sin(\theta)$, we have
\begin{align*}
    \rho&= i\zeta \rho \int_0^{2\pi}\Big(\lambda+i\sin-\nu \partial_\theta^2\Big)^{-1}[\sin](\theta) \, \mathrm{d}\theta  -i \zeta  p \int_0^{2\pi} \Big(\lambda+i\sin-\nu \partial_\theta^2\Big)^{-1}[1](\theta) \, \mathrm{d}\theta, \\
    p&=i\zeta \rho \int_0^{2\pi} \sin(\theta)\Big(\lambda+i\sin-\nu \partial_\theta^2\Big)^{-1}[\sin](\theta) \, \mathrm{d}\theta  -i\zeta  p  \int_0^{2\pi} \sin(\theta)\Big(\lambda+i\sin-\nu \partial_\theta^2\Big)^{-1}[1](\theta) \, \mathrm{d}\theta,
\end{align*}
and it is therefore enough to look for a solution of the system
\begin{align}
(\mathrm{I}+\zeta \mathbb{A}_{\lambda}^{\nu}) \begin{pmatrix}
\rho \\ p
\end{pmatrix}=0, \qquad \mathbb{A}_{\lambda}^{\nu} =  \begin{pmatrix}
-\iota_{\lambda}^{\nu} & j_{\lambda}^{\nu} \\[2mm] 
-\xi_{\lambda}^{\nu} & \ell_{\lambda}^{\nu}
\end{pmatrix},
\end{align}
where
\begin{align*}
\iota_{\lambda}^{\nu}\vcentcolon&= i \int_0^{2\pi}\Big(\lambda+i\sin-\nu \partial_\theta^2\Big)^{-1}[\sin](\theta) \, \mathrm{d}\theta, 
\qquad j_{\lambda}^{\nu}\vcentcolon=i\int_0^{2\pi} \Big(\lambda+i\sin-\nu \partial_\theta^2\Big)^{-1}[1](\theta) \, \mathrm{d}\theta,
\\
\xi_{\lambda}^{\nu}\vcentcolon&=i \int_0^{2\pi}\sin(\theta)\Big(\lambda+i\sin-\nu \partial_\theta^2\Big)^{-1}[\sin](\theta) \, \mathrm{d}\theta, 
\qquad \ell_{\lambda}^{\nu}\vcentcolon=i\int_0^{2\pi} \sin(\theta) \Big(\lambda+i\sin-\nu \partial_\theta^2 \Big)^{-1}[1](\theta) \, \mathrm{d}\theta.
\end{align*}
We therefore define the diffusive function as
\begin{align*}
\Gamma^{\nu}_\zeta(\lambda)\vcentcolon=\mathrm{det}\left(\mathrm{I}+\zeta \mathbb{A}_{\lambda}^{\nu}\right), \qquad \mathrm{Re}(\lambda)>0,
\end{align*}
and we are looking for a zero $\lambda \in \lbrace \mathrm{Re}>0 \rbrace$ of the function $\Gamma_\zeta^{\nu}$. All the coefficients are analytic in $\lambda \in \lbrace \mathrm{Re}>0 \rbrace$ and therefore $\Gamma^{\nu}_\zeta$ as well since
\begin{align*}
\Gamma^{\nu}_\zeta(\lambda)= (1-\iota_{\lambda}^{\nu})(1+\ell_{\lambda}^{\nu})+\xi_{\lambda}^{\nu}j_{\lambda}^{\nu}
 =1+ \zeta\ell_{\lambda}^{\nu}-\zeta\iota_{\lambda}^{\nu} -\zeta^2\iota_{\lambda}^{\nu} \ell_{\lambda}^{\nu}+\zeta^2j_{\lambda}^{\nu}\xi_{\lambda}^{\nu}.
\end{align*}
We can now compare this expression with that of the non-diffusive function from \eqref{eq:dispersion-rel-inviscidModelB}, that is
\begin{align*}
\Gamma_\zeta(\lambda)=\mathrm{det}(\mathrm{I}+\zeta \mathbb{A}_{\lambda})=1-\zeta^2\mathcal{I}_\lambda^2+ \zeta^2 \mathcal{K}_\lambda \mathcal{J}_\lambda.
\end{align*}
We then simply write thanks to the triangular inequality for all $\lambda \in \lbrace \mathrm{Re}>0\rbrace$
\begin{align*}
   \vert \Gamma^\nu_\zeta(\lambda)- \Gamma_\zeta(\lambda)\vert &\lesssim_{\zeta} \vert \ell_{\lambda}^{\nu}-\mathcal{I}_\lambda \vert + \vert \iota_{\lambda}^{\nu}-\mathcal{I}_\lambda \vert + \vert \iota_{\lambda}^{\nu} \ell_{\lambda}^{\nu}-\mathcal{I}_\lambda^2 \vert + \vert j_{\lambda}^{\nu}\xi_{\lambda}^{\nu} - \mathcal{K}_\lambda \mathcal{J}_\lambda \vert \\
   & \leq (1+ \vert \mathcal{I}_\lambda \vert)( \vert \ell_{\lambda}^{\nu}-\mathcal{I}_\lambda \vert + \vert \iota_{\lambda}^{\nu}-\mathcal{I}_\lambda \vert ) + \vert \ell_{\lambda}^{\nu}-\mathcal{I}_\lambda \vert\vert \iota_{\lambda}^{\nu}-\mathcal{I}_\lambda \vert + \vert \xi_{\lambda}^{\nu}-\mathcal{K}_\lambda \vert\vert j_{\lambda}^{\nu}-\mathcal{J}_\lambda \vert\\
   & \quad +\vert \mathcal{K}_\lambda \vert \vert \xi_{\lambda}^{\nu}-\mathcal{K}_\lambda \vert  +\vert \mathcal{J}_\lambda \vert \vert j_{\lambda}^{\nu}-\mathcal{J}_\lambda \vert. 
\end{align*}
We can then perform the same proof as for Proposition \ref{Prop-compa-reldispersion}: from Proposition \ref{Prop-instalinear-inviscid-MODEL2}, there exists $\lambda_\zeta>0$ such that $\Gamma_\zeta(\lambda_\zeta)=0$ so that, if one sets $r_0=\lambda_\zeta/2$ then the equivalent Proposition \ref{Prop-compa-reldispersion} holds: there exists $C(r_0)>0$ such that for any $\nu \in (0,1]$ and for any $\lambda \in B(\lambda_\zeta, r_0)$, we have
\begin{align*}
    \vert \Gamma^{\nu}_\zeta(\lambda)- \Gamma_\zeta(\lambda)\vert \leq C(r_0) \nu.
\end{align*}
Here, we have relied on the proof of Proposition \ref{Prop-compa-reldispersion} and its direct adaptation to $(\lambda+i\sin-\nu\partial_\theta^2)^{-1}[1]$. We can now conclude in the same way as in Section \ref{Section-instab-diffusive}, invoking analyticity and Rouché's theorem to find a zero of $ \Gamma^{\nu}_\zeta$ in $\lbrace \mathrm{Re}>0 \rbrace$, provided that $\nu$ is small enough. The same argument based on the vanishing order of the function $\Gamma_\zeta$ applies and details are left to the reader.

 \subsection{Stability analysis for small rotational diffusion}
We finally study the (spectral) stability for \eqref{eq:rescaled eq2-MODELB}, that is when $\nu>0$, in the stable regime $\zeta<\frac{1}{2\pi}$.

\subsubsection{Volterra equation}
Starting from \eqref{eq:rescaled eq2-MODELB}, we rewrite the equation on $f(t, \theta)$ as
\begin{align}\label{eq:rescaled-MODELB}
     \partial_t f + \mathrm{L}_{\nu} f = i\zeta  \sin(\theta) \rho-i\zeta  p,
\end{align}
with  $\rho(t)$ and $p(t)$ defined as in \eqref{def:normalize-rho-p-ModelB} and
$    \mathrm{L}_{\nu}=i \sin-\nu \partial_\theta^2$.
Compared to \eqref{eq:rescaled eq2-MODELA}, the new equation \eqref{eq:rescaled-MODELB} contains a new term that does not depend on $\theta$. 

\medskip
By Duhamel formula, we have
\begin{align}\label{eq:Duhamel-MODELB}
    f(t,\theta)=(\e^{-\mathrm{L}_{\nu} t} f^{\mathrm{in}})(\theta) 
      + \zeta  \int_0^t \rho(s) \Big(\e^{-\mathrm{L}_{\nu} (t-s)} i \sin \Big)(\theta) \, \mathrm{d}s-i\zeta \int_0^t p(s)\Big(\e^{-\mathrm{L}_{\nu} (t-s)} 1 \Big)(\theta) \, \mathrm{d}s.
\end{align}
Integrating in $\theta$, we classically get the following closed Volterra equation on the quantities $\rho$ and $p$:
\begin{align}
   \label{eq:Volterra-stab-rho-syst} \rho(t)-\zeta \int_0^t \mathrm{K}_{\nu}^{(1)}(t-s) \rho(s) \, \mathrm{d}s+\zeta \int_0^t \mathrm{K}_{\nu}^{(2)}(t-s) p(s) \, \mathrm{d}s  &= \int_0^{2\pi} (\e^{-\mathrm{L}_{\nu} t} f^{\mathrm{in}})(\theta) \, \mathrm{d}\theta, \\
    p(t)-\zeta \int_0^t \mathrm{K}_{\nu}^{(3)}(t-s) \rho(s) \, \mathrm{d}s+\zeta \int_0^t \mathrm{K}_{\nu}^{(4)}(t-s) p(s) \, \mathrm{d}s  &= \int_0^{2\pi} (\e^{-\mathrm{L}_{\nu} t} f^{\mathrm{in}})(\theta) \sin(\theta)\, \mathrm{d}\theta,
\end{align}
where
\begin{align*}
    \mathrm{K}_{\nu}^{(1)}(t)&\vcentcolon=\int_0^{2\pi} \Big(\e^{-\mathrm{L}_{\nu} t} (i \sin) \Big)(\theta) \, \mathrm{d}\theta, \qquad \mathrm{K}_{\nu}^{(2)}(t)\vcentcolon=\int_0^{2\pi} \Big(\e^{-\mathrm{L}_{\nu} t} i \Big)(\theta) \, \mathrm{d}\theta, \\
    \mathrm{K}_{\nu}^{(3)}(t)& \vcentcolon=\int_0^{2\pi} \sin(\theta)\Big(\e^{-\mathrm{L}_{\nu} t} (i \sin) \Big)(\theta) \, \mathrm{d}\theta, \qquad   \mathrm{K}_{\nu}^{(4)}(t) \vcentcolon=\int_0^{2\pi} \sin(\theta) \Big(\e^{-\mathrm{L}_{\nu} t} i \Big)(\theta) \, \mathrm{d}\theta.
\end{align*}
Hence, if one introduces the vectorial unknown 
\begin{align*}
    \mathrm{U}(t)\vcentcolon=\begin{pmatrix}
        \rho(t) \\ p(t)
    \end{pmatrix}, \qquad \mathrm{V}(t)\vcentcolon=\begin{pmatrix}
        \int_0^{2\pi} (\e^{-\mathrm{L}_{\nu} t} f^{\mathrm{in}})(\theta)\, \mathrm{d}\theta \\[2mm] \int_0^{2\pi} (\e^{-\mathrm{L}_{\nu} t} f^{\mathrm{in}})(\theta) \sin(\theta)\, \mathrm{d}\theta,
    \end{pmatrix}
\end{align*}
we get the following vectorial Volterra equation
\begin{align}\label{eq:Volterra-MODELB}
    \mathrm{U}(t)+\zeta \int_0^t \mathcal{K}_{\nu}(t-s)\mathrm{U}(t) \, \mathrm{d}s=\mathrm{V}(t), \qquad \mathcal{K}_{\nu}(t)\vcentcolon=\begin{pmatrix}
        -\mathrm{K}_{\nu}^{(1)}(t) & \mathrm{K}_{\nu}^{(2)} (t) \\
        -\mathrm{K}_{\nu}^{(3)}(t) & \mathrm{K}_{\nu}^{(4)}(t)
    \end{pmatrix}.
\end{align}
By the enhanced-diffusion estimate for $\e^{-\mathrm{L}_{\nu} t}$ from Lemma \ref{LM-enhanceddiffusion} and the same bound as in the end of Section \ref{Section-enhanceddiff}, we have
\begin{align*}
    \vert \mathcal{K}_{\nu}(t) \vert \lesssim \e^{-\eps_0 \nu^{1/2} t}, \qquad t>0, \ \  \nu < \nu_0,
\end{align*}
since it holds for each entry of the matrix $\mathcal{K}_{\nu}(t)$.

\subsubsection{Stability condition}
In what follows, we aim at proving the following proposition concerning the stability condition associated to the Volterra kernel appearing in \eqref{eq:Volterra-MODELB}.
\begin{Prop}\label{PROP:stability-cond-Laplacekernel-modelB}
Suppose that $\zeta \in (0, \frac{1}{2\pi})$. Then the following holds: 
\begin{align}\label{cond-Volterra-kernel-modelB}
\forall \lambda \in \lbrace \mathrm{Re} \geq 0 \rbrace, \qquad  \mathrm{det}\left( \mathrm{I}+\zeta \mathcal{L}[\mathcal{K}_{\nu}](\lambda) \right) \neq 0.
\end{align}
\end{Prop}
\begin{proof}
First,  invoking the decay of the semigroup from Lemma \ref{LM-enhanceddiffusion}, we observe that the Laplace transform (in time) of each entry of $\mathcal{K}_{\nu}$ is well defined for $\lambda \in \lbrace \mathrm{Re}(\lambda) \geq -C \nu^{1/2} \rbrace$: we have 
\begin{align*}
    \mathcal{L}[\mathcal{K}_{\nu}](\lambda)\vcentcolon=\begin{pmatrix}
        -\mathcal{L}[\mathrm{K}_{\nu}^{(1)}](\lambda) & \mathcal{L}[\mathrm{K}_{\nu}^{(2)}](\lambda) \\[2mm]
        -\mathcal{L}[\mathrm{K}_{\nu}^{(3)}](\lambda) & \mathcal{L}[\mathrm{K}_{\nu}^{(4)}](\lambda),
    \end{pmatrix}
\end{align*}
with 
\begin{align*}
    -\mathcal{L}[\mathrm{K}_{\nu}^{(1)}](\lambda)&=-\int_0^{2\pi} \Big[(\lambda+\mathrm{L}_{\nu})^{-1} (i\sin) \Big](\theta) \, \mathrm{d}\theta =\int_0^{2\pi} f(\theta) \, \mathrm{d}\theta, \\ \mathcal{L}[\mathrm{K}_{\nu}^{(2)}](\lambda)&=\int_0^{2\pi} \Big[(\lambda+\mathrm{L}_{\nu})^{-1} i  \Big](\theta) \, \mathrm{d}\theta=\int_0^{2\pi} g(\theta) \, \mathrm{d}\theta, \\
    -\mathcal{L}[\mathrm{K}_{\nu}^{(3)}](\lambda)& =-\int_0^{2\pi} \sin(\theta)\Big[(\lambda+\mathrm{L}_{\nu})^{-1} (i\sin) \Big](\theta) \, \mathrm{d}\theta=\int_0^{2\pi} \sin(\theta) f(\theta) \, \mathrm{d}\theta , \\ \mathcal{L}[\mathrm{K}_{\nu}^{(4)}](\lambda) &=\int_0^{2\pi} \sin(\theta) \Big[(\lambda+\mathrm{L}_{\nu})^{-1} i \Big](\theta) \, \mathrm{d}\theta=\int_0^{2\pi} \sin(\theta) g(\theta) \, \mathrm{d}\theta.
\end{align*}
Above, the functions $f$ and $g$ are defined through
\begin{align}
        (\lambda +i \sin(\theta) -\nu \partial_{ \theta}^2)f&=-i \sin(\theta), \\
        (\lambda +i \sin(\theta) -\nu \partial_{ \theta}^2)g&=i.
    \end{align}
By contradiction, let us now assume that there exists $\lambda \in \lbrace \mathrm{Re}(\lambda) \geq 0 \rbrace$ such that 
$$\mathrm{det}\left( \mathrm{I}+\zeta \mathcal{L}[\mathcal{K}_{\nu}](\lambda)\right) =0.$$
So there exists a nonzero vector $(\rho, p) \in \C^2$ such that $\left( \mathrm{I}+\zeta \mathcal{L}[\mathcal{K}_{\nu}](\lambda)\right)(\rho,p)=0$. 
By linearity, by introducing $h_{\rho, p}$ the solution to
\begin{align}\label{eq:f-rho-p}
    (\lambda +i \sin(\theta) -\nu \partial_{ \theta}^2)h_{\rho, p}&=-i \sin(\theta) \rho +i p,
\end{align}
the former eigenvector identity rewrites as
\begin{align}\label{eq:integrated1-f-rho-p}
    \rho+ \zeta\int_0^{2\pi} h_{\rho, p}(\theta) \, \mathrm{d}\theta&=0, \\
    \label{eq:integrated2-f-rho-p}p+ \zeta \int_0^{2\pi} h_{\rho, p}(\theta) \sin(\theta) \, \mathrm{d}\theta&=0.
\end{align}
To ease readability, we simply write $h=h_{\rho,p}$ in what follows. By integrating $\eqref{eq:f-rho-p}$ in $\theta$, we find
\begin{align}\label{eq:integrated3-f-rho-p}
    \lambda \int_0^{2\pi}  h+i \int_0^{2\pi} \sin(\theta)h =2i \pi p.
\end{align}
By using \eqref{eq:integrated1-f-rho-p} and \eqref{eq:integrated2-f-rho-p}, this yields
\begin{align}\label{rel-lambda-rho-p}
    -\lambda \frac{\rho}{\zeta}-i \frac{p}{\zeta} =+2i \pi p,
\end{align}
which turns, after multiplying by $\zeta \overline{\rho}$ and taking the real part, into
\begin{align}\label{eq:comput-rho^2}
    -\mathrm{Re}(\lambda) \vert \rho \vert^2 =(2\pi\zeta +1)\mathrm{Re}(i p \overline{\rho}).
\end{align}
Coming back to \eqref{eq:f-rho-p}, by multiplying by $\overline{h}$ and integrating, we get
\begin{align*}
   \lambda \int_0^{2\pi}  \vert h\vert^2 +i \int_0^{2\pi} \sin(\theta) \vert h\vert^2+\nu \int_0^{2\pi} \vert \partial_\theta  h\vert^2 =-i \rho \int_0^{2\pi} \sin(\theta)\overline{h}+ip \int_0^{2\pi} \overline{h}.
\end{align*}
By conjugating the equations \eqref{eq:integrated1-f-rho-p}--\eqref{eq:integrated2-f-rho-p} and replacing in the former r.h.s, we obtain
\begin{align}\label{eq:lambda-f-rho-p}
   \lambda \int_0^{2\pi}  \vert h\vert^2 +i \int_0^{2\pi} \sin(\theta) \vert h\vert^2+\nu \int_0^{2\pi} \vert \partial_\theta  h\vert^2 =i\frac{\rho \overline{p}}{\zeta}-i\frac{p \overline{\rho}}{\zeta},
\end{align}
and then taking the real part
\begin{align*}
   \mathrm{Re}(\lambda) \int_0^{2\pi}  \vert h\vert^2 +\nu \int_0^{2\pi} \vert \partial_\theta  h\vert^2 =-\frac{2}{\zeta}\mathrm{Re}(i p \overline{\rho}).
\end{align*}
Plugging the former expression for $\mathrm{Re}(i p \overline{\rho})$ in \eqref{eq:comput-rho^2}, we get 
\begin{align*}
    -\mathrm{Re}(\lambda) \vert \rho \vert^2 =-\frac{\zeta}{2}(2\pi\zeta +1)\left(\mathrm{Re}(\lambda) \int_0^{2\pi}  \vert h\vert^2 +\nu \int_0^{2\pi} \vert \partial_\theta  h\vert^2 \right),
\end{align*}
and using the fact from \eqref{eq:integrated1-f-rho-p} that $\vert \rho \vert^2=\zeta^2 \vert \int_0^{2\pi} h \vert^2 $, this entails after simplification
    \begin{align*}
    \mathrm{Re}(\lambda) \left\vert \int_0^{2\pi} h\right\vert^2 =\left(\pi+\frac{1}{2 \zeta} \right)\left(\mathrm{Re}(\lambda) \int_0^{2\pi}  \vert h\vert^2 +\nu \int_0^{2\pi} \vert \partial_\theta  h\vert^2 \right).
\end{align*}
By the former identity, we obtain
\begin{align*}
     \left\vert \int_0^{2\pi} h\right\vert^2 > \left(\pi+\frac{1}{2 \zeta} \right) \int_0^{2\pi}  \vert h\vert^2 ,
\end{align*}
where we have simplified by $\mathrm{Re}(\lambda)>0$ (the case $\mathrm{Re}(\lambda)=0$ leads to $f_{\rho,p}=0$) and discarded the diffusion term. By Cauchy-Schwarz inequality, we also have
\begin{align*}
    \left\vert \int_0^{2\pi} h\right\vert^2 \leq 2 \pi \int_0^{2\pi}  \vert h\vert^2,
\end{align*}
from which we infer $2\pi> \pi+\frac{1}{2 \zeta}$, and hence $\zeta>\frac{1}{2\pi}$: this is a contradiction since we assumed that $\zeta < \frac{1}{2\pi}$, and therefore concludes the proof.
\end{proof}

\appendix

\section{An integral lemma}
\begin{Lem}\label{LM:integral-rationalfrac}
For all $z \in \C \setminus \left( (-\infty,-1] \cup [1,+\infty) \right)$, we have
\begin{align*}
\int_{-\infty}^{\infty}\frac{ 1}{t^2+2  zt +1} \, \mathrm{d}t =\frac{\pi}{\sqrt{1-z^2}}, 
\end{align*}
where $\sqrt{\cdot}$ is the principal determination of the square-root, that is defined and holomorphic on $\C \setminus \R_-$ (and with positive real part). 
\end{Lem}
\begin{proof}
We first claim that if $z \in U\vcentcolon= \C \setminus \left( (-\infty,-1] \cup [1,+\infty) \right)$, then $t^2 + 2zt + 1 \neq 0$ for all $t \in \R$. Indeed, if one of the roots $t_1,t_2$ of the polynomial is real, relation $t_1 t_2 = 1$ shows that both roots are, and so is $z = -\frac{t_1+t_2}{2}$, hence $z$ belongs to $U \cap \R = (-1,1)$. But the discriminant of the polynomial is then $\Delta \vcentcolon= 4z^2 - 4 < 0$, contradiction. It follows that the l.h.s. of the formula is well-defined. So is the r.h.s, because for all $z \in U$, $1-z^2 \notin \C \setminus \R_-$. Moreover, both sides of the equality  define holomorphic functions of $z$ in $U$: this is obvious for the right-hand side, and follows from classical holomorphy under the integral for the left-hand side.  By analytic continuation, it is then enough to prove the lemma for $z =i\alpha$ with $\alpha >0$. We have 
\begin{align*}
\int_{-\infty}^{\infty}\frac{ 1}{t^2+2  i \alpha t +1} \, \mathrm{d}t=\int_{-\infty}^{\infty}\frac{ 1}{(t+i \alpha)^2+1 +\alpha^2} \, \mathrm{d}t =\int_{-\infty}^{\infty}\frac{ 1}{(t-t_1)(t-t_2)} \, \mathrm{d}t,
\end{align*}
with 
\begin{align*}
t_1 \vcentcolon= -i\alpha+i\sqrt{1 +\alpha^2}, \qquad\mathrm{Im}(t_1) >0, \\
t_2 \vcentcolon= -i\alpha-i\sqrt{1 +\alpha^2}, \qquad\mathrm{Im}(t_2) < 0.
\end{align*}
A standard application of residue theorem yields
\begin{align*}
\int_{-\infty}^{\infty}\frac{ 1}{t^2+2  i \alpha t +1} \, \mathrm{d}t=2i\pi\frac{1}{t_1-t_2}=\frac{\pi}{\sqrt{1+\alpha^2}}.
\end{align*}
Coming back to the original variable $z$, we therefore get
\begin{align*}
\int_{-\infty}^{\infty}\frac{ 1}{t^2+2  zt +1} \, \mathrm{d}t=\frac{\pi}{\sqrt{1-z^2}},
\end{align*}
which ends the proof.
\end{proof}

\section{Proof of the bound \eqref{eq:Bound-DCT}}\label{Appendix-proofBOUND-DCT}
Let us recall that 
  \begin{align*}
    \frac{1}{D_\zeta(\lambda)}=\frac{\sqrt{1+\frac{1}{\lambda^2}}}{(1-2\pi \zeta)\sqrt{1+\frac{1}{\lambda^2}}+2\pi \zeta}, \qquad\lambda \in \C \setminus i[-1, 1].
\end{align*}
We will study $\tau \mapsto D_\zeta(i\tau +\eps)^{-1}$ for $\tau \geq 0$ near $\tau = 1$ and $\tau=0$ separately (the case of $\tau \leq 0$ and/or of $-\eps$ being symmetric). More precisely, we write
\begin{align*}
    [0, 1]=I_0 \cup I_{\mathrm{int}}\cup I_{\mathrm{edges}}  ,
\end{align*}
where
\begin{align*}
    I_0=[0, \delta_0], \qquad I_{\mathrm{int}}=[\delta_0, 1-\delta_1], \qquad I_{\mathrm{edges}}=[1-\delta_1, 1].  
\end{align*}
for two parameters $\delta_0, \delta_1>$ to be chosen later on. We will use the fact that 
\begin{align*}
   \left\vert 1+\frac{1}{(i\tau+\eps)^2}\right\vert
   =\left\vert \frac{(i\tau+\eps)^2+1}{(i\tau+\eps)^2} \right\vert
   =\frac{\sqrt{(\Delta(\tau)+\eps^2)^2+4\eps^2 \tau^2}}{\tau^2+\eps^2}, \qquad\Delta(\tau)=1-\tau^2.
\end{align*}

\medskip

\textbf{Edge part $I_{\mathrm{edges}}$}. If $\tau \in I_{\mathrm{edges}}$ then $\tau^2+\eps^2 \geq (1-\delta_1)^2>0$ and $\Delta(\tau) \leq 1-(1-\delta_1)^2$ therefore for $\eps<\eps_0$ and $\tau \in I_{\mathrm{edges}}$ we have 
\begin{align*}
    \left\vert 1+\frac{1}{(i\tau+\eps)^2} \right\vert \le  \frac{\sqrt{(1-(1-\delta_1)^2+\eps_0^2)^2+4\eps_0^2 1}}{(1-\delta_1)^2} .
\end{align*}
This last quantity can be made as small as we want by taking $\eps_0>0$ and $\delta_1>0$ small enough. Hence, the exact same fact holds for the quantity $\sqrt{1+\frac{1}{(i\tau+\eps)^2}}$ since $$\left\vert \sqrt{1+\frac{1}{(i\tau+\eps)^2}} \right\vert = \sqrt{\left\vert 1+\frac{1}{(i\tau+\eps)^2} \right\vert}.$$
We infer that for $\eps_0>0$ and $\delta_1>0$ small enough, we have for all $\eps<\eps_0$ and $\tau \in I_{\mathrm{edges}}$
\begin{align*}
    \left\vert \frac{1}{D_\zeta(i\tau+\eps)} \right\vert \leq \frac{\left\vert \sqrt{1+\frac{1}{(i\tau+\eps)^2}} \right\vert}{ 2\pi\zeta  -  (1-2\pi\zeta)\left\vert \sqrt{1+\frac{1}{(i\tau+\eps)^2}} \right\vert}  \leq C_{\zeta, \eps_0, \delta_1},
\end{align*}
for some constant $C_{\zeta, \eps_0, \delta_1}>0$.

\medskip

\textbf{Zero part $I_{0}$}.  We have
\begin{align*}
    \frac{1}{D_\zeta(i\tau+\eps)}=\frac{1}{(1-2\pi \zeta) + \frac{\zeta}{\sqrt{1+\frac{1}{(i\tau+\eps)^2}}}}.
\end{align*}
If $\tau \in I_0$, then $\Delta(\tau) \geq 1-\delta_0^2$ and $\tau^2+\eps^2 \leq \delta_0^2+\eps_0^2$ so for $\eps<\eps_0$ and $\tau \in I_{0}$
\begin{align*}
    \left\vert 1+\frac{1}{(i\tau+\eps)^2} \right\vert \geq \frac{1-\delta_0^2}{\delta_0^2+\eps_0^2},
\end{align*}
which can be made large enough by taking $\eps_0>0$ and $\delta_0>0$ small enough. Hence, we deduce from
\begin{align*}
    \left\vert \frac{1}{D_\zeta(i\tau+\eps)} \right\vert \leq \frac{1}{(1-2\pi\zeta) -\frac{\zeta}{\left\vert \sqrt{1+\frac{1}{(i\tau+\eps)^2}} \right\vert }},
\end{align*}
that for $\eps_0>0$ and $\delta_0>0$ small enough, we have for all $\eps<\eps_0$ and $\tau \in I_{0}$
\begin{align*}
    \left\vert \frac{1}{D_\zeta(i\tau+\eps)} \right\vert   \leq C'_{\zeta, \eps_0, \delta_1}.
\end{align*}
for some constant $C'_{\zeta, \eps_0, \delta_1}>0$.

\medskip

\textbf{Intermediate part $I_{\mathrm{int}}$}. We now work with $\tau \in I_{\mathrm{int}}=[\delta_0, 1-\delta_1]$, having fixed the parameters $\delta_0, \delta_1>0$ from before, reducing $\delta_0$ if necessary so that $\delta_0^2 \leq  1$. Note that we can still reduce $\eps_0>0$ if necessary. We write for $\eps_0$ small enough and $0<\eps<\eps_0$
\begin{align*}
    \vert i \tau + \eps \vert \geq \tau-\eps \geq \frac{\delta_0}{2},
\end{align*}
hence for all $0<\eps<\eps_0$ and $\tau \in I_{\mathrm{int}}$
\begin{align*}
     \left\vert 1+\frac{1}{ (i\tau+\eps)^2} \right\vert \leq 1+4\frac{1}{\delta_0^2}=M_{\delta_0}.
\end{align*}
Furthermore, $\lambda \mapsto 1+\frac{1}{\lambda^2} $ vanishes only at $\lambda= \pm i$. By continuity on this function on the compact set $\lbrace \lambda=i\tau+\eps \mid \tau \in I_{\mathrm{int}}, 0 \leq \eps \leq \eps_0 \rbrace$ (which does not contain these zeros), it has a minimum $m_{\eps_0, \delta_0, \delta_1}>0$ and therefore for all $0<\eps<\eps_0$ and $\tau \in I_{\mathrm{int}}$ we have
\begin{align*}
     m_{\eps_0, \delta_0, \delta_1} \leq \left\vert 1+\frac{1}{ (i\tau+\eps)^2} \right\vert \leq M_{\delta_0}.
\end{align*}
We deduce that for all $0<\eps<\eps_0$ and $\tau \in I_{\mathrm{int}}$, there holds
\begin{align*}
     m_{\eps_0, \delta_0, \delta_1}^{1/2} \leq \left\vert \sqrt{1+\frac{1}{(i\tau+\eps)^2}} \right\vert \leq M_{\delta_0}^{1/2}.
\end{align*}
Similarly, we would obtain a uniform lower bound for the denominator $\tau \mapsto (1-2\pi\zeta)\sqrt{1+\frac{1}{(i\tau+\eps)^2}} + 2 \pi \zeta$ if we know that it does not vanish for $\tau \in I_{\mathrm{int}}$: as a matter of fact, we know from \eqref{pikachu} that
\begin{align*}
 \lim_{\eps \rightarrow 0^+}   \left((1-2\pi\zeta)\sqrt{1 + \frac{1}{(i\tau + \eps)^2}} + 2\pi\zeta \right)= -(1-2\pi\zeta)i  \sqrt{\frac{1}{\tau^2} - 1} + 2\pi\zeta \neq 0,
\end{align*}
because it has a non-zero imaginary part. Since the denominator of $1/D_\zeta(\lambda)$ is holomorphic and non-vanishing on $\lbrace \lambda=i\tau+\eps \mid \tau \in I_{\mathrm{int}}, 0< \eps \leq \eps_0 \rbrace$, and continuous up to the cut (approaching from the right) with a non-zero value there, we deduce by compactness that its modulus has a positive lower bound independent of $\eps$. Hence, we infer that for all $\eps<\eps_0$ and $\tau \in I_{\mathrm{int}}$
\begin{align*}
    \left\vert \frac{1}{D_\zeta(i\tau+\eps)} \right\vert   \leq C''_{\zeta, \eps_0, \delta_1}.
\end{align*}
for some constant $C''_{\zeta, \eps_0, \delta_1}>0$.

\medskip

We can now conclude that the bound \eqref{eq:Bound-DCT} holds by taking
\begin{align*}
    C_{\eps_0}=\max\left(C_{\zeta, \eps_0, \delta_1},C'_{\zeta, \eps_0, \delta_1},C''_{\zeta, \eps_0, \delta_1} \right).
\end{align*}

\section{Toolkit on Volterra equations}\label{Appendix-Volterra}
For a matrix kernel $K:\R^+ \rightarrow M_n(\C)$ and a forcing $f: \R^+ \rightarrow \C^n$, we consider the 
general Volterra equation
\begin{align}\label{eq:Volterra-appendix}
    u(t)+K \star u(t) =f(t), \qquad t \geq 0,
\end{align}
with unknown $u : \R^+ \rightarrow \C^n$. Here, the convolution is defined by
\begin{align*}
    M \star v(t)\vcentcolon= \int_0^t M(t-s) v(s) \, \mathrm{d}s.
\end{align*}
Typically, all the previous operations make sense when $u,f \in L^1_{\mathrm{loc}}(\R^+, \C^n)$ and $K \in L^1_{\mathrm{loc}}(\R^+, M_n(\C))$. Note that the convolution product does not commute in general if $n \neq 1$.

As explained in \cite{gripenberg1990volterra}, a classical way to solve \eqref{eq:Volterra-appendix} is to use the so called resolvent kernel $R :\R^+ \rightarrow M_n(\C)$ associated with $K$, and defined as a solution of 
\begin{align}\label{eq:resolvent-appendix}
    R(t)+K \star R(t)=R(t)+ R \star K(t) =K(t), \qquad t \geq 0.
\end{align}
 In that case, the variation of constant formula formally yields a solution $u(t)$ of $\eqref{eq:Volterra-appendix}$ given by $u=f-R \star f$. More precisely, we have the following theorem.
\begin{Thm}\label{Thm-solvVolterraLocally}
    Let $K \in L^1_{\mathrm{loc}}(\R^+, M_n(\C))$. Then the following holds.
    \begin{enumerate}
        \item There exists a unique resolvent $R \in L^1_{\mathrm{loc}}(\R^+, M_n(\C))$ associated to the kernel $K$, that is solving \eqref{eq:resolvent-appendix}. 
        \item For all $f \in L^1_{\mathrm{loc}}(\R^+, \C^n)$, there exists a unique solution $u \in L^1_{\mathrm{loc}}(\R^+, \C^n)$ to the Volterra   equation \eqref{eq:Volterra-appendix}. Moreover, it is given by
        \begin{align*}
            u(t)=f(t)-R \star f(t), \qquad t \geq 0.
        \end{align*}
    \end{enumerate}
\end{Thm}
A more subtle question concerns the global in time solvability of \eqref{eq:Volterra-appendix}, together with the decay in time of the solution $u$ under some decay assumptions on the kernel $K$ and forcing $f$.

Let us recall the definition of the Laplace transform: for a function $u: \R^+ \rightarrow \C$ such that $t \mapsto u(t) \e^{-c_0 t}$ is in $L^1(\R^+)$ for some $c_0 \in \R$, we define
\begin{align*}
    \mathcal{L}[u](\lambda)=\int_0^{+\infty} \e^{-\lambda t} u(t) \, \mathrm{d}t, \qquad \lambda \in \lbrace \mathrm{Re} \geq c_0 \rbrace.
\end{align*}
For a vector valued or matrix valued function, its Laplace transform is defined coefficient by coefficient. 

\medskip

Regarding the solvability of the Volterra equation \eqref{eq:Volterra-appendix} for large times, we have the following theorem (see again \cite{gripenberg1990volterra}).
\begin{Thm}[Paley-Wiener]\label{thm-Paley-Wiener}
 Let $K \in L^1(\R^+, M_n(\C))$ and let us assume that the following spectral condition
 \begin{align}\label{eq:spectral-condTHM}
     \forall \lambda \in \lbrace \mathrm{Re} \geq 0 \rbrace, \qquad \mathrm{det}\left( \mathrm{I}+ \mathcal{L}[K](\lambda) \right) \neq 0. 
 \end{align}
 holds. Then, there exists a unique resolvent kernel $R \in L^1(\R^+, M_n(\C))$. Moreover, for any $f \in L^1(\R^+, \C^n)$, the Volterra equation \eqref{eq:Volterra-appendix} has a unique solution $u  \in L^1(\R^+, \C^n)$ given by 
 \begin{align*}
            u(t)=f(t)-R \star f(t), \qquad t \geq 0.
        \end{align*}
 \end{Thm}

Concerning quantitative exponential decay estimates, we also have the following statement. Roughly speaking, it says that under some exponential decay of the source and the kernel, and \textit{modulo} a stronger stability condition than \eqref{eq:spectral-condTHM} (holding in an enlarged half-plane), the solution to the Volterra equation \eqref{eq:Volterra-appendix} also enjoys some exponential decay. In that case, the decay rate is less than those of the kernel/source, and of the width of the stable part of the strip. For our purpose, it will be enough to provide some integrability properties, without quantifying any pointwise decay.
\begin{Thm}\label{THM-expo-decayVolterra}. 
 Let $K \in L^1(\R^+, M_n(\C))$ and assume that the following spectral condition holds for some $\delta>0$:
 \begin{align}\label{eq:spectral-cond-enlarged}
     \forall \lambda \in \lbrace \mathrm{Re} \geq -\delta \rbrace, \qquad \mathrm{det}\left( \mathrm{I}+ \mathcal{L}[K](\lambda) \right) \neq 0. 
 \end{align}
 If there exist $C, \gamma_K>0$ such that
 \begin{align}\label{decay:kernelK-Volterra}
     \vert K(t) \vert \leq C \e^{-\gamma_K t}, \qquad t \geq 0,
 \end{align}
then for any $f \in  L^1(\R^+, \C^n)$, the unique solution $u=u(t)$ to the Volterra equation \eqref{eq:Volterra-appendix} satisfies the integrability condition 
\begin{align*}
    \e^{\gamma \bullet}u \in L^1(\R^+), \qquad 0 \leq \gamma < \min(\gamma_K, \delta).
\end{align*}

 \end{Thm}
 \begin{proof}
According to Theorem \ref{thm-Paley-Wiener}, since the condition \eqref{eq:spectral-cond-enlarged} implies \eqref{eq:spectral-condTHM}, the Volterra equation  \eqref{eq:Volterra-appendix} has a unique solution $u \in L^1(\R^+, \C^n)$. Setting $$(u_\gamma, f_\gamma)(t)\vcentcolon=\e^{\gamma t}(u,f)(t), \qquad \text{for} \qquad \gamma < \min(\gamma_K, \delta),$$ 
we have
\begin{align*}
    u_\gamma(t)+ \int_0^t K_\gamma(t-s) u_\gamma(s) \, \mathrm{d}s=f_\gamma(t), \qquad t \geq 0,
\end{align*}
where $K_\gamma(\tau)\vcentcolon=\e^{\gamma \tau }K(\tau)$. Since $\gamma<\gamma_K$, we observe that $K_\gamma \in L^1(\R^+, M_n(\C))$. The former equation being a new Volterra equation, a direct uniqueness argument shows that, if $R$ is the resolvent kernel associated to $K$, then $R_\gamma: t \mapsto \e^{\gamma t}  R(t)$ is the resolvent kernel associated to $K_\gamma$. By Theorem \ref{Thm-solvVolterraLocally}, we know that $u_\gamma=\e^{\gamma \bullet}u \in L^1(\R^+)$ so, if we show that $R_\gamma \in L^1(\R^+, \C^n)$, this would conclude the proof. By Theorem \ref{thm-Paley-Wiener}, it is therefore enough to check that 
\begin{align*}
     \forall \lambda \in \lbrace \mathrm{Re} \geq 0 \rbrace, \qquad \mathrm{det}\left( \mathrm{I}+ \mathcal{L}[K_\gamma](\lambda) \right) \neq 0. 
 \end{align*}
 By properties of the Laplace transform, we have $ \mathcal{L}[K_\gamma](\lambda)=\mathcal{L}[K](\lambda- \gamma )$ and we therefore have to ensure that  
\begin{align*}
     \forall \lambda \in \lbrace \mathrm{Re} \geq 0 \rbrace, \qquad \mathrm{det}\left( \mathrm{I}+ \mathcal{L}[K](\lambda- \gamma ) \right) \neq 0. 
 \end{align*}
 But the this shifted stability condition is exactly the assumed condition \eqref{eq:spectral-cond-enlarged}, and it therefore concludes the proof.
  \end{proof}

 \section*{Acknowledgements}
The research of MCZ was partially supported by the Royal Society URF\textbackslash R1\textbackslash 191492 and the ERC-EPSRC Horizon Europe Guarantee EP/X020886/1.
LE acknowledges partial financial support from the ERC-EPSRC through grant Horizon Europe Guarantee EP/X020886/1 while starting this project, as well as the hospitality of the Mathematics Department at Imperial College London where part of this work was undertaken. DGV acknowledges the support of the ANR Project Bourgeons, grant ANR- 23-CE40-0014-01 and of the ANR-DFG Project Suspensions, grant ANR-24-CE92-0028. He also benefited from the Project Complexflows  (ANR-23-EXMA-0004) and from the InidEx Complexcit\'e at Universit\'e Paris Cit\'e.


\bibliography{biblio}
 \bibliographystyle{abbrv}

\end{document}